\documentclass[11pt]{article}

\usepackage[letterpaper, margin=1in]{geometry}

\usepackage{amscd,amsfonts,amsmath,amssymb,amsthm,bbm,enumitem,etoolbox,fancyhdr,graphicx,hhline,ltablex,multicol,pinlabel,pdflscape,ragged2e,tabularx,thmtools,thm-restate,tikz,tikz-cd,titling,tocloft}
\usepackage{xcolor}
\usepackage[colorlinks,
    linkcolor={violet},
    citecolor={olive!80!orange},
    urlcolor={blue!80!black}]{hyperref}
\usepackage[alphabetic]{amsrefs}

\numberwithin{figure}{section}

\theoremstyle{plain}
\newtheorem{theorem}{Theorem}[section]
\newtheorem{corollary}[theorem]{Corollary}
\newtheorem{lemma}[theorem]{Lemma}
\newtheorem{proposition}[theorem]{Proposition}

\newtheorem*{theorem*}{Theorem}
\newtheorem*{milnortheorem}{Milnor's Theorem}
\newtheorem*{sdtheorem}{Stallings-Dwyer Theorem}
\newtheorem*{acconj}{Almost-Concordance Conjecture}
\newtheorem{innercustomthm}{Theorem}
\newenvironment{customthm}[1]
  {\renewcommand\theinnercustomthm{#1}\innercustomthm}
  {\endinnercustomthm}
\newtheorem{innercustomprop}{Proposition}
\newenvironment{customprop}[1]
  {\renewcommand\theinnercustomprop{#1}\innercustomprop}
  {\endinnercustomprop}
\newtheorem{innercustomcor}{Corollary}
\newenvironment{customcor}[1]
  {\renewcommand\theinnercustomcor{#1}\innercustomcor}
  {\endinnercustomcor}

\theoremstyle{definition}
\newtheorem{definition}[theorem]{Definition}
\newtheorem*{definition*}{Definition}
\newtheorem{question}[theorem]{Question}
\newtheorem{goal}[theorem]{Goal}

\theoremstyle{remark}

\usepackage{subfiles}

\pretitle{\begin{flushleft}\Large\rmfamily}
\posttitle{\end{flushleft}\vspace{0.25em}}

\preauthor{\begin{flushleft}\large}
\postauthor{\end{flushleft}}

\title{\textbf{Milnor's invariants for knots and links in closed orientable 3-manifolds}}
\author{Ryan Stees}
\date{}

\begin{document}

\maketitle

\vspace{-2.5em}

\noindent {\small \textbf{Abstract.}
In his 1957 paper, John Milnor introduced a collection of invariants for links in $S^3$ detecting higher-order linking phenomena by studying lower central quotients of link groups and comparing them to those of the unlink. These invariants, now known as Milnor's $\overline{\mu}$-invariants, were later shown to be topological link concordance invariants and have since inspired decades of consequential research. Milnor's invariants have many interpretations, and there have been numerous attempts to extend them to other settings. In this paper, we extend Milnor's invariants to topological concordance invariants of knots and links in general closed orientable 3-manifolds. These invariants unify and generalize all previous versions of Milnor's invariants in dimension 3, including Milnor's original invariants for links in $S^3$.}

{
\renewcommand{\cfttoctitlefont}{\large\bfseries}
\renewcommand{\cftsecfont}{\normalsize}
\renewcommand{\cftsecpagefont}{\normalsize}
\renewcommand{\cftsecaftersnum}{.}
\setlength{\cftbeforesecskip}{0em}

\tableofcontents
}


\section{Introduction}\label{intro}

In his 1954 Ph.D. thesis \cite{Milnor}, John Milnor introduced a collection of integer-valued invariants for links in $S^3$, now known as \emph{Milnor's $\overline{\mu}$-invariants}, which inductively compare lower central quotients of link groups to those of the unlink. The simplest $\overline{\mu}$-invariants are the classical linking numbers first defined by C.F. Gauss \cite{Gauss}, and the higher invariants are higher-order linking numbers. Much is known about Milnor's $\overline{\mu}$-invariants; in particular, Casson used work of Stallings \cite{Stallings} to show that Milnor's invariants are topological link concordance invariants \cite{Casson}. Milnor's invariants are ubiquitous, having interpretations in terms of Massey products \cites{Turaev76, Porter}, homotopy-theoretic data \cite{Orr89}, surface systems in $S^3$ \cite{Cochran90}, the Kontsevich integral \cite{HabeggerMasbaum}, and twisted Whitney towers in the 4-ball \cite{CST14}. We summarize the characterizing property of these $\overline{\mu}$-invariants in the following theorem:

\begin{milnortheorem}[\cite{Milnor} 1954]\label{THMmilnor}
Let $L\subset S^3$ be an $m$-component link, and let $U$ be the $m$-component unlink, with link groups $\pi$ and $F$, respectively. Suppose the $n^{\text{th}}$ lower central quotients of $\pi$ and $F$ are isomorphic. Then the following three statements are equivalent:
\begin{enumerate}[label=(\arabic*)]
\item The $\overline{\mu}$-invariants of $L$ of length $n$ agree with those of $U$ (that is, they vanish).
\item The $(n+1)^{\text{st}}$ lower central quotients of $\pi$ and $F$ are isomorphic.
\item The $\overline{\mu}$-invariants of $L$ of length $n+1$ are well-defined.
\end{enumerate}
\end{milnortheorem}

There have been numerous attempts to extend Milnor's $\overline{\mu}$-invariants to other settings; see \cites{MillerD,Heck,Kuzbary,ChaOrr}. This paper extends Milnor's invariants to concordance invariants of knots and links in any closed orientable 3-manifold. Our work unifies and generalizes all previous extensions of Milnor's invariants in dimension 3.


\subsection{New contributions}\label{intro-new}

To any link $L$ in a closed orientable 3-manifold $M$, we associate a tower of spaces $\{X_n(L)\}_{n\geq 1}$, well-defined up to homotopy equivalence, which come equipped with canonical homotopy classes of maps $\iota_n(L):M\to X_n(L)$. We have the following diagram:
\begin{center}
\begin{tikzcd}
& & M \arrow[dl,"\iota_{n+1}(L)"'] \arrow[d,"\iota_n(L)"] \arrow[drr,"\iota_1(L)"] & & &\\
\cdots \arrow[r] & X_{n+1}(L) \arrow[r] & X_n(L) \arrow[r] & \cdots \arrow[r] & X_1(L).
\end{tikzcd}
\end{center}
We will show that a concordance between two links $L, L'\subset M$ induces homotopy equivalences $X_n(L')\xrightarrow{\simeq}X_n(L)$ for all $n$ and that the maps $\iota_n(L'):M\to X_n(L')$ and $\iota_n(L):M\to X_n(L)$ correspond via these homotopy equivalences. 

To define invariants, we will fix a link $L\subset M$ and compare other links to $L$ inductively. Given some link $L'\subset M$, and assuming we have homotopy equivalences $X_k(L')\xrightarrow{\simeq}X_k(L)$ for all $k\leq n$, we obstruct the existence of a homotopy equivalence $X_{n+1}(L')\xrightarrow{\simeq}X_{n+1}(L)$ by comparing the maps $\iota_n(L'):M\to X_n(L')$ and $\iota_n(L):M\to X_n(L)$ via the homotopy equivalence $X_n(L')\xrightarrow{\simeq}X_n(L)$. Through this inductive comparison, we extract a sequence of topological concordance invariants $h_n$ which we call the \emph{lower central homotopy invariants} of $L'$ relative to $L$. These invariants are homotopy classes of maps in $[M,X_n(L)]_0/\text{Aut}(\pi/\Gamma_n,\partial)$, where $\text{Aut}(\pi/\Gamma_n,\partial)$ is a group acting on $[M,X_n(L)]_0$ we discuss in Section~\ref{nbasing-selfbasing}.

An approach K. Orr pioneered in \cite{Orr89} inspired our approach to defining the lower central homotopy invariants; we define them in a similar manner to invariants of P. Heck \cite{Heck}. We paraphrase the main results involving these invariants:
\begin{itemize}
\item \emph{Invariance:} Theorem~\ref{THMinvariance}, stated in Section~\ref{results-homotopy} and proved in Section~\ref{proofs-inv-A}, asserts that the $h_n$ are topological concordance invariants.
\item \emph{Characterization:} Theorem~\ref{THMcharacterization}, stated in Section~\ref{results-ncob} and proved in Section~\ref{htpychar-thmB}, states that the $h_n$ characterize $n$-cobordism, a geometric relationship between links. Roughly speaking, $L'$ and $L''$ are \emph{$n$-cobordant} if they cobound a surface in $M\times[0,1]$ which looks like a concordance to the $n^{\text{th}}$ lower central quotients. Theorem~\ref{THMcharacterization} states that $h_n(L')=h_n(L'')$ if and only if $L'$ is $n$-cobordant to $L''$.
\item \emph{Realization:} Theorem~\ref{THMrealization}, stated in Section~\ref{results-htpyreal} and proved in Section~\ref{htpychar-thmC}, gives exact conditions under which a homotopy class in $[M,X_n(L)]_0$ is realized as the invariant of some link $L'\subset M$.
\end{itemize}

The invariants $h_n$ also determine two additional new invariants we will discuss and explore:
\begin{itemize}
\item The invariants $\theta_n$, which we call the \emph{lower central homology invariants}, are homology classes in $H_3(X_n(L))/\text{Aut}(\pi/\Gamma_n,\partial)$. They are images of the fundamental class $[M]$ under the maps the $h_n$ induce on homology.
\item The invariants $\overline{\mu}_n$, which we refer to as \emph{Milnor's invariants}, are a further reduction of the $h_n$. Like Milnor's $\overline{\mu}$-invariants, they determine the lower central quotients of link groups one step at a time.
\end{itemize}

The invariants $\theta_n$ and $\overline{\mu}_n$ are useful in the greatly extended setting of $\mathbb{Z}[G]$-homology concordance, for suitable groups $G$, where we allow the links we compare to live in different 3-manifolds. A \emph{$\mathbb{Z}[G]$-homology concordance} between $L\subset M$ and $L'\subset M'$ is a concordance in a $\mathbb{Z}[G]$-homology cobordism between $M$ and $M'$. We now summarize the main results involving these reductions of $h_n$:
\begin{itemize}
\item \emph{Invariance:} Theorem~\ref{THMhinvariance}, stated in Section~\ref{results-homology} and proved in Section~\ref{proofs-inv-A'}, asserts that the $\theta_n$ and $\overline{\mu}_n$ are topological homology concordance invariants.
\item \emph{Characterization:} Theorem~\ref{THMhcharacterization}, stated in Section~\ref{results-determination} and proved in Section~\ref{homchar-thmD}, states that the $\overline{\mu}_n$ characterize algebraic information we call an \emph{$n$-basing}. Roughly speaking, an $n$-basing for $L'$ relative to $L$ is an isomorphism on $n^{\text{th}}$ lower central quotients of link groups which preserves the peripheral structures of meridians and longitudes. We state a concise version of Theorem~\ref{THMhcharacterization} below for the sake of direct comparison with \hyperref[THMmilnor]{Milnor's Theorem}. See Section~\ref{results-determination} for the definition of the $\overline{\mu}_n$ and the full statement of Theorem~\ref{THMhcharacterization}.
\item \emph{Lifting property:} Corollary~\ref{CORlift}, stated and proved in Section~\ref{results-determination}, asserts that when comparing links in the same 3-manifold the invariant $\overline{\mu}_n$ vanishes if and only if the lower central homotopy invariant $h_n$ lifts to a realizable homotopy class in $[M,X_{n+1}(L)]_0$.
\item \emph{Realization:} Theorem~\ref{THMhrealization}, stated in Section~\ref{results-homreal} and proved in Section~\ref{homchar-thmE}, gives exact conditions under which a homology class $\theta\in H_3(X_n(L))$ is realized as the integral homology concordance invariant $\theta=\theta_n(L')$ of some link $L'$ in some 3-manifold $M'$.
\end{itemize}

\begin{customthm}{\ref{THMhcharacterization}}
Fix an $m$-component link $L\subset M$. Let $L'\subset M$ be another $m$-component link, and let $\pi$ and $\pi'$ be the link groups of $L$ and $L'$, respectively. Suppose $L'$ admits an $n$-basing relative to $L$, that is, suppose the $n^{\text{th}}$ $\pi_1(M)$-lower central quotients of $\pi$ and $\pi'$ are admissibly isomorphic. Then the following three statements are equivalent:
\begin{enumerate}[label=(\arabic*)]
\item We have $\overline{\mu}_n(L')=\overline{\mu}_n(L)$.
\item The link $L'$ admits an $(n+1)$-basing relative to $L$, that is, the $(n+1)^{\text{st}}$ $\pi_1(M)$-lower central quotients of $\pi$ and $\pi'$ are admissibly isomorphic.
\item The invariant $\overline{\mu}_{n+1}(L')$ is well-defined.
\end{enumerate}
\end{customthm}

We emphasize a few features of the invariants $h_n$ and their reductions $\theta_n$ and $\overline{\mu}_n$ which relate them to previous extensions of Milnor's $\overline{\mu}$-invariants:

\begin{itemize}
\item The vanishing of the invariants $h_n$, taken in appropriate contexts, implies the vanishing of all previous versions of Milnor's invariants in dimension 3, including Milnor's original $\overline{\mu}$-invariants.
\item For links in $S^3$ and every $n\geq 2$, the invariant $\overline{\mu}_n$ relative to the unlink is equivalent to Milnor's (total) $\overline{\mu}$-invariant of length $n+1$.
\item For links in $S^3$ and $n\geq 2$, the invariants $h_n$ relative to the unlink are equivalent to invariants of Orr from \cite{Orr89}.
\item The invariants $h_n$ are defined for knots and can be nontrivial for knots in 3-manifolds $M\neq S^3$; see \cites{MillerD, Heck, Kuzbary}.
\item For empty links, the invariants $\theta_n$ and $\overline{\mu}_n$ are precisely the Cha-Orr invariants of homology cobordism of 3-manifolds, ``Milnor's invariants of 3-manifolds", from \cite{ChaOrr}. Cha-Orr show these invariants exhibit highly nontrivial behavior.
\end{itemize}


\subsection{Notation}\label{intro-notation}

Throughout this paper, $L\subset M$, $L'\subset M'$, and $L''\subset M''$ will denote $m$-component links in the closed orientable 3-manifolds $M$, $M'$, and $M''$. All links are assumed to be oriented and ordered. We will usually assume the 3-manifolds are oriented, so that meridians of links and fundamental classes of the 3-manifolds are well-defined. The 3-manifolds $M$, $M'$, and $M''$ will be equipped with surjective homomorphisms $\varphi$, $\varphi'$, and $\varphi''$ from their respective fundamental groups onto a fixed group $G$. Often, we will have $M=M'=M''$, $G=\pi_1(M)$, and $\varphi=\varphi'=\varphi''=\text{id}_{\pi_1(M)}$, and we will indicate when this is not the case. We will denote the regular neighborhood of a submanifold $X$ in a manifold $Y$ by $\nu X$ and the exterior of $X$ in $Y$ by $E_X$. The link group of $L\subset M$, the fundamental group $\pi_1(E_L)$, will be denoted by $\pi$. We will let $\Gamma\trianglelefteq\pi$ be the normal subgroup $\Gamma=\ker(\pi\twoheadrightarrow\pi_1(M)\twoheadrightarrow G)$. We will use analogous notation for the links $L'$ and $L''$.


\subsection{Acknowledgements}\label{intro-acknowledgements}

The author is grateful to Kent Orr for his unwavering support of this project, his many insightful comments, and his guidance. The author also thanks Jae Choon Cha, whose work with Kent Orr in \cite{ChaOrr} inspired several ideas in this paper. Additionally, the author thanks Anthony Conway, Matthew Hedden, Alexandra Kjuchukova, and Mark Powell for valuable conversations. Finally, the author is grateful to the referee for their careful reading of this work and for many helpful suggestions.


\section{Preliminaries}\label{prelims}


\subsection{Concordance in 3-manifolds}\label{prelims-conc}

A non-simply-connected 3-manifold $M$ allows for embedding phenomena which cannot occur in $S^3$. In such 3-manifolds, we do not have an analogue of the unknot or unlink in the sense that there is no natural choice of knot or link to which all others can be reasonably compared. For instance, knots which are concordant must be freely homotopic; in other words, concordant knots represent the same element of $[S^1,M]$. A free homotopy between concordant knots is given by the composition of the concordance with projection to $M$: \[S^1\times[0,1]\hookrightarrow M\times[0,1]\xrightarrow{\text{proj}} M.\] To develop concordance invariants, we will fix some link $L\subset M$ and compare other links relative to $L$. The fixed link $L$ will play the role of the unlink. 

The nontriviality of $\pi_1(M)$ also implies that knots $K\subset M$ may not all have preferred longitudes. A \emph{longitude} for a knot $K$ is a section of the sphere bundle of $K$, that is, a parallel of $K$ in $\partial\nu K$. We also refer to the homology class of such a curve as a longitude. Removing a regular neighborhood of a concordance between two knots $K,K'\subset M$ yields a cobordism rel boundary between knot exteriors $E_K$ and $E_{K'}$ which identifies $\partial \nu K$ and $\partial \nu K'$ and induces a correspondence between the meridians and longitudes of $K$ and $K'$. If $K$ is nullhomologous in $M$, then one can pick a preferred longitude by taking the pushoff of $K$ which is nullhomologous in the exterior $E_K$. A concordance induces a correspondence between preferred longitudes if they exist; however, in the absence of preferred longitudes, there is \emph{a priori} some ambiguity in how $\partial \nu K'$ and $\partial \nu K$ are to be identified when attempting to obstruct the existence of a concordance between $K'$ and $K$. In other words, the absence of a preferred longitude prevents a canonical identification of $\partial\nu K$ with $S^1\times S^1$. For knots representing torsion classes in $H_1(M)$, a rational-valued linking number remedies this problem (see \cite{FNOP} or \cite{Raoux}), but it is not clear what to do in the case where $[K]\in H_1(M)$ has infinite order. The techniques of Section~\ref{results} resolve the afformentioned issue. To define the lower central homotopy invariants, we fix an identification of longitudes and then understand how to account for the dependence on this chosen identification.

In the last few decades, there has been renewed interest in extending the study of concordance to general 3-manifolds \cites{MillerD, Schneiderman, Heck, DNPR, Kuzbary}, as well as to generalized versions of concordance. These include homology concordance \cites{Goldsmith, LevineA, Zhou, HLL} and almost-concordance \cites{Celoria, Yildiz, FNOP, NOPP}, concordance up to local knotting. Relatively little is known in these settings. For instance, the following conjecture of Celoria \cite{Celoria} and Friedl-Nagel-Orson-Powell \cite{FNOP} remains open in many cases:

\begin{acconj}[Celoria, Friedl-Nagel-Orson-Powell]\label{CONJ}
Every free homotopy class $x\in[S^1,M]$ in a closed orientable 3-manifold $M\neq S^3$ which does not admit a dual 2-sphere contains infinitely many concordance classes of knots modulo local knotting.
\end{acconj}

\noindent M. Powell has suggested the work in this paper could be useful for attacking many of the remaining cases; indeed, in forthcoming work we consider the \hyperref[CONJ]{Almost-Concordance Conjecture} in a large class of open cases, namely, we consider nontrivial $x\in[S^1,M]$ in aspherical $M$.


\subsection{Lower central series}\label{prelims-lcs}

Recall that the \emph{(classical) lower central series} of a group $\pi$ is the series of normal subgroups \[\cdots\trianglelefteq\pi_{n+1}\trianglelefteq\pi_n\trianglelefteq\cdots\trianglelefteq\pi_2\trianglelefteq\pi_1=\pi,\] where $\pi_2=[\pi,\pi]$, the commutator subgroup of $\pi$, and $\pi_{n+1}=[\pi,\pi_n]$ in general. We call the quotients $\pi/\pi_n$ the \emph{lower central quotients} of $\pi$. 

We recall a fundamental result of Stallings and Dwyer concerning lower central quotients and group homology which will play an important role in proving some of the main results of this work. For an abbreviated proof, see \cite{ChaOrr}.

\begin{sdtheorem}[\cite{Stallings} 1965, \cite{Dwyer} 1975] \label{THMsd}
Suppose $f:A\to B$ is a group homomorphism inducing an isomorphism $H_1(A)\xrightarrow{\cong}H_1(B)$. For any positive integer $n$, $f$ induces an isomorphism $f_n:A/A_n\xrightarrow{\cong} B/B_n$ if and only if $f$ induces a surjection $H_2(A)\twoheadrightarrow H_2(B)/K_j(B)$ for all $j<n$, where $K_j(B)=\ker\big(H_2(B)\to H_2(B/B_j)\big)$.
\end{sdtheorem}

Now fix a group $G$, and suppose a group $\pi$ admits a surjective homomorphism $\varphi:\pi\twoheadrightarrow G$. We consider a finer lower central series, the \emph{$G$-lower central series}, defined as \[\cdots\trianglelefteq\Gamma_{n+1}\trianglelefteq\Gamma_n\trianglelefteq\cdots\trianglelefteq\Gamma_2\trianglelefteq\Gamma\trianglelefteq\pi,\] where $\Gamma=\ker\varphi$. This lower central series was also used in work of P. Heck \cite{Heck}. We recover the classical lower central series when $G=\{1\}$. The quotients $\pi/\Gamma_n$ are the \emph{$G$-lower central quotients} of $\pi$. The $G$-lower central quotients of the link group $\pi$ of a link $L\subset M$, for suitable groups $G$, will play a central role in our study of concordance. Most often, we will take $G=\pi_1(M)$ and $\varphi$ the surjective homomorphism $\pi\twoheadrightarrow\pi_1(M)$ induced by the inclusion $j:E_L\hookrightarrow M$.


\subsection{Homology properties of concordance}\label{prelims-homology}

Cappell and Shaneson observed that a concordance exterior is a $\mathbb{Z}[\pi_1(M)]$-homology cobordism rel boundary \cite{CappellShaneson}. In particular, if $L$ and $L'$ are concordant, we have an isomorphism \[H_*(E_{L'};\mathbb{Z}[\pi_1(M)])\xrightarrow{\cong}H_*(E_L;\mathbb{Z}[\pi_1(M)]),\] along with an identification of $\partial E_{L'}$ and $\partial E_L$ preserving meridians, orientations, and orderings of components. We will call such an identification \emph{admissible}. 

Apart from the case of knots in $S^3$, the $\mathbb{Z}[\pi_1(M)]$-homology of $E_L$, along with boundary data, can provide nontrivial information about the concordance class of $L$. (On the other hand, $E_K$ is an integral homology $S^1$ for any knot $K\subset S^3$, where the meridian and longitude of $K$ represent 1 and 0, respectively, in $H_1(S^1)\cong\mathbb{Z}$.) Using the \hyperref[THMsd]{Stallings-Dwyer Theorem}, we see that the $\pi_1(M)$-lower central quotients of link groups are concordance invariants.

\begin{corollary}\label{prelims-corlcs}
A concordance $C$ between $L$ and $L'$ induces isomorphisms over $\pi_1(M)$ on $\pi_1(M)$-lower central quotients $\pi'/\Gamma'_n\xrightarrow{\cong}\pi/\Gamma_n$ for all $n$.
\end{corollary}

\begin{proof}
Let $\overline{E_{L'}}$, $\overline{E_C}$, and $\overline{E_L}$ denote the $\pi_1(M)$-covers of $E_{L'}$, $E_C$, and $E_L$, respectively. Consider the inclusion-induced isomorphisms \[H_*(E_{L'};\mathbb{Z}[\pi_1(M)])\xrightarrow{\cong}H_*(E_C;\mathbb{Z}[\pi_1(M)])\xleftarrow{\cong}H_*(E_L;\mathbb{Z}[\pi_1(M)])\] which yield isomorphisms $H_*(\overline{E_{L'}};\mathbb{Z})\xrightarrow{\cong}H_*(\overline{E_C};\mathbb{Z})\xleftarrow{\cong}H_*(\overline{E_L};\mathbb{Z})$. Let $\pi^C=\pi_1(E_C)$ and $\Gamma^C=\ker\big(\pi^C\twoheadrightarrow\pi_1(M\times[0,1])\big)$. Then $\overline{E_{L'}}$, $\overline{E_C}$, and $\overline{E_L}$ have fundamental groups $\Gamma'$, $\Gamma^C$, and $\Gamma$, respectively. By the \hyperref[THMsd]{Stallings-Dwyer Theorem}, we have isomorphisms on lower central quotients $\Gamma'/\Gamma'_n\xrightarrow{\cong}\Gamma^C/\Gamma^C_n\xleftarrow{\cong}\Gamma/\Gamma_n$ for all $n$. The result then follows from the Five Lemma.
\end{proof}

Corollary~\ref{prelims-corlcs} motivates the definitions of the lower central homotopy invariants $h_n$ given in Section~\ref{results-homotopy}. The $h_n$ will obstruct the existence of certain isomorphisms between $\pi_1(M)$-lower central quotients of link groups. As in Milnor's original work, they will compare the $\pi_1(M)$-lower central quotients of link groups to those of a fixed link $L\subset M$.


\section{Main results}\label{results}

In this section, we present our extension of Milnor's invariants to links in closed orientable 3-manifolds. For the duration of Sections~\ref{results-homotopy}-\ref{results-htpyreal}, all links will live in a fixed 3-manifold $M$, and for a link group $\pi$ we will take $\Gamma=\ker\big(\pi\twoheadrightarrow\pi_1(M)\big)$, so that the groups $\pi/\Gamma_n$ are the $\pi_1(M)$-lower central quotients of $\pi$. We will indicate when these assumptions change.


\subsection{Lower central homotopy invariants}\label{results-homotopy}

Recall from Section~\ref{intro-new} that we will associate to a link $L\subset M$ a tower of spaces $\{X_n(L)\}_{n\geq 1}$, well-defined up to homotopy equivalence, with canonical homotopy classes of maps $\iota_n(L):M\to X_n(L)$ as in the diagram below. 
\begin{center}
\begin{tikzcd}
& & M \arrow[dl,"\iota_{n+1}(L)"'] \arrow[d,"\iota_n(L)"] \arrow[drr,"\iota_1(L)"] & & &\\
\cdots \arrow[r] & X_{n+1}(L) \arrow[r] & X_n(L) \arrow[r] & \cdots \arrow[r] & X_1(L)
\end{tikzcd}
\end{center}
To define the concordance invariants $h_n$, we fix a link $L\subset M$ and compare other links to $L$. For another link $L'\subset M$, we inductively define homotopy classes of maps, depending on $L'$, from $M$ into the tower $\{X_n(L)\}$. The $n^{\text{th}}$ invariant $h_n$ of $L'$ relative to $L$ is a homotopy class of maps $h_n(L',\phi)\in[M,X_n(L)]_0$ into the $n^{\text{th}}$ level of this tower. If $L'$ is concordant to $L$, we can define such a homotopy class for every $n$.

Consider the diagram $\mathcal{D}_n(L)$ below, where $j:E_L\hookrightarrow M$ is the inclusion and $p_n$ is induced by the canonical projection $\pi\twoheadrightarrow\pi/\Gamma_n$.

\begin{center}
\begin{tikzcd}
 & & E_L \arrow[rr, "p_n"] \arrow[dd, hookrightarrow, "j"'] & & K(\pi/\Gamma_n,1) \\
\mathcal{D}_n(L) & = & & & \\
 & & M & &
\end{tikzcd}
\end{center}
Define $X_n(L)$ to be the standard homotopy pushout of $\mathcal{D}_n(L)$, that is, the based space \[ X_n(L)=\frac{K(\pi/\Gamma_n,1)\sqcup (E_L\times[0,1])^\times\sqcup M}{p_n(x)\sim (x,0), j(x)\sim (x,1)}, \] where $(E_L\times[0,1])^\times$ is the reduced cylinder on $E_L$. The standard homotopy pushout square equips the space $X_n(L)$ with a canonical based homotopy class of maps $\iota_n(L)\in[M,X_n(L)]_0$ represented by the inclusion $M\hookrightarrow X_n(L)$. This inclusion is a canonical extension of $E_L\hookrightarrow M_{p_n}^\times$, where $M_{p_n}^\times$ is the reduced mapping cylinder of $p_n$. We may interpret the homotopy pushout $X_n(L)$ as:
\begin{itemize}
\item The space $M_{p_n}^\times\cup_{E_L} M$. In particular, $X_n(L)$ is obtained from $M$ by attaching cells of dimensions 2 and higher.
\item The space $M_{p_n}^\times\cup_{\partial E_L}\nu L$ obtained by gluing $\nu L$ to the reduced mapping cylinder $M_{p_n}^\times$.
\end{itemize}

The canonical projections $p_{m,n}:\pi/\Gamma_m\twoheadrightarrow\pi/\Gamma_n$ for $m\geq n$ induce morphisms of diagrams $\mathcal{D}_m(L)\to\mathcal{D}_n(L)$ over $\pi_1(M)$ which induce based maps $\psi_{m,n}(L):X_m(L)\to X_n(L)$. We will often suppress notation and write $\psi_{m,n}$. 
Via these maps $\psi_{m,n}$, the collection of spaces $\{X_n(L)\}$ forms a tower over $\pi_1(M)$:
\begin{center}
\begin{tikzcd}
& & M \arrow[dl,"\iota_{n+1}(L)"'] \arrow[d,"\iota_n(L)"] \arrow[drr,"\iota_1(L)"] & & &\\
\cdots \arrow[r, "\psi_{n+2{,}n+1}"'] & X_{n+1}(L) \arrow[r,"\psi_{n+1{,}n}"'] & X_n(L) \arrow[r, "\psi_{n{,}n-1}"'] & \cdots \arrow[r, "\psi_{2{,}1}"'] & X_1(L).
\end{tikzcd}
\end{center}
For a list of well-definedness, naturality, and compatibility properties the $X_n(L)$, $\iota_n(L)$, and $\psi_{m,n}$ satisfy, see Appendix~\ref{appendix}.

Now fix a link $L\subset M$. The \emph{lower central homotopy invariants} relative to $L$ will be maps $h_n(L',\phi)\in[M,X_n(L)]_0$. The following notion of \emph{$n$-basing} gives a method of producing a homotopy class of maps depending on some other link $L'\subset M$ into the $n^{\text{th}}$ level $X_n(L)$ of the tower $\{X_n(L)\}$ corresponding to the fixed link $L$.

\begin{definition} \label{results-homotopy-defnbasing}
Fix an $m$-component link $L\subset M$, and let $L'\subset M$ be another $m$-component link. An \emph{$n$-basing} for $L'$ relative to $L$ is an ordered pair $(\phi,\phi_\partial)$ such that
\begin{enumerate}[label=(\arabic*)]
\item $\phi_\partial:H_1(\partial E_{L'})\xrightarrow{\cong}H_1(\partial E_L)$ is an admissible isomorphism, and
\item $\phi:\pi'/\Gamma'_n\xrightarrow{\cong}\pi/\Gamma_n$ is an isomorphism over $\pi_1(M)$ such that $\phi_\partial$ is \emph{compatible} with $\phi$, that is, given any basing $\tau'$ for $L'$, there exists a basing $\tau$ for $L$ such that the following diagram commutes:
\begin{center}
\begin{tikzcd}
\pi_1(\partial E_{L'}\cup\tau') \arrow[r, "\phi_{\partial*}"', "\cong"] \arrow[d] & \pi_1(\partial E_L\cup\tau) \arrow[d] \\
\pi' \arrow[d, twoheadrightarrow] & \pi \arrow[d, twoheadrightarrow] \\
\pi'/\Gamma'_n \arrow[r, "\cong", "\phi"'] & \pi/\Gamma_n.
\end{tikzcd}
\end{center}
\end{enumerate}
\end{definition}
Condition (2) in the above definition is satisfied for any basing $\tau'$ if and only if it is satisfied for a single basing $\tau'$. Recall that an isomorphism $H_1(\partial E_{L'})\xrightarrow{\cong}H_1(\partial E_L)$ is \emph{admissible} if it preserves meridians, orientations, and orderings of components. We call a homeomorphism $\partial E_{L'}\to\partial E_L$ \emph{admissible} if it induces an admissible isomorphism. We often suppress notation and write $\phi$ instead of $(\phi,\phi_\partial)$. We discuss $n$-basings in further detail in Section~\ref{nbasing}. For now, we list some of their relevant properties:
\begin{itemize}
\item If $L'$ is concordant to $L$, an $n$-basing exists for every $n$ (see Corollary~\ref{nbasing-cornbasing}).
\item An $n$-basing induces a $k$-basing for all $k\leq n$. 
\item The existence of a 1-basing is equivalent to the statement that, when based appropriately, corresponding components of $L'$ and $L$ agree in $\pi_1(M)$ (see the discussion at the beginning of Section~\ref{prelims-conc}).
\end{itemize}

Our terminology is inspired by the case of links in $S^3$, where a choice of meridians for a link (a basing of the link), modulo the $n^{\text{th}}$ lower central subgroup of its link group, is equivalent to a choice of isomorphism between the $n^{\text{th}}$ lower central quotient of the link group and that of the unlink, assuming such an isomorphism exists. In Section~\ref{previous-milnor}, we prove that for links in $S^3$ the existence of an $n$-basing for a link relative to the unlink is equivalent to the statement that Milnor's $\overline{\mu}$-invariants of length $n$ vanish. 

An $n$-basing for a link $L'$ relative to $L$ contains sufficient information to define the lower central homotopy invariant $h_n$. We postpone the proof of the following lemma until the end of this section.

\begin{lemma} \label{results-homotopy-lemmanbasing}
An $n$-basing $\phi$ for $L'$ relative to $L$ determines a well-defined based homotopy class of maps of pairs $(E_{L'},\partial E_{L'})\to(M_{p_n}^\times,\partial E_L)$, restricting to an admissible homeomorphism $\partial E_{L'}\xrightarrow{\cong}\partial E_L$, which extends canonically to a based homotopy class of maps of pairs \[h_n(L',\phi):(M,\nu L')\to(X_n(L),\nu L)\] which may be taken to map $L'$ homeomorphically to $L$, all over $\pi_1(M)$.
\end{lemma}

\begin{definition}\label{resluts-homotopy-defhpair}
Fix an $m$-component link $L\subset M$, and let $L'\subset M$ be another $m$-component link. Suppose $L'$ admits an $n$-basing $(\phi,\phi_\partial)$ relative to $L$. The \emph{$n^{\text{th}}$ lower central homotopy invariant} of the pair $(L',\phi)$ relative to $L$ is the homotopy class $h_n(L',\phi)\in[M,X_n(L)]_0$ induced by the $n$-basing $\phi$ (see Lemma~\ref{results-homotopy-lemmanbasing}).
\end{definition}

With our notation, $h_n(L,\text{id})=\iota_n(L)$, where $\text{id}=(\text{id}_{\pi/\Gamma_n},\text{id}_{H_1(\partial E_L)})$. By Lemma~\ref{results-homotopy-lemmanbasing}, $h_n(L',\phi)$ is a well-defined homotopy class, but it may depend on the choice of $n$-basing $\phi$. This dependence is measured by the group of \emph{self-$n$-basings} of $L$, written $\text{Aut}(\pi/\Gamma_n,\partial)$, which consists of all $n$-basings of $L$ relative to itself. The group $\text{Aut}(\pi/\Gamma_n,\partial)$ is a subgroup of $\text{Aut}(\pi/\Gamma_n)\times\text{Aut}(H_1(\partial E_L))$ and acts on the set $[M,X_n(L)]_0$ by post-composition with homotopy self-equivalences of $X_n(L)$. See Section~\ref{nbasing-selfbasing} for more details. To remove the indeterminacy introduced by the choice of $n$-basing, we take the value of $h_n(L',\phi)$ in the orbit space $[M,X_n(L)]_0/\text{Aut}(\pi/\Gamma_n,\partial)$.

\begin{definition}\label{results-homotopy-defh}
Let $L$ and $L'$ be as in the previous definition. The \emph{$n^{\text{th}}$ lower central homotopy invariant} of $L'$ relative to $L$ is the image $h_n(L')$ of the homotopy class $h_n(L',\phi)$ in the orbit space $[M,X_n(L)]_0/\text{Aut}(\pi/\Gamma_n,\partial)$.
\end{definition}

We call $n$ the \emph{length} of the invariant $h_n$. For conciseness, we will often refer to $h_n(L')$ as the \emph{$n^{\text{th}}$ $h$-invariant} of $L'$. To compare the link $L'$ to the fixed link $L$, we compare $h_n(L',\phi)$ and $h_n(L')$ to the distinguished classes $h_n(L,\text{id})$ and $h_n(L)$ associated to the fixed link $L$. 

\begin{definition}\label{results-homotopy-defvanish}
We say the $n^{\text{th}}$ $h$-invariant for $L'$ relative to $L$ \emph{vanishes} if $h_n(L',\phi)=h_n(L,\text{id})\in [M,X_n(L)]_0$ or $h_n(L')=h_n(L)\in [M,X_n(L)]_0/\text{Aut}(\pi/\Gamma_n,\partial)$.
\end{definition}

Theorem~\ref{THMinvariance}, which we prove in Section~\ref{inv}, asserts the concordance invariance of the lower central homotopy invariants $h_n$. Note that we may use $h_n$ not only to compare other links $L'$ to the fixed link $L$, but to compare links $L'$ and $L''$ \emph{relative to the fixed link $L$}.

\begin{customthm}{A} \label{THMinvariance}
The lower central homotopy invariants $h_n$ are invariants of concordance. More precisely, fix an $m$-component link $L\subset M$. Suppose the $m$-component links $L',L''\subset M$ are concordant. Then the following statements hold for all $n$:
\begin{enumerate}[label=(\arabic*)]
\item There is a canonical $n$-basing $(\phi',\phi'_\partial)$ for $L''$ relative to $L'$ induced by the concordance.
\item An $n$-basing for $L'$ relative to $L$ exists if and only if an $n$-basing for $L''$ relative to $L$ exists, and $h_n(L')$ is defined if and only if $h_n(L'')$ is defined.
\item If $(\phi,\phi_\partial)$ is an $n$-basing for $L'$ relative to $L$, then $h_n(L',\phi)=h_n(L'',\phi\circ\phi')$, where $\phi'$ is the $n$-basing in (1).
\item If $h_n(L')$ and $h_n(L'')$ are defined using any $n$-basings relative to $L$, then \[h_n(L')=h_n(L'')\in [M,X_n(L)]_0/\textnormal{Aut}(\pi/\Gamma_n,\partial).\]
\end{enumerate}
In particular, if $L'$ is concordant to the fixed link $L$, then $h_n(L')$ is defined and vanishes for all $n$.
\end{customthm}

By construction, if the invariant $h_n(L',\phi)$ is defined, then $\psi_{n,k}\circ h_n(L',\phi)=h_k(L',\phi_k)$ for all $k\leq n$, where $\phi_k$ is the $k$-basing induced by the $n$-basing $\phi$. Thus, the $n^{\text{th}}$ lower central homotopy invariant determines the $k^{\text{th}}$ invariant for every $k\leq n$. We will eventually show in Corollary~\ref{CORinductive} that the invariants $h_n$ are defined inductively for $n\geq 2$. If the $n^{\text{th}}$ invariant vanishes, then the $(n+1)^{\text{st}}$ invariant is defined. In fact, we will characterize precisely when the $(n+1)^{\text{st}}$ invariant is defined; see Section~\ref{results-determination}.

\begin{proof}[Proof of Lemma~\ref{results-homotopy-lemmanbasing}]
Suppose $L'$ admits an $n$-basing $(\phi,\phi_\partial)$ relative to $L$. Since $\phi_\partial$ is an admissible isomorphism, it corresponds to a unique (up to isotopy) admissible homeomorphism $h_\partial:\partial E_{L'}\to\partial E_L$. Choose any basing $\tau'$ for $L'$. Since $\phi_\partial$ is compatible with $\phi$, there exists a basing $\tau$ for $L$ such that $h_\partial$ may be extended to a map $\partial E_{L'}\cup\tau'\to\partial E_L\cup\tau$ where the following diagram commutes:
\begin{center}
\begin{tikzcd}[row sep=scriptsize, column sep=scriptsize]
\pi_1(\partial E_{L'}\cup\tau') \arrow[r, "\phi_{\partial*}"', "\cong"] \arrow[dd] & \pi_1(\partial E_L\cup\tau) \arrow[dd] \\
& \\
\pi'/\Gamma'_n \arrow[r, "\phi"', "\cong"] & \pi/\Gamma_n.
\end{tikzcd}
\end{center}

Postcomposing with the inclusion $\partial E_L\cup\tau\hookrightarrow M_{p_n}^\times$, we now have the extension problem over $\pi_1(M)$ below.
\begin{center}
\begin{tikzcd}
\partial E_{L'}\cup\tau' \arrow[r] \arrow[d, hookrightarrow] & M_{p_n}^\times  \arrow[r] & B\pi_1(M) \\
 E_{L'} \arrow[ur, dashed, "?\exists"] \arrow[urr] & &
\end{tikzcd}
\end{center}
\noindent The space $M_{p_n}^\times$ is a $K(\pi/\Gamma_n,1)$ and $\phi$ is an isomorphism over $\pi_1(M)$, so there is no obstruction to extending to a map $h:E_{L'}\to M_{p_n}^\times$ inducing $\phi$. The map $h$ is the desired map of pairs $h:(E_{L'},\partial E_{L'})\to(M_{p_n}^\times,\partial E_L)$ over $\pi_1(M)$.

We will next show that the homotopy class of this homotopy equivalence $h$ given above depends only on the $n$-basing $(\phi,\phi_\partial)$. Suppose $h':(E_{L'},\partial E_{L'})\to(M_{p_n}^\times,\partial E_L)$ is another map of pairs over $\pi_1(M)$ defined using the same $n$-basing. We first define the homotopy between $h$ and $h'$ on $\partial E_{L'}$. Since both $h_\partial$ and $h'|_{\partial E_{L'}}$ are homeomorphisms $\partial E_{L'}\xrightarrow{\cong}\partial E_L$ which induce $\phi_\partial$, they are homotopic via some homotopy $H_\partial$. To show $h$ and $h'$ are homotopic, we solve the extension problem over $\pi_1(M)$ below.
\begin{center}
\begin{tikzcd}
(E_{L'}\times\{0,1\})\cup (\partial E_{L'}\times[0,1]) \arrow[rr, "(h\sqcup h')\cup H_\partial"] \arrow[dd, hookrightarrow] & & M_{p_n}^\times \arrow[r] & B\pi_1(M) \\
& & & \\
E_{L'}\times[0,1] \arrow[uurr, dashed, "?\exists"] \arrow[uurrr] & &
\end{tikzcd}
\end{center}

Suppose $\tau'_0$ and $\tau_0$ are the basings for $L'$ and $L$ used to define the map $h$, as above, and $\tau'_1$ and $\tau_1$ are the basings used to define the map $h'$. The basings $\tau_0$ and $\tau_1$ for $L$ differ by a wedge of circles $\alpha$, so that $\alpha\tau_0\simeq\tau_1$ rel endpoints. Similarly, $\tau'_0$ and $\tau'_1$ differ by a wedge of circles $\alpha'$, so that $\alpha'\tau'_0\simeq\tau'_1$. We have $\phi_*(\alpha')\simeq\alpha$ by construction. Let $\alpha_0=h(\alpha')$. In $M_{p_n}^\times$, $\alpha_0\simeq\alpha$.

We extend our map to the trace of a homotopy $H_\tau\subset E_{L'}\times[0,1]$ between $\alpha'\tau'_0$ and $\tau'_1$ over a homotopy from $h(\alpha'\tau'_0)=\alpha_0\tau_0$ to $h'(\tau'_1)=\tau_1$. Note that $\alpha_0\tau_0\simeq\alpha\tau_0\simeq\tau_1$ rel endpoints. We now have a based map $(E_{L'}\times\{0,1\})\cup(\partial E_{L'}\times[0,1])\cup H_\tau\to M_{p_n}^\times$. We extend to all of $E_{L'}\times[0,1]$ as in the previous extension problem. There is no obstruction to extending since $M_{p_n}^\times$ is a $K(\pi/\Gamma_n,1)$ and $h$ and $h'$ both induce the isomorphism $\phi$ which is a homomorphism over $\pi_1(M)$. Thus, $h$ and $h'$ are based homotopic as maps of pairs.

To conclude the proof, we show $h$ extends over $\pi_1(M)$ to a canonical homotopy class of maps of pairs $\overline{h}:(M,\nu L')\to (X_n(L),\nu L)$. We show that $h$ and $h'$ above both extend canonically and the canonical extensions are homotopic. The homotopy seen above between $h$ and $h'$ is a homotopy of pairs sending $\partial E_{L'}\times[0,1]\to\partial E_L$. We may view this as a homotopy of maps from the sphere bundle of $L'\subset M'$ to the sphere bundle of $L\subset M$. There is a canonical extension of this homotopy to a homotopy of disk bundles $\nu L'\times[0,1]\to\nu L$ given by coning on the fibers of the sphere bundles. Since $X_n(L)=M_{p_n}^\times\cup_{\partial E_L}\nu L$ and $M=E_{L'}\cup_{\partial E_{L'}}\nu L'$, we may extend the homotopy $E_{L'}\times[0,1]\to M_{p_n}^\times$ constructed above to a homotopy $M\times[0,1]\to X_n(L)$. This yields a canonical homotopy class of maps of pairs $\overline{h}:(M,\nu L')\to (X_n(L),\nu L)$ which is still a map over $\pi_1(M)$ because the 2-cells we have attached are glued along elements in $\Gamma'$ and $\Gamma$. Just as $h$ may be taken to restrict to a homeomorphism of sphere bundles $h_\partial:\partial E_{L'}\xrightarrow{\cong}\partial E_L$, $\overline{h}$ may be taken to restrict to a homeomorphism of disk bundles $h_\nu:\nu L'\xrightarrow{\cong}\nu L$ sending $L'$ to $L$.
\end{proof}


\subsection{\texorpdfstring{$n$}{n}-cobordism of links}\label{results-ncob}

The lower central homotopy invariants provide obstructions to the existence of surfaces which successively approximate a concordance between links in $M\times[0,1]$. We call such surfaces \emph{$n$-cobordisms}. K. Orr introduced and studied such surfaces for links in $S^3$ in \cite{Orr89}.

\begin{definition}\label{results-ncob-defncob}
Two $m$-component links $L, L'\subset M$ are \emph{$n$-cobordant} if there exists a properly-embedded orientable $m$-component surface $\Sigma=\sqcup_i\Sigma_i\subset M\times [0,1]$ such that for each $i$:
\begin{enumerate}[label=(\arabic*)]
\item $\Sigma_i\cap (M\times\{0\})=L_i$ and $\Sigma_i\cap (M\times\{1\})=L'_i$.
\item For any basing of $\Sigma_i$ determined by a basing of $L_i$, the images of $\pi_1(L_i)$ and $\pi_1(\Sigma_i)$ in $\pi_1(E_\Sigma)$ agree modulo the $n^\text{th}$ $\pi_1(M)$-lower central subgroup, that is,  \[\text{im}\Big(\pi_1(L_i)\to\pi\to\pi^\Sigma\twoheadrightarrow\pi^\Sigma/\Gamma^\Sigma_n\Big)=\text{im}\Big(\pi_1(\Sigma_i)\to\pi^\Sigma\twoheadrightarrow\pi^\Sigma/\Gamma^\Sigma_n\Big),\] where $\pi^\Sigma=\pi_1(E_\Sigma)$ and $\Gamma^\Sigma=\ker\big(\pi^\Sigma\twoheadrightarrow\pi_1(M)\big)$, and where the inclusions $\Sigma_i\hookrightarrow E_\Sigma$ and $L_i\hookrightarrow E_L$ are induced by a trivialization of $\nu\Sigma$.
\end{enumerate}
\end{definition}
In other words, an $n$-cobordism is a surface that looks like a concordance to the $n^{\text{th}}$ $\pi_1(M)$-lower central quotients. A concordance is an $n$-cobordism for all $n$. Condition (2) implies that for all $i$ the images of the longitudes of $L_i$ and $L_i'$, when based appropriately, agree in $\pi^\Sigma/\Gamma^\Sigma_n$: They cobound a copy of $\Sigma_i$ on $\partial\nu\Sigma$ and are therefore homologous in $E_\Sigma$. Since the images of $\pi_1(L_i)$ and $\pi_1(\Sigma_i)$ in $\pi^\Sigma/\Gamma^\Sigma_n$ agree, and because the image of $\pi_1(L_i)$ is cyclic and therefore abelian, the longitudes agree in $\pi^\Sigma/\Gamma^\Sigma_n$. We prove the following property of $n$-cobordism in Section~\ref{htpychar}:

\begin{proposition} \label{results-ncob-propncob}
An $n$-cobordism between the links $L$ and $L'$ induces a canonical $n$-basing of $L'$ relative to $L$ and a canonical isomorphism $\pi'/\Gamma'_{n+1}\xrightarrow{\cong}\pi/\Gamma_{n+1}$.
\end{proposition}

We will also prove the following theorem which states that the lower central homotopy invariants are complete invariants of $n$-cobordism.  Note again that the invariants $h_n$ can compare two links $L'$ and $L''$ \emph{over the fixed link $L$}. A comparison between the $h$-invariants of $L'$ and $L''$ is still meaningful even if they do not agree with the $h$-invariant of $L$. P. Heck proved a version of this theorem for knots in prime 3-manifolds in the case where we compare some knot to the fixed knot \cite{Heck}.

\begin{customthm}{B} \label{THMcharacterization}
Fix an $m$-component link $L\subset M$. Let $L', L''\subset M$ be $m$-component links, and suppose there exist $n$-basings $(\phi',\phi'_\partial)$ and $(\phi'',\phi''_\partial)$ for $L'$ and $L''$, respectively, relative to $L$. Then the following statements are equivalent:
\begin{enumerate}[label=(\arabic*)]
\item We have $h_n(L',\phi')=h_n(L'',\phi'')$.
\item The links $L'$ and $L''$ are $n$-cobordant via an $n$-cobordism inducing the $n$-basing \[(\psi,\psi_\partial)=((\phi')^{-1}\circ\phi'',(\phi'_\partial)^{-1}\circ\phi''_\partial)\] of $L''$ relative to $L'$.
\end{enumerate}
As a consequence, the following basing-independent statements are equivalent:
\begin{enumerate}[label=(\arabic*\,$'$)]
\item We have $h_n(L')=h_n(L'')$.
\item The links $L'$ and $L''$ are $n$-cobordant. 
\end{enumerate}
In particular, $h_n(L',\phi')$ vanishes if and only if $L'$ is $n$-cobordant to the fixed link $L$ via a surface whose exterior induces the isomorphisms $\phi'$ and $\phi'_\partial$, and $h_n(L')$ vanishes if and only if $L'$ is $n$-cobordant to the fixed link $L$.
\end{customthm}


\subsection{Realization of lower central homotopy invariants}\label{results-htpyreal}

Questions concerning the realizability of Milnor-type invariants date back to Milnor's original work \cite{Milnor}. Leveraging a realization theorem, Orr computed the number of linearly independent $\overline{\mu}$-invariants of length $n$ for every $n$ \cite{Orr89}. Similar computations for knots and links in other 3-manifolds can provide answers to open questions. Recall the \hyperref[CONJ]{Almost-Concordance Conjecture} from Section~\ref{prelims-conc}. Celoria \cite{Celoria} and Friedl-Nagel-Orson-Powell \cite{FNOP} conjecture that, aside from certain exceptional cases, there are infinitely many almost-concordance classes representing any free homotopy class in $[S^1,M]$, where $M$ is any closed orientable 3-manifold. Cappell and Shaneson first studied concordance up to local knotting (see, for instance, \cite{CappellShaneson}). More recently, Celoria coined the term \emph{almost-concordance} to mean concordance up to local knotting.

As the lower central homotopy invariants are almost-concordance invariants, proving either of the following statements is sufficient for proving the \hyperref[CONJ]{Almost-Concordance Conjecture} for some fixed $M$ and some fixed free homotopy class $x\in[S^1,M]$:

\vspace{1em}
\begin{minipage}{.95\textwidth}
\begin{enumerate}[label=\emph{(\arabic*)}]
\item \emph{For some knot $K\subset M$ representing $x$ and some $n\in\mathbb{N}$, the set of equivalence classes of elements in $[M,X_n(K)]_0$ realized as nonvanishing invariants of knots in $M$ is infinite.}
\item \emph{For some knot $K\subset M$ representing $x$, there exists a family $\{K_n\}_{n\geq 1}$ of knots in $M$ such that $h_i(K_n)=h_i(K)$ for $i<n$ but $h_n(K_n)\neq h_n(K)$.}
\end{enumerate}
\end{minipage}
\vspace{1em}

\noindent The following theorem, which characterizes for a fixed $L\subset M$ which classes in $[M,X_n(L)]_0$ are realized by links in $M$, can be leveraged to provide far more information about concordance classes of knots in a fixed free homotopy class. We prove this realization theorem in Section~\ref{htpychar}.

\begin{customthm}{C} \label{THMrealization}
Fix an $m$-component link $L\subset M$, and let $f\in[M,X_n(L)]_0$. Then there exists an $m$-component link $L'\subset M$ which admits an $n$-basing $\phi$ relative to $L$ such that $h_n(L',\phi)=f$ if and only if the following three conditions hold.
\begin{enumerate}[label=(\arabic*)]
\item The composition $H_3(M)\xrightarrow{f_*}H_3(X_n(L))\xrightarrow{\Delta}H_2(\partial E_L)$, where $\Delta$ is the connecting homomorphism in the Mayer-Vietoris sequence corresponding to the decomposition $X_n(L)=M_{p_n}^\times\cup_{\partial E_L}\nu L$, sends the fundamental class $[M]$ to $\sum_i[T_i]$, the sum of the fundamental classes of the tori in $\partial E_L=\sqcup_i T_i$.
\item On fundamental groups, we have $f_*=\iota_n(L)_*$.
\item The image of the composition \[H^1(X_n(L))\xrightarrow{\cap f_*[M]}H_2(X_n(L))\xrightarrow{\textnormal{proj}}H_2(X_n(L))/\textnormal{im}(K_j(\pi/\Gamma_n))\] contains the image of $H_2(\pi/\Gamma_n)/K_j(\pi/\Gamma_n)$ for all $j<n$, where \[K_j(\pi/\Gamma_n)=\ker\Big(H_2(\pi/\Gamma_n)\to H_2(\pi/\Gamma_j)\Big),\] and where all coefficients are in $\mathbb{Z}[\pi_1(M)]$.
\end{enumerate}
\end{customthm}
Condition (3) above may be viewed as a Poincar\'{e} duality property imposed on the class $f_*[M]$. It holds for all $j<n$ if and only if it holds for $j=n-1$. We will see in the proof of Theorem~\ref{THMrealization} that condition (3) follows from conditions (1) and (2) in the case where $\pi_2(M)=0$ or, more generally, if the composition \[H_2(E_L;\mathbb{Z}[\pi_1(M)])\xrightarrow{j_*} H_2(M;\mathbb{Z}[\pi_1(M)])\xrightarrow{\iota_n(L)_*} H_2(X_n(L);\mathbb{Z}[\pi_1(M)])\] is zero; see Corollary~\ref{htpychar-thmC-corpi2}.


\subsection{\texorpdfstring{$G$}{G}-lower central homotopy invariants} \label{results-overG}

We may reduce the lower central homotopy invariants from Section~\ref{results-homotopy} by changing from the group $\pi_1(M)$ to some quotient group $G$ (which could be the trivial group). We will use the resulting $G$-lower central homotopy invariants to study $\mathbb{Z}[G]$-homology concordance in Sections~\ref{results-homology}-\ref{results-homreal}. For this section, we assume there is a fixed surjective homomorphism $\varphi:\pi_1(M)\twoheadrightarrow G$ onto some group $G$. For a link $L\subset M$, we now denote the kernel of the composition $\pi\twoheadrightarrow\pi_1(M)\twoheadrightarrow G$ by $\Gamma$, so that the groups $\pi/\Gamma_n$ are now the $G$-lower central quotients of $\pi$. The homological properties of concordance which hold over $\pi_1(M)$ also hold over $G$. A concordance exterior is a $\mathbb{Z}[G]$-homology cobordism rel boundary and induces isomorphisms over $G$ on $G$-lower central quotients of link groups.

Just as in Section~\ref{results-homotopy}, we define a tower of spaces $\{X_n^G(L)\}$ with maps $\iota_n^G(L):M\to X_n^G(L)$. The space $X_n^G(L)$ is the homotopy pushout of a diagram $\mathcal{D}^G_n(L)$ analogous to the diagram $\mathcal{D}_n(L)$ from Section~\ref{results-homotopy} but with the group $\pi/\Gamma_n$ the $n^{\text{th}}$ $G$-lower central quotient of $\pi$. The map $\iota_n^G(L)$ is still the homotopy class $M\to X_n^G(L)$ induced by the homotopy pushout diagram. We have the diagram over $G$ below.
\begin{center}
\begin{tikzcd}
& & M \arrow[dl,"\iota^G_{n+1}(L)"'] \arrow[d,"\iota^G_n(L)"] \arrow[drr,"\iota^G_1(L)"] & & &\\
\cdots \arrow[r, "\psi^G_{n+2{,}n+1}"'] & X^G_{n+1}(L) \arrow[r,"\psi^G_{n+1{,}n}"'] & X^G_n(L) \arrow[r, "\psi^G_{n{,}n-1}"'] & \cdots \arrow[r, "\psi^G_{2{,}1}"'] & X^G_1(L)
\end{tikzcd}
\end{center}
\vspace{0.5em}
\noindent We will often omit the group $G$ from our notation when it is understood from context.

To define invariants in this setting, we fix a link $L\subset M$. For other links $L'\subset M$, we require a $\mathbb{Z}[G]$-homology version of an $n$-basing which we will call an \emph{$n$-basing over $G$} relative to $L$. 
When $G=\pi_1(M)$ and $\varphi=\text{id}_{\pi_1(M)}$, an $n$-basing over $G$ is just an $n$-basing in the sense of Section~\ref{results-homotopy}. 
We will often refer to an $n$-basing over $G$ simply as an \emph{$n$-basing}. 
An analogue of Lemma~\ref{results-homotopy-lemmanbasing}, Proposition~\ref{nbasing-propnbasing}, implies that an $n$-basing $\phi$ over $G$ for $L'$ relative to $L$ induces a well-defined based homotopy class of maps of pairs $(E_{L'},\partial E_{L'})\to(M_{p_n}^\times,\partial E_L)$, restricting to an admissible homeomorphism $\partial E_{L'}\xrightarrow{\cong}\partial E_L$, which extends canonically to a based homotopy class of maps of pairs $h_n^G(L',\phi):(M,\nu L')\to(X_n^G(L),\nu L)$, all over $G$. 
We define the \emph{$G$-lower central homotopy invariant} to be this homotopy class $h^G_n(L',\phi)$. 
This invariant agrees with the lower central homotopy invariant $h_n(L',\phi)$ when $G=\pi_1(M)$ and $\varphi=\text{id}_{\pi_1(M)}$. 
Note that $h^G_n(L,\text{id}_{\pi/\Gamma_n})=\iota^G_n(L)$. We may again remove the dependence on the $n$-basing by defining $h^G_n(L')$, the image of $h^G_n(L',\phi)$ in the orbit space $[M,X^G_n(L)]_0/\text{Aut}(\pi/\Gamma_n,\partial)$, where the group acting is now the group of self-$n$-basings over $G$. Analogues of Theorems~\ref{THMinvariance} and~\ref{THMcharacterization} hold for the invariants $h^G_n$.

Any $n$-basing $\phi$ over $\pi_1(M)$ induces an $n$-basing $\phi^G$ over $G$. Assuming an $n$-basing for $L'$ relative to $L$ exists, we may consider the relationship between $h_n(L',\phi)$ and $h^G_n(L',\phi^G)$. The surjections of $\pi_1(M)$-lower central quotients onto $G$-lower central quotients yield morphisms of homotopy pushout squares $\mathcal{D}_n(L)\to\mathcal{D}^G_n(L)$ inducing well-defined homotopy classes of maps $\rho^G_n:X_n(L)\to X^G_n(L)$. It is straightforward to show that \[\rho^G_n\circ h_n(L',\phi)=h^G_n(L',\phi^G)\in[M,X^G_n(L)]_0.\]


\subsection{Homology concordance invariants}\label{results-homology}

We now consider the ideas from Sections~\ref{results-homotopy} and~\ref{results-overG} in the broader setting of $\mathbb{Z}[G]$-homology concordance, concordance in $\mathbb{Z}[G]$-homology cobordisms, for suitable groups $G$. 
We introduce two $\mathbb{Z}[G]$-homology concordance invariants $\theta^G_n$ and $\overline{\mu}^G_n$ determined by the $G$-lower central homotopy invariant $h^G_n$ from Section~\ref{results-overG}. As we will see in Theorem~\ref{THMhcharacterization}, the invariants $\overline{\mu}^G_n$ inductively determine the $G$-lower central quotients of link groups.

For the remainder of Section~\ref{results}, we allow the links $L$, $L'$, and $L''$ to live in different closed orientable 3-manifolds $M$, $M'$, and $M''$. 
We assume there are fixed surjective homomorphisms $\varphi$, $\varphi'$, and $\varphi''$ from the fundamental groups of these 3-manifolds onto a common group $G$ (which could be the trivial group). 
We consider these homomorphisms as fixed data accompanying the 3-manifolds and use them to define $\mathbb{Z}[G]$-coefficient systems on the 3-manifolds. 
Denote the kernels of the surjective homomorphisms from the link groups onto $G$ by $\Gamma$, $\Gamma'$, and $\Gamma''$, respectively. 

\begin{definition} \label{results-homology-defhcob}
A \emph{$\mathbb{Z}[G]$-homology cobordism} from $(M,\varphi)$ to $(M',\varphi')$ is an orientable 4-manifold $W$ with boundary $M\sqcup M'$, along with a diagram
\begin{center}
\begin{tikzcd}
M' \arrow[d,hookrightarrow] \arrow[dr, "\varphi'"] & \\
W \arrow[r] & BG \\
M \arrow[u,hookrightarrow] \arrow[ur, "\varphi"'] &
\end{tikzcd}
\end{center}
such that the inclusions $M\hookrightarrow W$ and $M'\hookrightarrow W$ induce isomorphisms on $H_*(-;\mathbb{Z}[G])$.
\end{definition}

\begin{definition} \label{homology-defhconc}
Two $m$-component links $L\subset M$ and $L'\subset M'$ are \emph{$\mathbb{Z}[G]$-homology concordant} if there exists a proper, locally flat embedding
\[ C: \bigsqcup_{i=1}^m (S^1\times[0,1])\hookrightarrow W \] into some $\mathbb{Z}[G]$-homology cobordism $W$ between $(M,\varphi)$ and $(M',\varphi')$ which restricts on the ends to the links $L$ and $L'$.
\end{definition}

Note that we return to the setting of Section~\ref{results-homotopy} if we take $M=M'$, $W=M\times[0,1]$, $G=\pi_1(M)$, and $\varphi=\varphi'=\text{id}_{\pi_1(M)}$, as a concordance is a $\mathbb{Z}[\pi_1(M)]$-homology concordance. 
Taking $G=\{1\}$ reduces Definition~\ref{homology-defhconc} to the more familiar notion of integral homology concordance.

The fact that $L$ and $L'$ live in different 3-manifolds does not prevent us from comparing the towers $\{X^G_n(L')\}$ and $\{X^G_n(L)\}$.
Given an $n$-basing $(\phi,\phi_\partial)$ over $G$ for $L'$ relative to $L$, we may define the homotopy class $h^G_n(L',\phi)$; however, comparing the homotopy classes $h^G_n(L',\phi)\in[M',X^G_n(L)]_0$ and $h^G_n(L,\text{id})\in[M,X^G_n(L)]_0$ is not straightforward. 
We would like to determine whether these two classes extend over some $\mathbb{Z}[G]$-homology cobordism between $(M,\varphi)$ and $(M',\varphi')$. 
This may be difficult to measure, so we instead ask if the images of the fundamental classes $[M]$ and $[M']$ agree in $H_3(X^G_n(L))$. 
This is the case if the homotopy classes $h^G_n(L',\phi)$ and $h^G_n(L,\text{id})$ extend over some $\mathbb{Z}[G]$-homology cobordism. 
We are immediately led to a reduction of the lower central homotopy invariants which we call the \emph{lower central homology invariants} $\theta^G_n$.

\begin{definition}\label{results-homology-defthetapairh}
Fix an $m$-component link $L\subset M$, and let $L'\subset M'$ be another $m$-component link. 
Suppose $L'$ admits an $n$-basing $(\phi,\phi_\partial)$ over $G$ relative to $L$. 
Then the \emph{$n^{\text{th}}$ $G$-lower central homology invariant} of the pair $(L',\phi)$ relative to $L$ is the homology class \[\theta^G_n(L',\phi)=h^G_n(L',\phi)_*[M']\in H_3(X^G_n(L)).\]
\end{definition}

We will often suppress the group $G$ from our notation when it is understood from context. 
To be concise, we will often refer to $\theta_n$ as the $n^{\text{th}}$ $\theta$-invariant. 
The $n^{\text{th}}$ $\theta$-invariant is well-defined because the homotopy class $h^G_n(L',\phi)$ is well-defined. 
As with the $h$-invariant, the $\theta$-invariant may depend on the choice of $n$-basing.
We account for this by taking the image of $\theta^G_n(L',\phi)$ in the orbit space $H_3(X^G_n(L))/\text{Aut}(\pi/\Gamma_n,\partial)$, where $\text{Aut}(\pi/\Gamma_n,\partial)$ now denotes the group of \emph{self-$n$-basings of L over G}.

\begin{definition}\label{results-homology-deftheta}
Let $L$ and $L'$ be as in the previous definition. 
The \emph{$n^{\text{th}}$ $G$-lower central homology invariant} of the link $L'$ relative to $L$ is the image $\theta^G_n(L')$ of the homology class $\theta^G_n(L',\phi)$ in the orbit space $H_3(X^G_n(L))/\text{Aut}(\pi/\Gamma_n,\partial)$.
\end{definition}

As before, we will say the $\theta$-invariant \emph{vanishes} if $\theta^G_n(L',\phi)=\theta^G_n(L,\text{id})\in H_3(X^G_n(L))$ or if $\theta^G_n(L')=\theta^G_n(L)\in H_3(X^G_n(L))/\text{Aut}(\pi/\Gamma_n,\partial)$. 
We summarize invariance properties of the $\theta$-invariants in Theorem~\ref{THMhinvariance} which we prove in Section~\ref{inv}.

\begin{customthm}{A$'$}\label{THMhinvariance}
The $n^{\text{th}}$ $G$-lower central homology invariant $\theta^G_n$ is an invariant of $\mathbb{Z}[G]$-homology concordance.
More precisely, fix an $m$-component link $L\subset (M,\varphi)$. 
Suppose the $m$-component links $L'\subset (M',\varphi')$ and $L''\subset (M'',\varphi'')$ are $\mathbb{Z}[G]$-homology concordant. 
Then the following statements hold for all $n$:
\begin{enumerate}[label=(\arabic*)]
\item There is an $n$-basing $(\phi',\phi'_\partial)$ over $G$ for $L''$ relative to $L'$ induced by the homology concordance.
\item An $n$-basing over $G$ for $L'$ relative to $L$ exists if and only if an $n$-basing over $G$ for $L''$ relative to $L$ exists, and the homology class $\theta^G_n(L')$ is defined if and only if $\theta^G_n(L'')$ is defined.
\item If $(\phi,\phi_\partial)$ is an $n$-basing over $G$ for $L'$ relative to $L$, then $\theta^G_n(L',\phi)=\theta^G_n(L'',\phi\circ\phi')$, where $\phi'$ is any $n$-basing from (1).
\item If $\theta^G_n(L')$ and $\theta^G_n(L'')$ are defined using any $n$-basings over $G$ relative to $L$, then \[\theta^G_n(L')=\theta^G_n(L'')\in H_3(X^G_n(L))/\textnormal{Aut}(\pi/\Gamma_n,\partial).\]
\end{enumerate}
\end{customthm}


\subsection{Determination of \texorpdfstring{$G$}{G}-lower central quotients} \label{results-determination}

We now consider the set of homology classes realized by the invariants $\theta^G_n$. 
Not all classes are necessarily realizable. 
In fact, the same is true for the invariants from \cite{ChaOrr} which the $\theta^G_n$ generalize. 
Define
\[\mathcal{R}^G_n(L)=\left\lbrace \theta\in H_3(X^G_n(L)) \;\middle|\;
\begin{tabular}{@{}l@{}}
$\theta=\theta^G_n(L',\phi)$ for some link $L'$ in some 3-manifold\\
$(M',\varphi')$ with $\varphi':\pi_1(M')\twoheadrightarrow G$ which admits\\
some $n$-basing $\phi$ over $G$ relative to $L$
\end{tabular}
\right\rbrace. \]
We will often suppress $G$ from the notation and write $X_n(L)$, $\theta_n(L)$, and $\mathcal{R}_n(L)$.
We have a well-defined function $\mathcal{R}_{n+1}(L)\to\mathcal{R}_n(L)$ induced by the canonical map $\psi_{n+1,n}$. 
This is a restriction of the homomorphism $(\psi_{n+1,n})_*:H_3(X_{n+1}(L))\to H_3(X_n(L))$.

As seen in \cite{ChaOrr}, the set $\mathcal{R}_n(L)$ is not necessarily a subgroup of $H_3(X_n(L))$, so the notion of the cokernel of $\mathcal{R}_{n+1}(L)\to\mathcal{R}_n(L)$ is not well-defined. 
Nevertheless, we make the following definition:
\begin{definition}\label{results-determination-defcoker}
A class $\theta\in\mathcal{R}_n(L)$ \emph{vanishes in the cokernel} of $\mathcal{R}_{n+1}(L)\to\mathcal{R}_n(L)$ if $\theta$ lies in the image of the map $\mathcal{R}_{n+1}(L)\to\mathcal{R}_n(L)$.
\end{definition}
\noindent In other words, a realizable class $\theta\in H_3(X_n(L))$ vanishes in the cokernel if there exists a link $L'\subset (M',\varphi')$ which admits an $(n+1)$-basing $\phi$ over $G$ relative to $L$ such that $(\psi_{n+1,n})_*(\theta_{n+1}(L',\phi))=\theta$. 

Likewise, we will say a class $\theta\in\mathcal{R}_n(L)/\text{Aut}(\pi/\Gamma_n,\partial)$ \emph{vanishes in the cokernel} if $\theta$ lies in the image of the composite \[\mathcal{R}_{n+1}(L)\to\mathcal{R}_n(L)\to\mathcal{R}_n(L)/\text{Aut}(\pi/\Gamma_n,\partial).\] 
The orbit space $\mathcal{R}_n(L)/\text{Aut}(\pi/\Gamma_n,\partial)$ is well-defined because the action of any self-basing on a realizable class yields another realizable class. 
The following theorem, which we prove in Section~\ref{homchar}, characterizes the notion of vanishing in the cokernel.

\begin{customthm}{D$_0$}\label{THMcoker}
Fix an $m$-component link $L\subset M$. 
Suppose the $m$-component link $L'\subset M'$ admits an $n$-basing $\phi$ over $G$ relative to $L$ for some $n\geq 2$. 
Then the following statements are equivalent:
\begin{enumerate}[label=(\arabic*)]
\item The link $L'$ admits an $(n+1)$-basing $\widetilde{\phi}$ over $G$ relative to $L$ which is a lift of $\phi$.
\item The invariant $\theta_n(L',\phi)$ vanishes in the cokernel of $\mathcal{R}_{n+1}(L)\to\mathcal{R}_n(L)$.
\end{enumerate}
As a consequence, the following basing-independent statements are equivalent:
\begin{enumerate}[label=(\arabic*\,$'$)]
\item The link $L'$ admits some $(n+1)$-basing over $G$ relative to $L$.
\item The invariant $\theta_n(L')$ vanishes in the cokernel of $\mathcal{R}_{n+1}(L)\to\mathcal{R}_n(L)/\textnormal{Aut}(\pi/\Gamma_n,\partial)$.
\end{enumerate}
\end{customthm}

The notion of vanishing in the cokernel extends to an equivalence relation $\sim$ on all of $\mathcal{R}_n(L)$. 
Rather than comparing a link $L'$ to the fixed link $L$, we can compare two links $L'$ and $L''$ \emph{over the fixed link $L$}. 
Given $\theta\in\mathcal{R}_n(L)$, there exists a 3-manifold $(M',\varphi')$ and a link $L'\subset M'$ which admits an $n$-basing $\phi$ relative to $L$ such that $\theta_n(L',\phi)=\theta$. 
Observe that $\phi$ induces a bijection $\mathcal{R}_n(L')\xrightarrow{\cong}\mathcal{R}_n(L)$ via $\theta_n(L'',\phi')\mapsto\theta_n(L'',\phi\circ\phi')$. 
Define $I_\theta$ to be the image of the composition $\mathcal{R}_{n+1}(L')\to\mathcal{R}_n(L')\xrightarrow{\cong}\mathcal{R}_n(L)$. 
We prove in Lemma~\ref{homchar-thmD-lemmaequiv1} that the set $\{\,I_\theta\,|\, \theta\in\mathcal{R}_n(L)\,\}$ partitions $\mathcal{R}_n(L)$. 
Let $\sim$ be the associated equivalence relation.

\begin{definition}\label{results-homology-defequiv1}
Let $\theta,\theta'\in\mathcal{R}_n(L)$. We say that $\theta\sim\theta'$ if and only if $\theta'\in I_\theta$.
\end{definition}
The invariant $\overline{\mu}^G_n$ is the equivalence class of the $\theta$-invariant under this equivalence relation.

\begin{definition}\label{results-homology-defmupair}
Fix an $m$-component link $L\subset M$, and let $L'\subset M'$ be another $m$-component link.
Suppose $L'$ admits an $n$-basing $(\phi,\phi_\partial)$ over $G$ relative to $L$. 
The \emph{$n^{\text{th}}$ $G$-Milnor invariant} of the pair $(L',\phi)$ relative to $L$ is $\overline{\mu}^G_n(L',\phi)=[\theta^G_n(L',\phi)]\in \mathcal{R}^G_n(L)/\sim$, the class of the $n^{\text{th}}$ $\theta$-invariant under the equivalence relation $\sim$.
\end{definition}

Again, we sometimes suppress $G$ when it is understood from context. 
As with the $h$- and $\theta$-invariants, we can remove dependency on the choice of $n$-basing $\phi$. 
We define a coarser equivalence relation $\approx$ on $\mathcal{R}_n(L)$ as follows:

\begin{definition}\label{results-homology-defequiv2}
Let $\theta,\theta'\in\mathcal{R}_n(L)$. 
We say that $\theta\approx\theta'$ if and only if there exists a self-basing $\psi=(\psi,\psi_\partial)\in\text{Aut}(\pi/\Gamma_n,\partial)$ of $L$ such that $\psi\cdot\theta'\sim\theta$. 
In other words, $\theta\approx\theta'$ if, upon choosing $L'\subset M'$ which admits an $n$-basing $\phi$ over $G$ relative to $L$ with $\theta_n(L',\phi)=\theta$, the class $[\theta']$ is in the image of the composition \[\mathcal{R}_{n+1}(L')\to\mathcal{R}_n(L')\xrightarrow[\phi_*]{\cong}\mathcal{R}_n(L)\to\mathcal{R}_n(L)/\text{Aut}(\pi/\Gamma_n,\partial).\]
\end{definition}

\begin{definition}\label{results-homology-defmu}
Let $L$ and $L'$ be as in Definition~\ref{results-homology-defmupair}. 
The \emph{$n^{\text{th}}$ $G$-Milnor invariant} of the link $L'$ relative to $L$ is $\overline{\mu}^G_n(L')=[\theta^G_n(L',\phi)]\in \mathcal{R}^G_n(L)/\approx$, the class of the $n^{\text{th}}$ $\theta$-invariant under the equivalence relation $\approx$.
\end{definition}

\noindent As with the $h$- and $\theta$-invariants, we will say the $\overline{\mu}$-invariant \emph{vanishes} if we have $\overline{\mu}^G_n(L',\phi)=\overline{\mu}^G_n(L,\text{id})$ or $\overline{\mu}^G_n(L')=\overline{\mu^G}_n(L)$. 
Observe that $\overline{\mu}^G_n(L',\phi)$ vanishes if and only if $\theta^G_n(L',\phi)$ vanishes in the cokernel, and likewise for $\overline{\mu}^G_n(L')$ and $\theta^G_n(L')$.

We prove the following theorem, which strengthens Theorem~\ref{THMcoker}, in Section~\ref{homchar}. 
This theorem reduces to \hyperref[THMmilnor]{Milnor's Theorem} when we take $M=M'=S^3$ and our fixed link $L$ to be the unlink. 
It is both a justification for calling these new invariants $\overline{\mu}_n$ ``Milnor's invariants" and a way to show that \emph{the vanishing of these invariants implies the vanishing of all previous versions of Milnor's invariants in dimension 3}. 
Note that we have already seen a version of this theorem in Section~\ref{intro-new}.

\begin{customthm}{D} \label{THMhcharacterization}
Fix an $m$-component link $L\subset (M,\varphi)$. 
Suppose the $m$-component links $L'\subset (M',\varphi')$ and $L''\subset (M'',\varphi'')$ admit $n$-basings $\phi'$ and $\phi''$ over $G$ relative to $L$ for some $n\geq 2$. Then the following statements are equivalent:
\begin{enumerate}[label=(\arabic*)]
\item The link $L'$ admits an $(n+1)$-basing $\widetilde{\phi'}$ over $G$ relative to $L$ which is a lift of $\phi'$.
\item The invariant $\overline{\mu}_n(L',\phi')$ vanishes.
\item The invariant $\overline{\mu}_{n+1}(L',\widetilde{\phi'})$ is well-defined.
\end{enumerate}
As a consequence, the following basing-independent statements are equivalent:
\begin{enumerate}[label=(\arabic*\,$'$)]
\item The link $L'$ admits some $(n+1)$-basing over $G$ relative to $L$.
\item The invariant $\overline{\mu}_n(L')$ vanishes.
\item The invariant $\overline{\mu}_{n+1}(L')$ is well-defined.
\end{enumerate}
Even if the $\overline{\mu}$-invariant of $L'$ does not vanish, we have the following equivalent statements concerning the links $L'$ and $L''$:
\begin{enumerate}[label=(\arabic*)]
\setcounter{enumi}{3}
\item The $n$-basing $(\phi')^{-1}\circ\phi''$ of $L''$ over $G$ relative to $L'$ lifts to an $(n+1)$-basing.
\item We have $\overline{\mu}_n(L',\phi')=\overline{\mu}_n(L'',\phi'')$.
\end{enumerate}
As a consequence, the following basing-independent statements are equivalent:
\begin{enumerate}[label=(\arabic*\,$'$)]
\setcounter{enumi}{3}
\item The link $L''$ admits some $(n+1)$-basing over $G$ relative to $L'$.
\item We have $\overline{\mu}_n(L')=\overline{\mu}_n(L'')$.
\end{enumerate}
\end{customthm}

\noindent Note that $(1)$-$(3)$ and $(1')$-$(3')$ in the above theorem are just the statement of Theorem~\ref{THMcoker}.

\begin{customcor}{D$_1$}\label{CORinductive}
For $n\geq 2$, the lower central homotopy invariants, $G$-lower central homology invariants, and $G$-Milnor's invariants are defined inductively. 
If any of these invariants vanishes at level $n$, the invariant at level $n+1$ is well-defined.
\end{customcor}

\begin{proof}
All that is required to define any of the invariants for $L'$ relative to $L$ at level $n+1$ is an $(n+1)$-basing for $L'$ relative to $L$. 
The $\overline{\mu}$-invariants are determined by both the $h$- and $\theta$- invariants. 
If any $n^{\text{th}}$ invariant is defined and vanishes, then the $n^{\text{th}}$ $\overline{\mu}$-invariant is defined and vanishes. 
Theorem~\ref{THMhcharacterization} then implies an $(n+1)$-basing exists.
\end{proof}

We have just seen in Theorem~\ref{THMhcharacterization} that the invariants $\overline{\mu}_n$ determine the lower central quotients one step at a time. Corollary~\ref{CORlift} below, which follows from Theorems~\ref{THMrealization} and~\ref{THMhcharacterization}, provides an alternative perspective: The invariant $\overline{\mu}_n(L')$ is the obstruction to lifting the homotopy invariant $h_n(L')$ to the next level of the tower $\{X_n(L)\}$. Note that it concerns links $L$ and $L'$ in the same 3-manifold $M$.

\begin{customcor}{D$_2$}\label{CORlift}
Suppose $L'\subset M$ admits an $n$-basing $\phi$ over $\pi_1(M)$ relative to $L\subset M$ for some $n\geq 2$. 
Then $h_n(L',\phi)$ is defined. 
The following statements are equivalent:
\begin{enumerate}[label=(\arabic*)]
\item The invariant $h_n(L',\phi)$ lifts to a class $f\in[M,X_{n+1}(L)]_0$ such that the image of the composition \[H^1(X_{n+1}(L))\xrightarrow{\cap f_*[M]}H_2(X_{n+1}(L))\xrightarrow{\textnormal{proj}}H_2(X_{n+1}(L))/\textnormal{im}(K_j(\pi/\Gamma_{n+1}))\] contains the image of $H_2(\pi/\Gamma_{n+1})/K_j(\pi/\Gamma_{n+1})$ for all $j<n+1$, where all coefficients are in $\mathbb{Z}[\pi_1(M)]$.
\item The invariant $\overline{\mu}_n(L',\phi)$ vanishes.
\end{enumerate}
As a consequence, the following basing-independent statements are equivalent:
\begin{enumerate}[label=(\arabic*$\,'$)]
\item For some $n$-basing $\phi'$, the invariant $h_n(L',\phi')$ lifts to a class $f\in[M,X_{n+1}(L)]_0$ such that the image of the composition \[H^1(X_{n+1}(L))\xrightarrow{\cap f_*[M]}H_2(X_{n+1}(L))\xrightarrow{\textnormal{proj}}H_2(X_{n+1}(L))/\textnormal{im}(K_j(\pi/\Gamma_{n+1}))\] contains the image of $H_2(\pi/\Gamma_{n+1})/K_j(\pi/\Gamma_{n+1})$ for all $j<n+1$, where all coefficients are in $\mathbb{Z}[\pi_1(M)]$.
\item The invariant $\overline{\mu}_n(L')$ vanishes.
\end{enumerate}
\end{customcor}

\begin{proof}
(1)$\implies$(2). 
Suppose $h_n(L',\phi)$ lifts to such a map $f\in[M,X_{n+1}(L)]_0$. 
The map $f$ automatically satisfies conditions (1) and (2) of Theorem~\ref{THMrealization} since $h_n(L',\phi)$ is a realizable class and therefore satisfies those conditions. 
The additional assumption on the composition $\text{proj}\circ(\cap f_*[M])$ allows us to apply Theorem~\ref{THMrealization} and conclude $f=h_{n+1}(L'',\widetilde{\phi'})$ for some $L''\subset M$ and some $(n+1)$-basing $\widetilde{\phi'}$.
This implies the homology invariant $\theta_n(L',\phi)=(\psi_{n+1,n})_*\theta_{n+1}(L'',\widetilde{\phi'})=\theta_n(L'',\phi')$ vanishes in the cokernel of $\mathcal{R}_{n+1}(L)\to\mathcal{R}_n(L)$, where $\phi'$ is the $n$-basing induced by $\widetilde{\phi'}$. 
Equivalently, $\overline{\mu}_n(L',\phi)$ vanishes. \vspace{1em}

\noindent (2)$\implies$(1). 
Now suppose $\overline{\mu}_n(L',\phi)$ vanishes. 
By Theorem~\ref{THMhcharacterization}, $\phi$ lifts to an $(n+1)$-basing $\widetilde{\phi}$, so $h_{n+1}(L',\widetilde{\phi})$ is defined. 
We have $\psi_{n+1,n}\circ h_{n+1}(L',\widetilde{\phi})=h_n(L',\phi)$, so $h_n(L',\phi)$ lifts to a realizable class for which the desired condition on the composition $\text{proj}\circ(\cap f_*[M])$ is satisfied by Theorem~\ref{THMrealization}. \vspace{1em}

\noindent The equivalence of (1$'$) and (2$'$) follows similarly.
\end{proof}

The assumption on the composition $\text{proj}\circ (\cap f_*[M])$ in Corollary~\ref{CORlift} is the statement of condition (3) from Theorem~\ref{THMrealization} for the class $f$, where $n$ is replaced with $n+1$. Recall from our discussion in Section~\ref{results-htpyreal} that condition (3) of Theorem~\ref{THMrealization} follows from conditions (1) and (2) of Theorem~\ref{THMrealization} in many cases; see also Corollary~\ref{htpychar-thmC-corpi2}. 
Thus, in many cases, Corollary~\ref{CORlift} states that the lower central homotopy invariant $h_n$ lifts if and only if the Milnor invariant $\overline{\mu}_n$ vanishes.


\subsection{Realization of lower central homology invariants} \label{results-homreal}

The following realization theorem for the invariants $\theta_n$ in the case where the group $G$ is the trivial group is analogous to Theorem~\ref{THMrealization} from Section~\ref{results-htpyreal}. We prove Theorem~\ref{THMhrealization} in Section~\ref{homchar}.

\begin{customthm}{E} \label{THMhrealization}
Let the fixed group $G$ be the trivial group. Fix an $m$-component link $L\subset M$, and let $\theta\in H_3(X^G_n(L))$. Then $\theta\in\mathcal{R}^G_n(L)$ if and only if the following three conditions hold:
\begin{enumerate}[label=(\arabic*)]
\item The connecting homomorphism $H_3(X^G_n(L))\xrightarrow{\Delta}H_2(\partial E_L)$ from the Mayer-Vietoris sequence corresponding to the decomposition $X^G_n(L)=M_{p_n}^\times\cup_{\partial E_L}\nu L$ sends $\theta$ to $\sum_i[T_i]$, the sum of the fundamental classes of the tori in $\partial E_L=\sqcup_i T_i$.
\item The cap product \[ \cap\,\theta: tH^2(X^G_n(L))\to tH_1(X^G_n(L)) \] is an isomorphism.
\item The image of the composition \[ H^1(X^G_n(L))\xrightarrow{\cap\,\theta}H_2(X^G_n(L))\xrightarrow{\textnormal{proj}} H_2(X^G_n(L))/\textnormal{im}(K_j(\pi/\pi_n)) \] contains the image of $H_2(\pi/\pi_n)/K_j(\pi/\pi_n)$ for all $j<n$.
\end{enumerate}
\end{customthm}

We briefly compare this theorem to Theorem~\ref{THMrealization}. Conditions (1) and (3) of Theorem~\ref{THMhrealization} are analogous to conditions (1) and (3) of Theorem~\ref{THMrealization}. Condition (3) holds for all $j<n$ if and only if it holds for $j=n-1$. Analogous to Theorem~\ref{THMrealization}, condition (3) of Theorem~\ref{THMhrealization} follows from conditions (1) and (2) in the case where the composition $H_2(E_L;\mathbb{Z})\to H_2(M;\mathbb{Z})\to H_2(X^G_n(L);\mathbb{Z})$ is zero. 

An analogue of condition (2) of Theorem~\ref{THMhrealization} is missing from Theorem~\ref{THMrealization}. In the proof of Theorem~\ref{THMhrealization}, this condition allows us to leverage a 3-dimensional homology surgery result of V. Turaev \cite{Turaev84} to produce a 3-manifold with the desired homological properties along with a link realizing the class $\theta$. Condition (2) in Theorem~\ref{THMhrealization} is necessary to apply Turaev's theorem. This is not needed in Theorem~\ref{THMrealization}, where we begin with a homotopy class of maps from $M$. Proving Theorem~\ref{THMhrealization} in the case of other groups $G$ would rely on extending Turaev's result to the setting of $\mathbb{Z}[G]$-coefficients; this is a topic of future consideration.


\section{\texorpdfstring{$n$}{n}-basings}\label{nbasing}

We require an $n$-basing over $G$ for a link $L'\subset M'$ relative to the fixed link $L\subset M$ in order to define the $n^{\text{th}}$ $G$-lower central homotopy invariant, the $n^{\text{th}}$ $G$-lower central homology invariant, and the $n^{\text{th}}$ $G$-Milnor invariant of $L'$ relative to $L$. Recall from Sections~\ref{results-homotopy} and~\ref{results-overG} that an $n$-basing over $G$ is a pair $(\phi,\phi_\partial)$, where $\phi:\pi'/\Gamma'_n\xrightarrow{\cong}\pi/\Gamma_n$ is an isomorphism over $G$ on $n^{\text{th}}$ $G$-lower central quotients and $\phi_\partial:H_1(\partial E_{L'})\xrightarrow{\cong}H_1(\partial E_L)$ is an \emph{admissible} isomorphism which is \emph{compatible} with $\phi$. We now motivate this definition by showing that a $\mathbb{Z}[G]$-homology cobordism induces an $n$-basing over $G$ for all $n$.

\begin{proposition}\label{nbasing-propconc}
For all $n$, a $\mathbb{Z}[G]$-homology concordance between the links $L\subset M$ and $L'\subset M'$ induces a based homotopy equivalence of pairs $(M_{p'_n}^\times,\partial E_{L'})\xrightarrow{\simeq}(M_{p_n}^\times,\partial E_L)$, restricting to an admissible homeomorphism $\partial E_{L'}\xrightarrow{\cong}\partial E_L$ and extending canonically to a based homotopy equivalence of pairs $(X^G_n(L'),\nu L')\xrightarrow{\simeq}(X^G_n(L),\nu L)$. In the case of a concordance in $M\times[0,1]$, where $M'=M$ and $G=\pi_1(M)$, the based homotopy equivalence $(M_{p'_n}^\times,\partial E_{L'})\xrightarrow{\simeq}(M_{p_n}^\times,\partial E_L)$ is canonical.
\end{proposition}

\begin{proof}
Suppose $C$ is a concordance between $L\subset M$ and $L'\subset M'$ in the $\mathbb{Z}[G]$-homology cobordism $W$. We take the basepoint of the homology cobordism $W$ to be the basepoint of $M$. Choose some path $\gamma$ in $W$ missing $C$ from the basepoint of $M$ to the basepoint of $M'$ whose image in $BG$ under the based map $W\to BG$ is nullhomotopic. An analogue of Corollary~\ref{prelims-corlcs}, with an identical proof, implies that the inclusions $E_{L'}\hookrightarrow E_C\hookleftarrow E_L$ induce isomorphisms $\pi'/\Gamma'_n\xrightarrow{\cong}\pi^C/\Gamma^C_n\xleftarrow{\cong}\pi/\Gamma_n$ for all $n$, where $\pi^C=\pi_1(E_C)$ and $\Gamma^C=\ker(\pi^C\twoheadrightarrow\pi_1(W)\twoheadrightarrow G)$.

We obtain the diagram below, where $X^G_n(C)$ is a standard homotopy pushout analogous to $X^G_n(L)$, and where $\partial\nu C$ denotes the intersection $\partial E_C\cap\nu C$, not the manifold boundary of $\partial\nu C$ which includes $\nu L$ and $\nu L'$. Since $\gamma$ is nullhomotopic in $BG$, this diagram is a diagram over $G$. The homotopy equivalences of pairs on the second row are induced by the group isomorphisms, and we may extend to the homotopy equivalences on the bottom row by gluing in $\nu C$. 

\begin{center}
\begin{tikzcd}[column sep=0em]
(E_{L'}\cup\gamma,\partial E_{L'}) \arrow[rr, hookrightarrow] \arrow[dr, hookrightarrow] \arrow[dd, hookrightarrow] & & (E_C,\partial\nu C) \arrow[dr,hookrightarrow] \arrow[dd, hookrightarrow] \arrow[rr, hookleftarrow] & & (E_L,\partial E_L) \arrow[dd, hookrightarrow] \arrow[dr, hookrightarrow] & \\ 
& (M_{p'_n}^\times\cup\gamma,\partial E_{L'}) \arrow[rr, crossing over, "\simeq" near start] & & (M_{p^C_n}^\times, \partial\nu C) \arrow[from=rr, crossing over, "\simeq"' near start] & & (M_{p_n}^\times, \partial E_L) \arrow[dd,hookrightarrow] \\
(M'\cup\gamma,\nu L') \arrow[dr, hookrightarrow, "\iota_n(L')"'] \arrow[rr, hookrightarrow] & & (W,\nu C) \arrow[dr, hookrightarrow] \arrow[rr,hookleftarrow] & & (M,\nu L) \arrow[dr, hookrightarrow, "\iota_n(L)"] & \\
& (X^G_n(L')\cup\gamma,\nu L') \arrow[from=uu, crossing over, hookrightarrow] \arrow[rr, "\simeq"] & & (X^G_n(C),\nu C) \arrow[from=uu, crossing over, hookrightarrow] \arrow[from=rr, "\simeq"'] & & (X^G_n(L),\nu L)
\end{tikzcd}
\end{center}
In the case where $C$ is a concordance in $M\times[0,1]$ (where $M'=M$ and $G=\pi_1(M)$), we may take $\gamma$ to be $\ast\times[0,1]$, and all homotopy equivalences in the diagram are then canonical.
\end{proof}

We now prove that such homotopy equivalences $(M_{p'_n}^\times,\partial E_{L'})\xrightarrow{\simeq}(M_{p_n}^\times,\partial E_L)$ as in Proposition~\ref{nbasing-propconc} are in 1-1 correspondence with $n$-basings over $G$ for $L'$ relative to $L$. The following Proposition is a refinement of Lemma~\ref{results-homotopy-lemmanbasing} from Section~\ref{results-homotopy}.

\begin{proposition}\label{nbasing-propnbasing}
The link $L'$ admits an $n$-basing $\phi$ over $G$ relative to $L$ if and only if there exists a based homotopy equivalence of pairs $h_\phi:(M_{p_n'}^\times,\partial E_{L'})\xrightarrow{\simeq}(M_{p_n}^\times,\partial E_L)$ over $G$ restricting to an admissible homeomorphism $\partial E_{L'}\xrightarrow{\cong}\partial E_L$. Additionally, any two homotopy equivalences realizing the same $n$-basing are based homotopic as maps of pairs. Furthermore,  such a homotopy equivalence $h$ extends canonically (up to based homotopy) to a based homotopy equivalence of pairs $\overline{h}_\phi:(X^G_n(L'),\nu L')\xrightarrow{\simeq} (X^G_n(L),\nu L)$.
\end{proposition}

\begin{proof}
We begin with the first statement. One direction requires little work: Suppose such a homotopy equivalence of pairs $h$ exists. Then we may take $\phi$ and $\phi_\partial$ to be the isomorphisms $h_*:\pi_1(M_{p_n'}^\times)\xrightarrow{\cong}\pi_1(M_{p_n}^\times)$ and $(h|_{\partial E_{L'}})_*:H_1(\partial E_{L'})\xrightarrow{\cong} H_1(\partial E_L)$, respectively. By assumption, the homeomorphism $\partial E_{L'}\xrightarrow{\cong}\partial E_L$ is admissible, so $\phi_\partial$ is admissible. All that remains is to see that $\phi_\partial$ is compatible with $\phi$ in the sense of Definition~\ref{results-homotopy-defnbasing}. Take any basing $\tau'$ for $L'$. Its image $h(\tau')$ is a basing for $L$ in $M_{p_n}^\times$. We find a corresponding basing in $E_L$ as follows: Take any basing $\tau^0$ for $L$ in $E_L$. The basings $h(\tau')$ and $\tau^0$ differ up to homotopy in $M_{p_n}^\times$ by some wedge of circles which represent elements of $\pi/\Gamma_n$. Adjust $\tau^0$ by lifts of these elements to $\pi$ to obtain a new basing $\tau$ in $E_L$. By construction, the desired diagram commutes.

The other direction is entirely similar to the proof of Lemma~\ref{results-homotopy-lemmanbasing}. We replace the pair $(E_{L'},\partial E_{L'})$ in that proof with the pair $(M_{p'_n}^\times,\partial E_{L'})$. The resulting map \[h_\phi:(M_{p_n'}^\times,\partial E_{L'})\xrightarrow{\simeq}(M_{p_n}^\times,\partial E_L)\] is a homotopy equivalence of pairs because we construct it using the isomorphism $\phi$.

We also follow the proof of Lemma~\ref{results-homotopy-lemmanbasing} to prove the remainder of the proposition. Using similar arguments, the homotopy equivalence $h_\phi$ is unique up to based homotopy of maps of pairs and extends canonically (up to based homotopy) to a map of pairs \[\overline{h}_\phi:(X^G_n(L'),\nu L')\xrightarrow{\simeq} (X^G_n(L),\nu L)\] over $G$. Because $h_\phi$ is a homotopy equivalence of pairs, $\overline{h}_\phi$ is, too.
\end{proof}

By the proof of the above proposition, the map of pairs $(E_{L'},\partial E_{L'})\to (M_{p_n}^\times,\partial E_L)$ and its extension constructed in the proof of Lemma~\ref{results-homotopy-lemmanbasing} factor through the homotopy equivalences $h_\phi$ and $\overline{h}_\phi$, respectively. Therefore, the composition $\overline{h}_\phi\circ\iota_n(L')$ is precisely the $n^{\text{th}}$ $h$-invariant $h_n(L',\phi)$.

\begin{corollary}\label{nbasing-cornbasing}
A $\mathbb{Z}[G]$-homology concordance between the links $L\subset M$ and $L'\subset M'$ induces an $n$-basing over $G$ for $L'$ relative to $L$ for all $n$. In the case of a concordance in $M\times[0,1]$, where $M'=M$ and $G=\pi_1(M)$, the induced $n$-basing is canonical.
\end{corollary}


\subsection{On the choice of \texorpdfstring{$n$}{n}-basing}\label{nbasing-selfbasing}

The invariants $h_n$, $\theta^G_n$, and $\overline{\mu}^G_n$ defined in Sections~\ref{results-homotopy},~\ref{results-homology}, and~\ref{results-determination} at first depend on the choice of $n$-basing for $L'$. We therefore study how different choices of $n$-basing are related. Suppose $(\phi,\phi_\partial)$ and $(\phi',\phi'_\partial)$ are two $n$-basings for $L'$ relative to $L$. We will show \[ (\psi,\psi_\partial)=(\phi'\circ\phi^{-1},\phi'_\partial\circ\phi_\partial^{-1})\in\text{Aut}(\pi/\Gamma_n)\times\text{Aut}(H_1(\partial E_L)) \] constitutes an $n$-basing over $G$ for $L$ relative to itself. We will call such an $n$-basing a \emph{self-$n$-basing} or \emph{self-basing}. By Proposition~\ref{nbasing-propnbasing}, to each self-basing we may associate a canonical homotopy class of homotopy self-equivalences of the pair $(X^G_n(L),\nu L)$ over $G$. In Proposition~\ref{nbasing-selfbasing-propselfbasing}, we show that the collection of all self-$n$-basings $(\psi,\psi_\partial)$, denoted $\text{Aut}(\pi/\Gamma_n,\partial)$, forms a subgroup of  $\text{Aut}(\pi/\Gamma_n)\times\text{Aut}(H_1(\partial E_L))$ isomorphic to the group of homotopy classes of homotopy self-equivalences of the pair $(M_{p_n}^\times,\partial E_L)$ over $G$ which restrict to admissible homeomorphisms $\partial E_L\to\partial E_L$. In fact, because every such $\psi$ is an isomorphism over $G$ and every such $\psi_\partial$ is admissible, we will have that \[\text{Aut}(\pi/\Gamma_n,\partial)\leq\text{Aut}_G(\pi/\Gamma_n)\times\mathbb{Z}^m\leq\text{Aut}(\pi/\Gamma_n)\times\text{Aut}(H_1(\partial E_L)),\] where $\text{Aut}_G$ denotes automorphisms over $G$, and where $\mathbb{Z}^m\leq\text{Aut}(H_1(\partial E_L))$ is the subgroup of automorphisms of $H_1(\partial E_L)$ which fix the meridians of $L$.

\begin{lemma}\label{nbasing-selfbasing-lemmainverse}
Suppose $(\phi,\phi_\partial)$ is an $n$-basing over $G$ for $L'$ relative to $L$. Then $(\phi^{-1},\phi_\partial^{-1})$ is an $n$-basing over $G$ for $L$ relative to $L'$.
\end{lemma}

\begin{proof}
Since $\phi_\partial$ is admissible, so is $\phi_\partial^{-1}$. We next see that $\phi_\partial^{-1}$ is compatible with $\phi^{-1}$. Let $\tau$ be a basing for $L$. Choose any basing $\tau'_0$ for $L'$. Since $\phi_\partial$ is compatible with $\phi$, there exists a basing $\tau_0$ for $L$ such that the following square commutes:
\begin{center}
\begin{tikzcd}
\pi_1(\partial E_{L'}\cup\tau'_0) \arrow[r, "\phi_{\partial*}"', "\cong"] \arrow[dd] & \pi_1(\partial E_L\cup\tau_0) \arrow[dd] \\
& \\
\pi'/\Gamma'_n \arrow[r, "\phi"', "\cong"] & \pi/\Gamma_n.
\end{tikzcd}
\end{center}
Since $\tau$ and $\tau_0$ are both basings for $L$, they differ up to homotopy by some wedge of circles in $E_L$. Consider the image of this wedge of circles in $\pi/\Gamma_n$. Mapping this to $\pi'/\Gamma_n'$ via $\phi^{-1}$ and choosing any lift to the group $\pi'$ yields a wedge of circles in $E_{L'}$. Adjust $\tau'_0$ by this wedge of circles to obtain a new basing $\tau'$ for $L'$. Using $\phi_\partial^{-1}$, define an isomorphism $\pi_1(\partial E_L\cup\tau)\xrightarrow{\cong}\pi_1(\partial E_{L'}\cup\tau')$ which fits into the diagram
\begin{center}
\begin{tikzcd}
\pi_1(\partial E_L\cup\tau) \arrow[r, "\phi_{\partial*}^{-1}"', "\cong"] \arrow[dd] & \pi_1(\partial E_{L'}\cup\tau') \arrow[dd] \\
& \\
\pi/\Gamma_n \arrow[r, "\phi^{-1}"', "\cong"] & \pi'/\Gamma'_n.
\end{tikzcd}
\end{center}
Thus, $\phi_\partial^{-1}$ is compatible with $\phi^{-1}$, so $(\phi^{-1},\phi_\partial^{-1})$ is an $n$-basing over $G$ for $L$ relative to $L'$.
\end{proof}

\begin{lemma}\label{nbasing-selfbasing-lemmacompose}
Suppose $(\phi',\phi'_\partial)$ is an $n$-basing over $G$ for $L''$ relative to $L'$ and $(\phi,\phi_\partial)$ is an $n$-basing over $G$ for $L'$ relative to $L$. Then the composition $(\phi\circ\phi',\phi_\partial\circ\phi'_\partial)$ is an $n$-basing over $G$ for $L''$ relative to $L$.
\end{lemma}

\begin{proof}
Since $\phi_\partial$ and $\phi'_\partial$ are both admissible, so is the composition $\phi_\partial\circ\phi'_\partial$. Let $\tau''$ be a basing for $L''$. Since $\phi'_\partial$ is compatible with $\phi'$, we obtain the desired basing $\tau'$ for $L'$. This in turn yields the desired basing $\tau$ for $L$ since $\phi_\partial$ is compatible with $\phi$. We have the following diagram:
\begin{center}
\begin{tikzcd}
\pi_1(\partial E_{L''}\cup\tau'') \arrow[r, "\phi'_{\partial*}"', "\cong"] \arrow[dd] & \pi_1(\partial E_{L'}\cup\tau') \arrow[r, "\phi_{\partial*}"', "\cong"] \arrow[dd] & \pi_1(\partial E_L\cup\tau) \arrow[dd] \\
& & \\
\pi''/\Gamma''_n \arrow[r, "\phi'"', "\cong"] & \pi'/\Gamma'_n \arrow[r, "\phi"', "\cong"] & \pi/\Gamma_n.
\end{tikzcd}
\end{center}
Thus, $\phi_\partial\circ\phi'_\partial$ is compatible with $\phi\circ\phi'$, so $(\phi\circ\phi',\phi_\partial\circ\phi'_\partial)$ is an $n$-basing over $G$ for $L''$ relative to $L$.
\end{proof}

The preceding Lemmas~\ref{nbasing-selfbasing-lemmainverse} and~\ref{nbasing-selfbasing-lemmacompose} reveal that the set of self-$n$-basings is a group. The following proposition reveals some of the structure of this group.

\begin{proposition}\label{nbasing-selfbasing-propselfbasing}
The collection $\textnormal{Aut}(\pi/\Gamma_n,\partial)$ of all self-$n$-basings over $G$ of $L$ forms a subgroup of $\textnormal{Aut}(\pi/\Gamma_n)\times\textnormal{Aut}(H_1(\partial E_L))$ isomorphic to the group of homotopy classes of homotopy self-equivalences of the pair $(M_{p_n}^\times,\partial E_L)$ over $G$ which restrict to admissible homeomorphisms $\partial E_L\xrightarrow{\cong}\partial E_L$. Furthermore, every self-basing gives rise to a canonical homotopy class of homotopy self-equivalences of the pair $(X_n(L),\nu L)$ over $G$.
\end{proposition}

\begin{proof}
Any self-$n$-basing $(\psi,\psi_\partial)$ is an element of $\textnormal{Aut}(\pi/\Gamma_n)\times\textnormal{Aut}(H_1(\partial E_L))$. Taking $L''=L'=L$ in Lemmas~\ref{nbasing-selfbasing-lemmainverse} and~\ref{nbasing-selfbasing-lemmacompose}, we see that $\textnormal{Aut}(\pi/\Gamma_n,\partial)$ is closed under inverses and the group operation. Thus, $\textnormal{Aut}(\pi/\Gamma_n,\partial)\leq\textnormal{Aut}(\pi/\Gamma_n)\times\textnormal{Aut}(H_1(\partial E_L))$. By Proposition~\ref{nbasing-propnbasing}, a self-basing $(\psi,\psi_\partial)$ over $G$ is equivalent to a homotopy class of homotopy self-equivalences of the pair $(M_{p_n}^\times,\partial E_L)$ over $G$ which restricts to an admissible homeomorphism $\partial E_L\xrightarrow{\cong}\partial E_L$. Such a homotopy self-equivalence extends over $G$ to a canonical homotopy class of homotopy self-equivalences of the pair $(X_n(L),\nu L)$, again by Proposition~\ref{nbasing-propnbasing}.
\end{proof}

By Proposition~\ref{nbasing-selfbasing-propselfbasing}, $\text{Aut}(\pi/\Gamma_n,\partial)$ acts (on the left) on both $[M,X_n(L)]_0$ and $H_3(X_n(L))$ by post-composition with homotopy self-equivalences of $X_n(L)$, so that the orbit space of each of these sets modulo the group action is well-defined.


\section{Concordance invariance}\label{inv}

In this section, we prove Theorems~\ref{THMinvariance} and~\ref{THMhinvariance}. See Sections~\ref{results-homotopy} and~\ref{results-homology}, respectively, for the statements of these theorems. Theorem~\ref{THMinvariance} asserts the invariance of $h_n$ (hence $\theta_n$ and $\overline{\mu}_n$) under concordance. Theorem~\ref{THMhinvariance}, whose proof is similar to that of Theorem~\ref{THMinvariance}, asserts the invariance of $\theta_n$ (hence $\overline{\mu}_n$) under $\mathbb{Z}[G]$-homology concordance.


\subsection{Proof of Theorem~\ref{THMinvariance}} \label{proofs-inv-A}

\begin{proof}
Fix $L\subset M$, and assume $L',L''\subset M$ are concordant.

Statement (1), which asserts the existence of a canonical $n$-basing ($\phi'$,$\phi'_\partial)$ for $L''$ relative to $L'$ for each $n$, follows directly from Corollary~\ref{nbasing-cornbasing}.

Next, we prove (2), which states that for each $n$ an $n$-basing for $L'$ relative to $L$ exists if and only if an $n$-basing for $L''$ relative to $L$ exists, and that the invariant $h_n(L')$ is defined if and only if $h_n(L'')$ is defined. Suppose $(\phi,\phi_\partial)$ is an $n$-basing for $L'$ relative to $L$. By Lemma~\ref{nbasing-selfbasing-lemmacompose}, the composition $(\phi\circ\phi',\phi_\partial\circ\phi'_\partial)$ constitutes an $n$-basing for $L''$ relative to $L$. Conversely, suppose $(\phi,\phi_\partial)$ is an $n$-basing for $L''$ relative to $L$. By Lemma~\ref{nbasing-selfbasing-lemmainverse}, $((\phi')^{-1},(\phi'_\partial)^{-1})$ is an $n$-basing for $L'$ relative to $L''$, so that $(\phi\circ(\phi')^{-1},\phi_\partial\circ(\phi'_\partial)^{-1})$ is an $n$-basing for $L'$ relative to $L$, again by Lemma~\ref{nbasing-selfbasing-lemmacompose}. This proves (2), since the existence of an $n$-basing relative to $L$ is all that is needed to define the invariant $h_n$.

We now prove (3), which gives the precise relationship between the $h$-invariants of $L'$ and those of $L''$. Again assume $(\phi,\phi_\partial)$ is an $n$-basing for $L'$ relative to $L$. By Proposition~\ref{nbasing-propnbasing}, we obtain a homotopy equivalence of pairs $h_\phi:(M_{p'_n}^\times,\partial E_{L'})\xrightarrow{\simeq}(M_{p_n}^\times,\partial E_L)$ restricting to an admissible homeomorphism $\partial E_{L'}\xrightarrow{\cong}\partial E_L$ and extending to a homotopy equivalence of pairs $\overline{h}_\phi:(X_n(L'),\nu L')\xrightarrow{\simeq}(X_n(L),\nu L)$, all over $\pi_1(M)$. Since $L'$ and $L''$ are concordant, from Proposition~\ref{nbasing-propconc} we obtain a canonical homotopy equivalence of pairs $h':(M_{p''_n}^\times,\partial E_{L''})\xrightarrow{\simeq}(M_{p'_n}^\times,\partial E_{L'})$ restricting to an admissible homeomorphism $\partial E_{L''}\xrightarrow{\cong}\partial E_{L'}$ and extending canonically to a homotopy equivalence of pairs $\overline{h'}:(X_n(L''),\nu L'')\xrightarrow{\simeq}(X_n(L'),\nu L')$. In fact, using the language of Proposition~\ref{nbasing-propnbasing}, $h'=h_{\phi'}$ and $\overline{h'}=\overline{h}_{\phi'}$, where $\phi'$ is the canonical $n$-basing from statement (1). We have the following commutative diagram over $\pi_1(M)$ from the proof of Proposition~\ref{nbasing-propconc} (assume the concordance $C$ misses $\ast\times[0,1]$, and recall that $\partial\nu C$ denotes the intersection $\partial E_C\cap\nu C$):
\begin{center}
\begin{tikzcd}[column sep=0em]
(E_{L''},\partial E_{L''}) \arrow[rr, hookrightarrow] \arrow[dr, hookrightarrow] \arrow[dd, hookrightarrow] & & (E_C,\partial\nu C) \arrow[dr,hookrightarrow] \arrow[dd, hookrightarrow] \arrow[rr, hookleftarrow] & & (E_{L'},\partial E_{L'}) \arrow[dd, hookrightarrow] \arrow[dr, hookrightarrow] & \\ 
& (M_{p''_n}^\times,\partial E_{L''}) \arrow[rr, crossing over, "\simeq" near start] & & (M_{p^C_n}^\times, \partial\nu C) \arrow[from=rr, crossing over, "\simeq"' near start] & & (M_{p'_n}^\times,\partial E_{L'}) \arrow[dd,hookrightarrow] \\
(M,\nu L'') \arrow[dr, hookrightarrow, "\iota_n(L'')"'] \arrow[rr, hookrightarrow] & & (M\times[0,1],\nu C) \arrow[dr, hookrightarrow] \arrow[rr,hookleftarrow] & & (M,\nu L') \arrow[dr, hookrightarrow, "\iota_n(L')"] & \\
& (X_n(L''),\nu L'') \arrow[from=uu, crossing over, hookrightarrow] \arrow[rr, "\simeq"] & & (X_n(C),\nu C) \arrow[from=uu, crossing over, hookrightarrow] \arrow[from=rr, "\simeq"'] & & (X_n(L'),\nu L')
\end{tikzcd}
\end{center}

Thus, the composite homotopy equivalence $h_\phi\circ h_{\phi'}:(M_{p_n''}^\times,\partial E_{L''})\xrightarrow{\simeq}(M_{p_n}^\times,\partial E_L)$ restricts to an admissible homeomorphism $\partial E_{L''}\xrightarrow{\cong}\partial E_L$ and induces the $n$-basing $(\phi\circ\phi',\phi_\partial\circ\phi_\partial')$. The map $h_\phi\circ h_{\phi'}$ extends to the homotopy equivalence of pairs $\overline{h}_\phi\circ\overline{h}_{\phi'}$. Thus, \[h_n(L'',\phi\circ\phi')=\overline{h}_\phi\circ\overline{h}_{\phi'}\circ\iota_n(L'')=\overline{h}_\phi\circ\iota_n(L')=h_n(L',\phi),\] where the first and last equalities follow by definition, and the middle equality follows from the commutativity of the above diagram.

Finally, we prove (4), which states that $h_n(L')=h_n(L'')$ when defined. Suppose $(\phi,\phi_\partial)$ and $(\phi'',\phi''_\partial)$ are $n$-basings for $L'$ and $L''$, respectively, relative to $L$. Then $(\phi\circ\phi',\phi_\partial\circ\phi'_\partial)$ is also an $n$-basing for $L''$ relative to $L$, and by (3) we have $h_n(L',\phi)=h_n(L'',\phi\circ\phi')$. Observe that $(\psi,\psi_\partial):=(\phi''\circ(\phi\circ\phi')^{-1},\phi''_\partial\circ(\phi_\partial\circ\phi'_\partial)^{-1})$
is a self-$n$-basing for $L$ by Lemmas~\ref{nbasing-selfbasing-lemmainverse} and~\ref{nbasing-selfbasing-lemmacompose}. Then $h_n(L'',\phi'')=(\psi,\psi_\partial)\cdot h_n(L'',\phi\circ\phi')=(\psi,\psi_\partial)\cdot h_n(L',\phi)$, so $h_n(L')=h_n(L'')\in[M,X_n(L)]_0/\text{Aut}(\pi/\Gamma_n,\partial)$.
\end{proof}

Theorem~\ref{THMinvariance} implies that the invariants $\theta_n$ and $\overline{\mu}_n$ defined in this context are also invariants of concordance.


\subsection{Proof of Theorem \texorpdfstring{\ref{THMhinvariance}}{A'}} \label{proofs-inv-A'}

\begin{proof}
We proceed in a manner similar to the proof of Theorem~\ref{THMinvariance}. Fix $L\subset (M,\varphi)$, and suppose $L'\subset(M',\varphi')$ and $L''\subset(M'',\varphi'')$ are $\mathbb{Z}[G]$-homology concordant. Statement (1), which asserts the existence of an $n$-basing over $G$ for $L''$ relative to $L'$ for all $n$, is the statement of Corollary~\ref{nbasing-cornbasing}. The proofs of statements (2) and (4) are entirely similar to the proofs of the analogous statements (2) and (4) of Theorem~\ref{THMinvariance}. Together, these state that $\theta_n(L')$ is defined if and only if $\theta_n(L'')$ is defined, in which case $\theta_n(L')=\theta_n(L'')$.

The proof of statement (3), which gives the precise relationship between the $\theta$-invariants for $L'$ and the $\theta$-invariants for $L''$, is nearly the same as the proof of (3) of Theorem~\ref{THMinvariance}. Using the notation from the proof of Theorem~\ref{THMinvariance}, we again have $h_n(L'',\phi\circ\phi')=\overline{h}_\phi\circ\overline{h'}\circ\iota_n(L'')$ and $h_n(L',\phi)=\overline{h}_\phi\circ\iota_n(L')$. By the commutativity of the diagram from the proof of Proposition~\ref{nbasing-propconc} (which now involves a $\mathbb{Z}[G]$-homology cobordism $W$ instead of $M\times[0,1]$ as seen in the proof of Theorem~\ref{THMinvariance}), we have
\begin{align*}
\theta_n(L'',\phi\circ\phi')&=h_n(L'',\phi\circ\phi')_*[M'']=(\overline{h}_\phi)_*\overline{h'}_*\iota_n(L'')_*[M'']\\
&=(\overline{h}_\phi)_*\iota_n(L')_*[M']=h_n(L',\phi)_*[M']=\theta_n(L',\phi),
\end{align*}
which is the claimed equality in statement (3). 
\end{proof}

Just as with Theorem~\ref{THMinvariance}, Theorem~\ref{THMhinvariance} immediately implies that the $\overline{\mu}_n$ are also invariants of $\mathbb{Z}[G]$-homology concordance.


\section{Characterization and realization for \texorpdfstring{$h_n$}{hn}}\label{htpychar}

In this section, we prove characterization and realization properties of the lower central homotopy invariants $h_n$. We begin in Section~\ref{htpychar-ncob} with a more detailed discussion of $n$-cobordism of links, including the proof of Proposition~\ref{results-ncob-propncob} which states that $n$-cobordisms induce $n$-basings. In Section~\ref{htpychar-thmB}, we prove Theorem~\ref{THMcharacterization}, which asserts that the invariants $h_n$ are complete invariants of $n$-cobordism. We then prove Theorem~\ref{THMrealization}, which determines exactly which elements of $[M,X_n(L)]_0$ are realized as the $h$-invariants of links in $M$, in Section~\ref{htpychar-thmC}. For all of Section~\ref{htpychar}, all of our links live in $M$ and the fixed group $G$ is $\pi_1(M)$. Throughout this section, given a properly embedded surface $\Sigma\subset M\times[0,1]$, we will use the notation $\partial\nu\Sigma$ to denote the intersection $\partial E_\Sigma\cap\nu\Sigma$, not the manifold boundary of $\nu\Sigma$ which also includes $(M\times\{0,1\})\cap\nu\Sigma$. Some of the proofs in this section are similar to proofs seen in \cite{Heck}.


\subsection{\texorpdfstring{$n$}{n}-cobordism of links} \label{htpychar-ncob}

Recall Definition~\ref{results-ncob-defncob}: The links $L$ and $L'$ are \emph{$n$-cobordant} if there exists a properly embedded orientable $m$-component surface $\Sigma=\sqcup_i\Sigma_i\subset M\times [0,1]$ such that for each $i$:
\begin{enumerate}[label=(\arabic*)]
\item $\Sigma_i\cap (M\times\{0\})=L_i$ and $\Sigma_i\cap (M\times\{1\})=L'_i$.
\item For any basing of $\Sigma_i$ determined by a basing of $L_i$, the images of $\pi_1(L_i)$ and $\pi_1(\Sigma_i)$ in $\pi_1(E_\Sigma)$ agree modulo the $n^\text{th}$ $\pi_1(M)$-lower central subgroup, that is,  \[\text{im}\Big(\pi_1(L_i)\to\pi\to\pi^\Sigma\twoheadrightarrow\pi^\Sigma/\Gamma^\Sigma_n\Big)=\text{im}\Big(\pi_1(\Sigma_i)\to\pi^\Sigma\twoheadrightarrow\pi^\Sigma/\Gamma^\Sigma_n\Big),\] where $\pi^\Sigma=\pi_1(E_\Sigma)$ and $\Gamma^\Sigma=\ker\big(\pi^\Sigma\twoheadrightarrow\pi_1(M)\big)$, and where the inclusions $\Sigma_i\hookrightarrow E_\Sigma$ and $L_i\hookrightarrow E_L$ are induced by a trivialization of $\nu\Sigma$.
\end{enumerate}
We prove some properties of $n$-cobordisms culminating with the fact that an $n$-cobordism induces an $n$-basing for $L'$ relative to $L$.

\begin{lemma}\label{htpychar-ncob-lemmancob}
Suppose $\Sigma$ is an $n$-cobordism between $L$ and $L'$. Then the image of the composition \[H_2(\partial E_L;\mathbb{Z}[\pi_1(M)])\to H_2(\partial\nu\Sigma;\mathbb{Z}[\pi_1(M)])\to H_2(\pi^\Sigma/\Gamma^\Sigma_n;\mathbb{Z}[\pi_1(M)])\] equals the image of $H_2(\partial\nu\Sigma;\mathbb{Z}[\pi_1(M)])\to H_2(\pi^\Sigma/\Gamma^\Sigma_n;\mathbb{Z}[\pi_1(M)])$. The same holds for $L'$.
\end{lemma}

\begin{proof}
In this proof, we denote the $\pi_1(M)$-cover of a space $X$ by $\overline{X}$. Observe that $\partial\nu \Sigma\cong\Sigma\times S^1$ and $\overline{\partial\nu\Sigma}\cong\overline{\Sigma}\times S^1$, where $\Sigma$ and $\overline{\Sigma}$ denote parallels of the cores of $\nu\Sigma$ and $\overline{\nu\Sigma}$, respectively. Since each component of $\Sigma$ is a surface with boundary, it follows that a collection of generators of $H_2(\partial\nu\Sigma;\mathbb{Z}[\pi_1(M)])$ may be represented by embedded tori in $\overline{\partial\nu\Sigma}$. Each such torus has a symplectic basis consisting of a meridian to a component of $\overline{\Sigma}$, which we refer to as a meridian for the torus, and a curve lying on that component, which we call a longitude for the torus. Therefore, the image of any such generator of $H_2(\partial\nu\Sigma;\mathbb{Z}[\pi_1(M)])$ in $H_2(\pi^\Sigma/\Gamma^\Sigma_n;\mathbb{Z}[\pi_1(M)])$ is represented by a map $T^2\to \overline{\partial\nu\Sigma}\to K(\Gamma^\Sigma/\Gamma^\Sigma_n,1)$. It suffices to find a solution, up to homotopy, to the lifting problem
\begin{center}
\begin{tikzcd}
& \overline{\partial E_L} \arrow[d] \\
T^2 \arrow[r] \arrow[ur,dashed,"?\exists"] & K(\Gamma^\Sigma/\Gamma^\Sigma_n,1).
\end{tikzcd}
\end{center} 

Base $L$ (a parallel of $L$ in $\partial E_L$) using some basing $\tau$. This induces an equivariant basing $\overline{\tau}$ for $\overline{L}$ (a parallel of the core of $\nu\overline{L}$ in $\overline{\partial E_L}$), hence for $\overline{\Sigma}$. Extend this basing in the cover along the appropriate component of $\overline{\Sigma}$ to give a basing of the torus. The (based) meridian of $T^2$ is homotopic in $\overline{\partial\nu\Sigma}$ to a meridian for $\overline{L}$, so first define the image of the meridian of $T^2$ to be this meridian of $\overline{L}$. Next, note that the (based) longitude of $T^2$ represents an element of $\pi_1(\overline{\Sigma}\cup\overline{\tau})\leq\pi_1(\Sigma\cup\tau)$. Since $\Sigma$ is an $n$-cobordism, the image of this longitude in $\pi^\Sigma/\Gamma^\Sigma_n$ agrees with the image of some element $x\in\pi_1(L_i\cup\tau_i)$ for some $i$. But the image of the longitude is also an element of $\Gamma^\Sigma$. Since $E_L\hookrightarrow E_\Sigma$ is a map over $\pi_1(M)$, $x\in\Gamma\cap\pi_1(L_i\cup\tau_i)=\pi_1(\overline{L_i}\cup\overline{\tau_i})$. Map the longitude of $T^2$ to a curve representing $x$. This gives a lifting of the 1-skeleton of $T^2$ up to homotopy. 

Finally, we may choose any lift of the 2-cell of $T^2$. The resulting composite $T^2\to\overline{\partial E_L}\to K(\Gamma^\Sigma/\Gamma^\Sigma_n,1)$ is homotopic to the original map $T^2\to \overline{\partial\nu\Sigma}\to K(\Gamma^\Sigma/\Gamma^\Sigma_n,1)$ since the maps agree on $\pi_1$. Thus, we have produced a map $T^2\to\overline{\partial E_L}$ that represents an element of $H_2(\partial E_L;\mathbb{Z}[\pi_1(M)])$ which maps onto the image of the chosen generator of $H_2(\partial\nu\Sigma;\mathbb{Z}[\pi_1(M)])$ in $H_2(\pi^\Sigma/\Gamma^\Sigma_n;\mathbb{Z}[\pi_1(M)])$.

The same proof holds for $L'$. Recall from Section~\ref{results-ncob} that the condition \[\text{im}\Big(\pi_1(L_i)\to\pi\to\pi^\Sigma\twoheadrightarrow\pi^\Sigma/\Gamma^\Sigma_n\Big)=\text{im}\Big(\pi_1(\Sigma_i)\to\pi^\Sigma\twoheadrightarrow\pi^\Sigma/\Gamma^\Sigma_n\Big)\] implies the analogous condition for $L'$.
\end{proof}

We are now prepared to prove that $n$-cobordisms induce $n$-basings.

\begin{customprop}{\ref{results-ncob-propncob}}
Suppose $\Sigma\subset M\times[0,1]$ is an $n$-cobordism between $L$ and $L'$. Then $\Sigma$ induces a canonical $n$-basing $(\phi,\phi_\partial)$ for $L'$ relative to $L$ and a canonical isomorphism $\pi'/\Gamma'_{n+1}\xrightarrow{\cong}\pi/\Gamma_{n+1}$ over $\pi_1(M)$.
\end{customprop}

\begin{proof} We first obtain the desired isomorphisms on $\pi_1(M)$-lower central quotients. \vspace{1em}

\noindent Claim. The inclusion $E_L\hookrightarrow E_\Sigma$ induces isomorphisms $\pi/\Gamma_k\xrightarrow{\cong}\pi^\Sigma/\Gamma^\Sigma_k$ for all $k\leq n+1$. \vspace{1em}

The claim is true for $k=1$ because $E_L\hookrightarrow E_\Sigma$ is a map over $\pi_1(M)$. We use an induction argument along with the Five Lemma applied to the following diagram with exact rows:
\begin{center}
\begin{tikzcd}
1 \arrow[r] & \Gamma_{k-1}/\Gamma_{k} \arrow[r,hookrightarrow] \arrow[d] & \pi/\Gamma_{k} \arrow[r,twoheadrightarrow] \arrow[d] & \pi/\Gamma_{k-1} \arrow[r] \arrow[d] & 1 \\
1 \arrow[r] & \Gamma^\Sigma_{k-1}/\Gamma^\Sigma_{k} \arrow[r,hookrightarrow] & \pi^\Sigma/\Gamma^\Sigma_{k} \arrow[r,twoheadrightarrow] & \pi^\Sigma/\Gamma^\Sigma_{k-1} \arrow[r] & 1.
\end{tikzcd}
\end{center}

\noindent For the $k=2$ base case, it suffices to prove $H_1(E_L;\mathbb{Z}[\pi_1(M)])\to H_1(E_\Sigma;\mathbb{Z}[\pi_1(M)])$ induced by inclusion is an isomorphism since $H_1(E_L;\mathbb{Z}[\pi_1(M)])=\Gamma/\Gamma_2$, $H_1(E_\Sigma;\mathbb{Z}[\pi_1(M)])=\Gamma^\Sigma/\Gamma^\Sigma_2$, and $\pi/\Gamma=\pi^\Sigma/\Gamma^\Sigma=\pi_1(M)$.

Consider the following diagram of exact sequences of pairs. All coefficients are in $\mathbb{Z}[\pi_1(M)]$, which we suppress.
\begin{center}
\begin{tikzcd}
H_2(M) \arrow[r] \arrow[d,equal] & H_2(M,E_L) \arrow[r, twoheadrightarrow] \arrow[d] & H_1(E_L) \arrow[d] \arrow[r] & 0 \\
H_2(M\times[0,1]) \arrow[r] & H_2(M\times[0,1],E_\Sigma) \arrow[r, twoheadrightarrow] & H_1(E_\Sigma) \arrow[r] & 0
\end{tikzcd}
\end{center}
It suffices to prove that $H_2(M,E_L;\mathbb{Z}[\pi_1(M)])\to H_2(M\times[0,1],E_\Sigma;\mathbb{Z}[\pi_1(M)])$ is an isomorphism. By excision, \[H_2(M,E_L;\mathbb{Z}[\pi_1(M)])\cong H_2(\nu L,\partial E_L;\mathbb{Z}[\pi_1(M)])\] and \[H_2(M\times[0,1],E_\Sigma;\mathbb{Z}[\pi_1(M)])\cong H_2(\nu \Sigma,\partial\nu\Sigma;\mathbb{Z}[\pi_1(M)]).\]

Now consider the diagram of exact sequences of pairs below. Note that $H_2(\nu \Sigma;\mathbb{Z}[\pi_1(M)])=0$ since the components of $\Sigma$ are surfaces with boundary. Again, all coefficients are in $\mathbb{Z}[\pi_1(M)]$.
\begin{center}
\begin{tikzcd}
0 \arrow[r] & H_2(\nu L,\partial E_L) \arrow[r] \arrow[d] & H_1(\partial E_L) \arrow[r,twoheadrightarrow] \arrow[d] & H_1(\nu L) \arrow[d] \arrow[r] & 0 \\
0 \arrow[r] & H_2(\nu \Sigma, \partial\nu\Sigma) \arrow[r] & H_1(\partial\nu\Sigma) \arrow[r,twoheadrightarrow] & H_1(\nu \Sigma) \arrow[r] & 0
\end{tikzcd}
\end{center}
Even with coefficients in $\mathbb{Z}[\pi_1(M)]$, we have splittings \[H_1(\partial E_L;\mathbb{Z}[\pi_1(M)])\cong H_1(L;\mathbb{Z}[\pi_1(M)])\oplus H_1(\sqcup m_i;\mathbb{Z}[\pi_1(M)])\] and \[H_1(\partial\nu\Sigma;\mathbb{Z}[\pi_1(M)])\cong H_1(\Sigma;\mathbb{Z}[\pi_1(M)])\oplus H_1(\sqcup m_i;\mathbb{Z}[\pi_1(M)]),\] where the $m_i$ denote both meridians of $L$ and meridians of $\Sigma$. The homomorphism \[H_1(\partial E_L;\mathbb{Z}[\pi_1(M)])\to H_1(\partial\nu\Sigma;\mathbb{Z}[\pi_1(M)])\] may be written as a direct sum with respect to these splittings. Furthermore, the kernels of the rightmost maps in the diagram are precisely $H_1(\sqcup m_i;\mathbb{Z}[\pi_1(M)])$. Thus, we have (with coefficients in $\mathbb{Z}[\pi_1(M)]$)
\begin{align*}
H_2(\nu L,\partial E_L)&\cong\ker(H_1(\partial E_L)\to H_1(\nu L))\cong H_1(\sqcup m_i) \\
&\cong\ker(H_1(\partial\nu\Sigma)\to H_1(\nu \Sigma))\cong H_2(\nu \Sigma,\partial\nu\Sigma).
\end{align*}
Therefore, $\Gamma/\Gamma_2\to\Gamma^\Sigma/\Gamma^\Sigma_2$ induced by inclusion is an isomorphism. By the Five Lemma, $\pi/\Gamma_2\to\pi^\Sigma/\Gamma^\Sigma_2$ is an isomorphism. This proves the base case $k=2$.

Now suppose $\pi/\Gamma_k\xrightarrow{\cong}\pi^\Sigma/\Gamma^\Sigma_k$ is an isomorphism for some $2\leq k\leq n$. We will obtain an isomorphism $\pi/\Gamma_{k+1}\xrightarrow{\cong}\pi^\Sigma/\Gamma^\Sigma_{k+1}$. We invoke the 5-term natural exact sequence of Stallings \cite{Stallings}; see the proof of the \hyperref[THMsd]{Stallings-Dwyer Theorem} in \cite{ChaOrr}. 
We have the following diagram, where all homology groups have coefficients in $\mathbb{Z}[\pi_1(M)]$:
\begin{center}
\begin{tikzcd}
H_2(E_L) \arrow[r, twoheadrightarrow] \arrow[d] & H_2(\pi) \arrow[r] \arrow[d] & H_2(\pi/\Gamma_k) \arrow[r,twoheadrightarrow] \arrow[d,"\cong"] & \Gamma_k/\Gamma_{k+1} \arrow[d] \arrow[r] & 0 \\
H_2(E_\Sigma) \arrow[r,twoheadrightarrow] & H_2(\pi^\Sigma) \arrow[r] & H_2(\pi^\Sigma/\Gamma^\Sigma_k) \arrow[r,twoheadrightarrow] & \Gamma^\Sigma_k/\Gamma^\Sigma_{k+1} \arrow[r] & 0.
\end{tikzcd}
\end{center}
Note that $H_2(\pi;\mathbb{Z}[\pi_1(M)])$ is the second homology of the $\pi_1(M)$-cover of a $K(\pi,1)$, which is a $K(\Gamma,1)$. Therefore, $H_2(\pi;\mathbb{Z}[\pi_1(M)])=H_2(\Gamma;\mathbb{Z})$, and similarly for the other group homology terms appearing above. The surjectivity of the first map in each row follows from Hopf's Theorem, which gives an exact sequence \[\pi_2(X)\to H_2(X)\twoheadrightarrow H_2(\pi_1(X))\to 0\] for any CW complex $X$; see Section II.5 of \cite{Brown}. The rest of each row is exact by Stallings \cite{Stallings}. (Each Stallings 5-term exact sequence we use has an isomorphism between the final two terms.) The third vertical arrow is an isomorphism by the induction hypothesis. By the diagram above, we obtain an isomorphism $\Gamma_k/\Gamma_{k+1}\xrightarrow{\cong}\Gamma^\Sigma_k/\Gamma^\Sigma_{k+1}$ if we prove \[\text{im}\Big(H_2(E_L;\mathbb{Z}[\pi_1(M)])\to H_2(\pi/\Gamma_k;\mathbb{Z}[\pi_1(M)])\Big)\] is mapped isomorphically to \[\text{im}\Big(H_2(E_\Sigma;\mathbb{Z}[\pi_1(M)])\to H_2(\pi^\Sigma/\Gamma^\Sigma_k;\mathbb{Z}[\pi_1(M)])\Big).\] We then apply the Five Lemma to obtain the desired isomorphism $\pi/\Gamma_{k+1}\xrightarrow{\cong}\pi^\Sigma/\Gamma^\Sigma_{k+1}$.

We build a tower of spaces $X_j(\Sigma)$, analogous to $X_j(L)$, similarly to the proof of Proposition~\ref{nbasing-propconc}. The space $X_j(\Sigma)$ has the decomposition $X_j(\Sigma)=M_{p^\Sigma_j}^\times\cup_{\partial\nu\Sigma}\nu\Sigma$. For each $j$, there is a canonical morphism of diagrams $\mathcal{D}_j(L)\to\mathcal{D}_j(\Sigma)$ which induces a map $X_j(L)\to X_j(\Sigma)$ restricting to inclusions $\nu L\hookrightarrow\nu\Sigma$ and $E_L\hookrightarrow E_\Sigma$. Thus, we have the commutative square of Mayer-Vietoris sequences below, where all coefficients are in $\mathbb{Z}[\pi_1(M)]$. We use this diagram to study the images of $H_2(E_L)\to H_2(\pi/\Gamma_k)$ and $H_2(E_\Sigma)\to H_2(\pi^\Sigma/\Gamma^\Sigma_k)$. Observe that $H_2(\nu L)=H_2(\nu\Sigma)=0$ and $H_1(\partial E_L)\to H_1(\partial\nu\Sigma)$ is injective.

\begin{center}
\begin{tikzcd}[column sep=-0.5em]
H_2(\partial E_L) \arrow[rr] \arrow[dr,equal] \arrow[dd] & & H_2(E_L) \arrow[dr] \arrow[dd] \arrow[rr] & & H_2(M) \arrow[rr] \arrow[dd,equal] \arrow[dr] & & H_1(\partial E_L) \arrow[dd,hookrightarrow] \arrow[dr,equal] & \\ 
& H_2(\partial E_L) \arrow[rr, crossing over] & & H_2(\pi/\Gamma_k) \arrow[rr,crossing over] & & H_2(X_k(L)) \arrow[rr, crossing over] & & H_1(\partial E_L) \arrow[dd,hookrightarrow] \\
H_2(\partial\nu\Sigma) \arrow[dr,equal] \arrow[rr] & & H_2(E_\Sigma) \arrow[dr] \arrow[rr] & & H_2(M\times[0,1]) \arrow[dr] \arrow[rr] & & H_1(\partial\nu\Sigma) \arrow[dr,equal] & \\
& H_2(\partial\nu\Sigma) \arrow[from=uu, crossing over] \arrow[rr] & & H_2(\pi^\Sigma/\Gamma^\Sigma_k) \arrow[from=uu, crossing over, "\cong" near start] \arrow[rr] & & H_2(X_k(\Sigma)) \arrow[from=uu, crossing over] \arrow[rr] & & H_1(\partial\nu\Sigma)
\end{tikzcd}
\end{center}

First, we observe that \[\text{im}\Big(H_2(E_L)\to H_2(\pi/\Gamma_k)\to H_2(\pi^\Sigma/\Gamma^\Sigma_k)\Big)\subseteq\text{im}\Big(H_2(E_\Sigma)\to H_2(\pi^\Sigma/\Gamma^\Sigma_k)\Big).\] The reverse inclusion is a diagram chase which we briefly describe. Let $x\in\text{im}\big(H_2(E_\Sigma)\to H_2(\pi^\Sigma/\Gamma^\Sigma_k)\big)$. Denote the image of $x$ in $H_2(X_k(\Sigma))$ by $y$. Then $y\in\text{im}\big(H_2(M\times[0,1])\to H_2(X_k(\Sigma))\big)$. We may then consider $y$ as the image of some $z\in H_2(M)$. The image $\Delta(z)$ of $z$ under the connecting homomorphism is 0 because $\Delta(y)=0$ and $H_1(\partial E_L)\hookrightarrow H_1(\partial\nu\Sigma)$ is injective. Thus, $z$ is the image of some $\tilde{z}\in H_2(E_L)$. Since the images of $\tilde{z}$ and $x$ agree in $H_2(X_k(\Sigma))$, the image of $\tilde{z}$ in $H_2(\pi^\Sigma/\Gamma^\Sigma_k)$ differs from $x$ by the image of some element in $H_2(\partial\nu\Sigma)$. By Lemma~\ref{htpychar-ncob-lemmancob}, since $\Sigma$ is an $n$-cobordism, this difference is represented by an element in the image of $H_2(\partial E_L)$. Adjusting $\tilde{z}$ by this element, we obtain an element $\tilde{x}\in H_2(E_L)$ which maps to $x\in H_2(\pi^\Sigma/\Gamma^\Sigma_k)$. This completes the proof of the claim. The induction carries through the step from $k=n$ to $k=n+1$, after which Lemma~\ref{htpychar-ncob-lemmancob} no longer applies since an $n$-cobordism only guarantees the desired property through level $n$.

The same argument shows that the inclusion $E_{L'}\hookrightarrow E_\Sigma$ induces isomorphisms $\pi'/\Gamma'_k\xrightarrow{\cong}\pi^\Sigma/\Gamma^\Sigma_k$ for all $k\leq n+1$. We will take $\phi$, part of the data for the desired $n$-basing, to be the composite $\phi:\pi'/\Gamma'_n\xrightarrow{\cong}\pi^\Sigma/\Gamma^\Sigma_n\xrightarrow{\cong}\pi/\Gamma_n$. To complete the proof of the Proposition, we must find a canonical admissible isomorphism $\phi_\partial: H_1(\partial E_{L'})\xrightarrow{\cong} H_1(\partial E_L)$ compatible with $\phi$. Define $\phi_\partial$ as follows: Send each meridian $m_i'$ of $L'$ to the corresponding meridian $m_i$ of $L$. A longitude $\lambda_i'$ for the component $L'_i$ of $L'$ determines a corresponding copy of $\Sigma_i$ in $\partial\nu\Sigma$, which we also denote $\Sigma_i$. Send $\lambda_i'$ to the longitude $\lambda_i$ for the component $L_i$ of $L$ which is homologous to $\lambda_i'$ in this pushoff. The isomorphism $\phi_\partial$ is admissible.

We now check that $\phi_\partial$ is compatible with $\phi$. We may assume $\Sigma\subset M\times[0,1]$ misses $*\times[0,1]$. To simplify notation, we work inside the reduced cylinder $(M\times[0,1])^\times$. Choose any basing $\tau'$ for $L'$. Also choose any basing $\tau^0$ for $L$. We will eventually modify $\tau^0$. We consider these basings in $E_{L'}$ and $E_L$, respectively, and assume for each $i$ that the basings of the components $L_i'$ and $L_i$ meet $\partial\nu\Sigma_i$ in $\Sigma_i$. 

For each $i$, choose a path $\gamma_i$ on $\Sigma_i$ connecting the basings of $L_i'$ and $L_i$ and forming a based loop $\alpha_i$ together with the two basings.
The $\alpha_i$ represent elements in $\pi^\Sigma$. Let $\beta_i$ denote a lift of the image of $\alpha_i$ under the composite $\pi^\Sigma\twoheadrightarrow\pi^\Sigma/\Gamma^\Sigma_n\xrightarrow{\cong}\pi/\Gamma_n$ to an element of $\pi$. Modify the $i^{\text{th}}$ factor of the basing $\tau^0$ of $L$ by $\beta_i$. Call this new basing $\tau=\beta\tau^0$. We claim that $\tau$ is the desired basing for $L$, that is, the following diagram commutes:
\begin{center}
\begin{tikzcd}
\pi_1(\partial E_{L'}\cup\tau') \arrow[rr, "\phi_{\partial*}"', "\cong"] \arrow[d] & & \pi_1(\partial E_L\cup\tau) \arrow[d] \\
\pi'/\Gamma'_n \arrow[r, "\cong"] & \pi^\Sigma/\Gamma^\Sigma_n \arrow[from=r,"\cong"'] & \pi/\Gamma_n.
\end{tikzcd}
\end{center}

It is enough to check that this commutes for (based) meridians and longitudes of $L'$. We have two basings for $L$ in $E_\Sigma$: the basing $\tau=\beta\tau^0$ (which is a basing in $E_L$), and the basing $\tau'\gamma$ given by following the basing for $L'$ with the paths on $\Sigma$ connecting $L'$ to $L$. The based meridians and longitudes of $L$ given by these two different basings agree in $\pi^\Sigma/\Gamma^\Sigma_n$: In $\pi^\Sigma/\Gamma^\Sigma_n$, the images of $\beta_i$ and $\alpha_i$ agree. But $\alpha_i\tau^0_i$ is homotopic rel endpoints to $\tau'_i\gamma_i$. It is therefore enough to show that the following diagram commutes: 
\begin{center}
\begin{tikzcd}
\pi_1(\partial E_{L'}\cup\tau') \arrow[r, "\phi_{\partial*}"', "\cong"] \arrow[d] & \pi_1(\partial E_L\cup\tau'\gamma) \arrow[d] \\
\pi'/\Gamma'_n \arrow[r, "\cong"] & \pi^\Sigma/\Gamma^\Sigma_n.
\end{tikzcd}
\end{center}
This commutes on meridians, as the based meridians $\tau'\gamma$ induces for $L$ are just the based meridians $\tau'$ induces for $L'$ dragged down $\Sigma$ along $\gamma$. To see that this diagram also commutes on longitudes, we consider $\tau'$ as a basing for $\Sigma$ and leverage the definition of $n$-cobordism. Consider the based longitudes $l'_i=\tau'_i\lambda'_i$ for $L'$ and $l_i=\tau'_i\gamma_i\lambda_i$ for $L$ as elements of $\pi_1(\Sigma_i)$ (with the basing $\tau'_i$). Since the two longitudes cobound $\Sigma_i$, the product $l_i'l_i^{-1}$ is equal to $\prod_j[a^i_j,b^i_j]$, where $\{a^i_j,b^i_j\}_j$ is a symplectic basis for $\Sigma_i$ with appropriate basings. Since $\Sigma$ is an $n$-cobordism, the image of $\pi_1(\Sigma_i)$ in $\pi^\Sigma/\Gamma^\Sigma_n$ is cyclic, so $l_i'l_i^{-1}=\prod_j[a^i_j,b^i_j]\mapsto 1\in \pi^\Sigma/\Gamma^\Sigma_n$. Thus, the above diagram commutes, and $(\phi,\phi_\partial)$ is a canonical $n$-basing for $L'$ relative to $L$ induced by $\Sigma$.
\end{proof}


\subsection{Proof of Theorem~\ref{THMcharacterization}}\label{htpychar-thmB}

We now move toward the proof of Theorem~\ref{THMcharacterization}, which asserts the equivalence of $n$-cobordism and equality of $h$-invariants; see Section~\ref{results-ncob} for the statement of the theorem. We first prove Theorem~\ref{THMcob}, which is just Theorem~\ref{THMcharacterization} in the case where one of the two links we compare is the fixed link $L$.

\begin{customthm}{B$_0$}\label{THMcob}
Fix an $m$-component link $L\subset M$. Let $L'\subset M$ be another $m$-component link, and suppose there exists an $n$-basing $(\phi,\phi_\partial)$ for $L'$ relative to $L$. Then the following statements are equivalent:
\begin{enumerate}[label=(\arabic*)]
\item The invariant $h_n(L',\phi)$ vanishes.
\item The links $L$ and $L'$ are $n$-cobordant via an $n$-cobordism inducing the $n$-basing $(\phi,\phi_\partial)$.
\end{enumerate}
As a consequence, the following basing-independent statements are equivalent:
\begin{enumerate}[label=(\arabic*\,$'$)]
\item The invariant $h_n(L')$ vanishes.
\item The links $L$ and $L'$ are $n$-cobordant.
\end{enumerate}
\end{customthm}
\begin{proof}
\vspace{-0.5em}
The equivalence of (1$'$) and (2$'$) will be a consequence of the equivalence of (1) and (2). \vspace{1em}

\noindent (1)$\implies$(2). Suppose $h_n(L',\phi)$ vanishes. Then $h_n(L',\phi)=h_n(L,\text{id})$, so there is a based homotopy $H:M\times[0,1]\to X_n(L)$ restricting to $h_n(L,\text{id})$ on $M\times\{0\}$ and $h_n(L',\phi)$ on $M\times\{1\}$. Take the transverse preimage of $L\subset X_n(L)$ under $H$ to obtain a properly-embedded surface $\Sigma\subset M\times[0,1]$. Since $L$ and $L'$ are the respective preimages of $L$ under $h_n(L,\text{id})$ and $h_n(L',\phi)$, the surface $\Sigma$ restricts to $L$ and $L'$ on $M\times\{0\}$ and $M\times\{1\}$, respectively. 

Consider the surface exterior $E_\Sigma$, and note that $H$ restricts to a map $E_\Sigma\to M_{p_n}^\times$. This induces a homomorphism $\pi^\Sigma\to\pi/\Gamma_n$ over $\pi_1(M)$ which induces $\pi^\Sigma/\Gamma^\Sigma_n\to\pi/\Gamma_n$, where $\pi^\Sigma=\pi_1(E_\Sigma)$ and $\Gamma^\Sigma=\ker\big(\pi^\Sigma\twoheadrightarrow\pi_1(M)\big)$. We also have the homomorphism $\pi\to\pi^\Sigma$ induced by inclusion which induces a homomorphism $\pi/\Gamma_n\to\pi^\Sigma/\Gamma^\Sigma_n$. Since the homotopy $H$ restricts to $h_n(L,\text{id})$ on $M\times\{0\}$, the composition $\pi/\Gamma_n\to\pi^\Sigma/\Gamma^\Sigma_n\to\pi/\Gamma_n$ is $\text{id}_{\pi/\Gamma_n}$. In particular, $\pi/\Gamma_n\to\pi^\Sigma/\Gamma^\Sigma_n$ is injective and $\pi^\Sigma/\Gamma^\Sigma_n\to\pi/\Gamma_n$ is surjective.

The surface $\Sigma$ would be our candidate for an $n$-cobordism between the two links, except it may have too many components. In Claim 1, we perform surgeries on $\Sigma$ to ensure the preimage of each component of $L$ has only one component. More precisely, we surger $M\times[0,1]$ by attaching cancelling pairs of 5-dimensional 1- and 2-handles to the interior of $(M\times[0,1])\times\{1\}\subset(M\times[0,1])\times[0,1]$. The resulting surgery cobordism contains a cobordism from $\Sigma$ to a new surface in $M\times[0,1]$ with one fewer component. The surjectivity of $\pi^\Sigma/\Gamma^\Sigma_n\twoheadrightarrow\pi/\Gamma_n$ is key for guaranteeing that $H:M\times[0,1]\to X_n(L)$ extends over the surgery cobordism. In Claim 2, we argue that $\pi/\Gamma_n\hookrightarrow\pi^\Sigma/\Gamma^\Sigma_n$ is an isomorphism. We use an induction argument, proving $\Gamma_k/\Gamma_{k+1}\to\Gamma^\Sigma_k/\Gamma^\Sigma_{k+1}$ is an isomorphism for all $k<n$ and applying the Five Lemma. \vspace{1em}

\noindent Claim 1. We may modify the surface $\Sigma$ and the homotopy $H$ rel boundary so that the preimage of each component of $L$ under $H$ is a connected surface with boundary. \vspace{1em}

Suppose the preimage of some component $L_i$ of $L$ under $H$ has more than one component. Choose some component $\Sigma'_i$ of $H^{-1}(L_i)$, and let $\Sigma_i$ be the union of the remaining components. Find a path $\alpha$ in $E_\Sigma$ from the interior of $\partial\nu\Sigma_i$ to the interior of $\partial\nu\Sigma'_i$ whose boundary maps under $H$ to a single point of $\partial\nu L_i$ and whose image in $M_{p_n}^\times$ is nullhomotopic. Such a path exists since $\pi^\Sigma/\Gamma^\Sigma_n\twoheadrightarrow\pi/\Gamma_n$ is surjective. 

We now perform surgeries to connect $\Sigma'_i$ to $\Sigma_i$, modifying the map $H$ rel boundary. Connect $\alpha$ to $\Sigma_i$ and $\Sigma'_i$ by small intervals in $D^2$-fibers of the regular neighborhoods of the surfaces whose images in $X_n(L)$ agree. We also call this extended path $\alpha$. Its endpoints map to some point $p\in L_i\subset X_n(L)$. Parametrize the closure of a neighborhood of $p$ as $D^1\times D^2$, where $D^1\subset L_i$ is the closure of a neighborhood of $p$ in $L_i$. The endpoints of $\alpha$ have neighborhoods whose closures may be parametrized as $D^1\times D^1\times D^2$ which map under $H$ to $D^1\times D^2$ by projection onto the last two coordinates. 

Attach a 5-dimensional 1-handle $D^1\times D^4=D^1\times(D^1\times D^1\times D^2)$ to $(M\times[0,1])\times\{1\}\subset (M\times[0,1])\times[0,1]$ along the two 4-balls $D^1\times D^1\times D^2$ to obtain a surgery cobordism $Z$ from $M\times[0,1]$ to a 4-manifold $W$ which is itself a cobordism from $M$ to $M$, the result of surgery on $M\times[0,1]$ corresponding to the attachment of the 1-handle. Extend the map $H\times[0,1]:(M\times[0,1])\times[0,1]\to X_n(L)$ to $Z$ by projection of $D^1\times D^4$ to $D^4$ followed by $H$, which agrees on $\{0\}\times D^4$ and $\{1\}\times D^4$. The map $Z\to X_n(L)$ restricts to a map $H':W\to X_n(L)$ which further restricts to $h_n(L,\text{id})$ and $h_n(L',\phi)$ on $\partial W=M\sqcup M$. Define a surface $\Sigma_i''\subset W$ by removing from $\Sigma_i$ and $\Sigma'_i$ the $D^2$-neighborhoods of the endpoints of $\alpha$ and gluing in $D^1\times S^1\subset D^1\times D^4$. The surface $\Sigma''_i$ has one fewer component than $\Sigma_i\cup\Sigma'_i$.

We now attach a 5-dimensional 2-handle to $Z$ to cancel the 1-handle described above. The attaching circle for this 2-handle is the path $\alpha$ along with the core of the 1-handle pushed to the boundary of $Z$. The map $Z\to X_n(L)$ extends over the 2-handle with the image of the 2-handle missing $\nu L\subset X_n(L)$ since the image of the attaching circle in $M_{p_n}^\times$ is nullhomotopic. Because the handles cancel, the resulting space is homeomorphic to $(M\times[0,1])\times[0,1]$. The resulting map $(M\times[0,1])\times[0,1]\to X_n(L)$ restricts to a map $H'':(M\times[0,1])\times\{1\}\to X_n(L)$ which is another homotopy from $h_n(L,\text{id})$ to $h_n(L',\phi)$. 

By construction, $(H'')^{-1}(L)$ has one fewer component than $H^{-1}(L)$. We may continue inductively, reducing the number of components of $H^{-1}(L)$ until we reach an $m$-component surface in $M\times[0,1]$ cobounded by $L$ and $L'$ which is the preimage of $L$ under a homotopy $M\times[0,1]\to X_n(L)$ from $h_n(L,\text{id})$ to $h_n(L',\phi)$. We continue to call this surface $\Sigma$ and this homotopy $H$. This proves Claim 1. \vspace{1em}

\noindent Claim 2. The new surface $\Sigma\subset M\times[0,1]$ induces an isomorphism $\pi/\Gamma_n\xrightarrow{\cong}\pi^\Sigma/\Gamma^\Sigma_n$. \vspace{1em}

The proof of Claim 2 is analogous to part of the proof of Proposition~\ref{results-ncob-propncob}, so we will omit some details. Since we modified the homotopy $H$ rel boundary, the new homotopy obtained through the proof of Claim 1, which we still call $H$, restricts to $h_n(L,\text{id})$ on $M\times\{0\}$. Thus, the composition $\pi/\Gamma_n\to\pi^\Sigma/\Gamma^\Sigma_n\to\pi/\Gamma_n$ remains $\text{id}_{\pi/\Gamma_n}$. In fact, the composition $\pi/\Gamma_k\to\pi^\Sigma/\Gamma^\Sigma_k\to\pi/\Gamma_k$ is $\text{id}_{\pi/\Gamma_k}$ for all $k\leq n$. So $\pi/\Gamma_k\hookrightarrow\pi^\Sigma/\Gamma^\Sigma_k$ is injective for all $k\leq n$.

The proof of Claim 2 for $n=2$ is entirely analogous to part of the proof of Proposition~\ref{results-ncob-propncob}. We now use induction. Suppose $\pi/\Gamma_k\to\pi^\Sigma/\Gamma^\Sigma_k$ is an isomorphism for some $2\leq k< n$. We have already seen that $\pi/\Gamma_{k+1}\hookrightarrow\pi^\Sigma/\Gamma^\Sigma_{k+1}$ is injective. Consider the diagram of short exact sequences below.
\begin{center}
\begin{tikzcd}
1 \arrow[r] & \Gamma_k/\Gamma_{k+1} \arrow[r,hookrightarrow] \arrow[d, hookrightarrow] & \pi/\Gamma_{k+1} \arrow[r,twoheadrightarrow] \arrow[d, hookrightarrow] & \pi/\Gamma_k \arrow[r] \arrow[d,"\cong"] & 1 \\
1 \arrow[r] & \Gamma^\Sigma_k/\Gamma^\Sigma_{k+1} \arrow[r,hookrightarrow] & \pi^\Sigma/\Gamma^\Sigma_{k+1} \arrow[r,twoheadrightarrow] & \pi^\Sigma/\Gamma^\Sigma_k \arrow[r] & 1
\end{tikzcd}
\end{center}
To prove $\pi/\Gamma_{k+1}\hookrightarrow\pi^\Sigma/\Gamma^\Sigma_{k+1}$ is also surjective, it suffices to show $\Gamma_k/\Gamma_{k+1}\hookrightarrow\Gamma^\Sigma_k/\Gamma^\Sigma_{k+1}$ is surjective. We again invoke the Stallings 5-term natural exact sequence \cite{Stallings} in the following diagram:

\begin{center}
\begin{tikzcd}
H_2(\Gamma) \arrow[r] \arrow[d] & H_2(\Gamma/\Gamma_k) \arrow[r,twoheadrightarrow] \arrow[d,"\cong"] & \Gamma_k/\Gamma_{k+1} \arrow[d] \arrow[r] & 0 \\
H_2(\Gamma^\Sigma) \arrow[r] & H_2(\Gamma^\Sigma/\Gamma^\Sigma_k) \arrow[r,twoheadrightarrow] & \Gamma^\Sigma_k/\Gamma^\Sigma_{k+1} \arrow[r] & 0.
\end{tikzcd}
\end{center}
Since the composite $H_2(\Gamma/\Gamma_k)\xrightarrow{\cong} H_2(\Gamma^\Sigma/\Gamma^\Sigma_k)\twoheadrightarrow\Gamma^\Sigma_k/\Gamma^\Sigma_{k+1}$ is surjective, so is $\Gamma_k/\Gamma_{k+1}\to\Gamma^\Sigma_k/\Gamma^\Sigma_{k+1}$. This proves Claim 2.

We now have the following diagram of groups induced by the surface $\Sigma\subset M\times[0,1]$ and the homotopy $H$:
\begin{center}
\begin{tikzcd}
\pi'/\Gamma'_n \arrow[d] \arrow[dr,"\phi", "\cong"'] & \\
\pi^\Sigma/\Gamma^\Sigma_n \arrow[r] & \pi/\Gamma_n. \\
\pi/\Gamma_n \arrow[u,"\cong"] \arrow[ur,equal] &
\end{tikzcd}
\end{center}
Thus, we also have isomorphisms $\pi'/\Gamma'_n\xrightarrow{\cong}\pi^\Sigma/\Gamma^\Sigma_n$ and $\pi^\Sigma/\Gamma^\Sigma_n\xrightarrow{\cong}\pi/\Gamma_n$. In particular, since the map $M\times[0,1]\to X_n(L)$ sends $\Sigma$ to $L$, for each component $\Sigma_i$ of $\Sigma$ the image of $\pi_1(\Sigma_i)$ is sent to the image of $\pi_1(L_i)$ under the isomorphism $\pi^\Sigma/\Gamma^\Sigma_n\xrightarrow{\cong}\pi/\Gamma_n$. Because the inverse of this isomorphism is induced by the inclusion $E_L\hookrightarrow E_\Sigma$, the surface $\Sigma$ is an $n$-cobordism by definition. Furthermore, $\Sigma$ induces the isomorphisms $\phi$ and $\phi_\partial$ by construction. This proves (1) implies (2). \vspace{1em}

\noindent (2)$\implies$(1). Now suppose $\Sigma\subset M\times[0,1]$ is an $n$-cobordism between $L$ and $L'$. By Proposition~\ref{results-ncob-propncob}, $\Sigma$ induces some $n$-basing, say, $(\phi,\phi_\partial)$. This yields inclusions of the pairs $(E_L,\partial E_L)$ and $(E_{L'},\partial E_{L'})$ into $(E_\Sigma,\partial\nu\Sigma)$ which induce the isomorphism $\phi_\partial$ on the boundary. We have the following diagram of isomorphisms:
\begin{center}
\begin{tikzcd}
\pi'/\Gamma'_n \arrow[d,"\cong"'] \arrow[dr,"\phi", "\cong"'] & \\
\pi^\Sigma/\Gamma^\Sigma_n \arrow[r,"\cong"] & \pi/\Gamma_n. \\
\pi/\Gamma_n \arrow[u,"\cong"] \arrow[ur,equal] &
\end{tikzcd}
\end{center}
We would like to obtain the diagram of pairs of spaces below. We wish to construct the map $(E_\Sigma,\partial\nu\Sigma)\to(M_{p_n}^\times,\partial E_L)$.
\begin{center}
\begin{tikzcd}
(E_{L'},\partial E_{L'}) \arrow[d,hookrightarrow] \arrow[dr] & \\
(E_\Sigma,\partial\nu\Sigma) \arrow[r,dashed] & (M_{p_n}^\times,\partial E_L) \\
(E_L,\partial E_L) \arrow[u,hookrightarrow] \arrow[ur] &
\end{tikzcd}
\end{center}

We first extend the maps on $E_L$ and $E_{L'}$ to $\partial\nu\Sigma$. Choose a basing for $L$ (which induces a basing on $\Sigma$), and choose a trivialization of $\nu L$ (hence $\nu\Sigma$ and $\nu L'$). We have the diagram below for each $i$, where $L_i$ and $\Sigma_i$ are parallels of the cores of $\nu L$ and $\nu\Sigma$, respectively, induced by the trivializations.
\begin{center}
\begin{tikzcd}
\pi_1(L_i) \arrow[d] \arrow[r] & \pi/\Gamma_n \arrow[d,"\cong"] \\
\pi_1(\Sigma_i) \arrow[r] & \pi^\Sigma/\Gamma^\Sigma_n
\end{tikzcd}
\end{center}

Since $\Sigma$ is an $n$-cobordism, the image of $\pi_1(L_i)$ in $\pi^\Sigma/\Gamma^\Sigma_n$ is onto the image of $\pi_1(\Sigma_i)$. Because $\pi/\Gamma_n\xrightarrow{\cong}\pi^\Sigma/\Gamma^\Sigma_n$ is an isomorphism, the same is true in $\pi/\Gamma_n$. We extend the map to the 1-skeleton of $\Sigma$ by sending each based loop in $\Sigma$ to a based loop in $L$ whose image agrees in $\pi/\Gamma_n$. There is no obstruction to defining a map from the 2-skeleton of $\Sigma$ (i.e. all of $\Sigma$) to $L$. To complete the extension to $\partial\nu\Sigma$, we use the trivializations to map $\partial\nu\Sigma\cong S^1\times\Sigma\to S^1\times L\cong\partial E_L$. We now have a map of pairs $(E_L\cup\partial\nu\Sigma\cup E_{L'},\partial\nu\Sigma)\to(M_{p_n}^\times,\partial E_L)$.

As $M_{p_n}^\times$ is an Eilenberg-MacLane space, there is no obstruction to extending to all of $E_\Sigma$ given the diagram of group isomorphisms above. Thus, we obtain a map of pairs $(E_\Sigma,\partial\nu\Sigma)\to(M_{p_n}^\times,\partial E_L)$ which fits into the diagram above. This map restricts to a bundle map $\partial\nu\Sigma\to\partial E_L$ from the sphere bundle of $\Sigma$ to the sphere bundle of $L\subset X_n(L)$. Extend to the corresponding disk bundles $\nu\Sigma$ and $\nu L$ by coning on fibers. This yields a map $H:M\times[0,1]\to X_n(L)$ which is a homotopy between $h_n(L,\text{id})$ and $h_n(L',\phi)$. Thus, the invariant $h_n(L',\phi)$ vanishes. \vspace{1em}

We have now proven the equivalence of (1) and (2). To complete the proof of the theorem, we show the corresponding basing-independent statements are equivalent. \vspace{1em}

\noindent (1$'$) $\Longleftrightarrow$ (2$'$). The invariant $h_n(L')$ vanishes if and only if $h_n(L',\phi)$ vanishes for some $n$-basing $\phi$. By the equivalence of (1) and (2), this is true if and only if $L$ and $L'$ are $n$-cobordant via some $n$-cobordism inducing the $n$-basing $\phi$, i.e. if and only if $L$ and $L'$ are $n$-cobordant.
\end{proof}

Theorem~\ref{THMcharacterization}, which compares links $L'$ and $L''$ relative to the fixed link $L$, follows readily from Theorem~\ref{THMcob}.

\begin{proof}[Proof of Theorem~\ref{THMcharacterization}]
Fix a link $L\subset M$, and let $L', L''\subset M$ be links which admit $n$-basings $(\phi',\phi'_\partial)$ and $(\phi'',\phi''_\partial)$, respectively, relative to $L$. Lemmas~\ref{nbasing-selfbasing-lemmainverse} and~\ref{nbasing-selfbasing-lemmacompose} imply $(\psi,\psi_\partial)=((\phi')^{-1}\circ\phi'',(\phi'_\partial)^{-1}\circ\phi''_\partial)$ is an $n$-basing of $L''$ relative to $L'$. Let $\overline{h}_{\phi'}$ be the homotopy equivalence of pairs $\overline{h}_{\phi'}:(X_n(L'),\nu L')\xrightarrow{\simeq}(X_n(L),\nu L)$ induced by the $n$-basing $(\phi',\phi'_\partial)$ (provided by Proposition~\ref{nbasing-propnbasing}). Observe that $\overline{h}_{\phi'}\circ h_n(L',\text{id})=h_n(L',\phi')$ and $\overline{h}_{\phi'}\circ h_n(L'',\psi)=h_n(L'',\phi'')$. By Theorem~\ref{THMcob}, $L'$ and $L''$ are $n$-cobordant via an $n$-cobordism inducing the $n$-basing $\psi$ if and only if $h_n(L'',\psi)=h_n(L',\text{id})$. Since $\overline{h}_{\phi'}$ is a homotopy equivalence, this is the case if and only if $\overline{h}_{\phi'}\circ h_n(L'',\psi)=\overline{h}_{\phi'}\circ h_n(L',\text{id})$, that is, $h_n(L'',\phi'')=h_n(L',\phi')\in[M,X_n(L)]_0$.

The equivalence of the analogous basing-independent statements follows immediately from the equivalence of (1) and (2) in a similar manner to the proof of Theorem~\ref{THMcob}.
\end{proof}


\subsection{Proof of Theorem~\ref{THMrealization}}\label{htpychar-thmC}

Recall that Theorem~\ref{THMrealization} states precise conditions under which a homotopy class $f\in[M,X_n(L)]_0$ is realized as $h_n(L',\phi)$ for some link $L'\subset M$ and some $n$-basing $\phi$; see Section~\ref{results-htpyreal} for the full statement. Some elements of the proof of Theorem~\ref{THMrealization} mirror arguments from the proof of Theorem~\ref{THMcob}.

\begin{proof}
($\implies$). Suppose $f\in[M,X_n(L)]_0$ is realizable. Then $f=h_n(L',\phi)$ for some $m$-component link $L'\subset M$ which admits the $n$-basing $(\phi,\phi_\partial)$ relative to $L$. By Proposition~\ref{nbasing-propnbasing}, the $n$-basing $\phi$ induces a canonical homotopy equivalence of pairs \[\overline{h}_\phi:(X_n(L'),\nu L')\xrightarrow{\simeq}(X_n(L),\nu L).\] Since $f=h_n(L',\phi)$, we have by definition that $f\simeq \overline{h}_\phi\circ\iota_n(L')$. Recall that we obtain the following diagram:
\begin{center}
\begin{tikzcd}
M \arrow[r,hookrightarrow,"\iota_n(L')"] \arrow[drr] & X_n(L') \arrow[dr] & \\
& & B\pi_1(M). \\
M \arrow[r,hookrightarrow,"\iota_n(L)"'] \arrow[urr] & X_n(L) \arrow[ur] \arrow[from=uu, crossing over, "\overline{h}_\phi","\simeq"'] &
\end{tikzcd}
\end{center}
We can see that $f_*=(\overline{h}_\phi)_*\circ\iota_n(L')_*=\iota_n(L)_*$, that is, condition (2) holds, from the induced diagram on fundamental groups:
\begin{center}
\begin{tikzcd}
\pi_1(M) \arrow[r,"\cong"] \arrow[drr,equal] & \pi_1(X_n(L')) \arrow[dr,"\cong"] & \\
& & \pi_1(M). \\
\pi_1(M) \arrow[r,"\cong"'] \arrow[urr,equal] & \pi_1(X_n(L)) \arrow[ur,"\cong"'] \arrow[from=uu, crossing over, "\cong"'] &
\end{tikzcd}
\end{center}

Next, since $\overline{h}_\phi$ is a map of pairs, the composition $\Delta\circ f_*=\Delta\circ(\overline{h}_\phi)_*\circ\iota_n(L')_*$ sends $[M]$ to $\Sigma[T_i]$, so condition (1) holds.

Finally, we show condition (3) holds, which concerns the homomorphism \[\text{proj}\circ(\cap f_*[M]):H^1(X_n(L);\mathbb{Z}[\pi_1(M)])\to H_2(X_n(L);\mathbb{Z}[\pi_1(M)])/\text{im}(K_j(\pi/\Gamma_n)).\] Consider diagram (\ref{diagram-CO}) below, where all coefficients are in $\mathbb{Z}[\pi_1(M)]$ and $j<n$. Recall that $K_j(\pi/\Gamma_n)$ is the kernel of $H_2(\pi/\Gamma_n)\to H_2(\pi/\Gamma_j)$.
\begin{equation*}
\begin{tikzcd}
& H^1(M) \arrow[d,"\cap{[M]}"',"\cong"] & H^1(X_n(L)) \arrow[l,"f^*"',"\cong"] \arrow[d,"\cap f_*{[M]}"] \arrow[dr,"\text{proj}\circ(\cap f_*{[M]})"] & \\
H_2(E_{L'}) \arrow[r,"j'_*"] \arrow[d,twoheadrightarrow] & H_2(M) \arrow[r,"f_*"] & H_2(X_n(L)) \arrow[r,twoheadrightarrow,"\text{proj}"] & H_2(X_n(L))/\text{im}(K_j(\pi/\Gamma_n)) \\
H_2(\pi') \arrow[r] & H_2(\pi'/\Gamma'_n) \arrow[r,"\phi_*","\cong"'] & H_2(\pi/\Gamma_n) \arrow[r,twoheadrightarrow,"\text{proj}"] \arrow[u] & H_2(\pi/\Gamma_n)/K_j(\pi/\Gamma_n) \arrow[u]
\end{tikzcd} \tag{$\ast$} \label{diagram-CO}
\end{equation*}
The map $\cap[M]$ is an isomorphism by Poincar\'{e} duality. The map $\phi_*$ is an isomorphism by the assumption that $\phi$ is an $n$-basing.

To prove that the map $f^*$ on the top row is an isomorphism, we invoke the (cohomology) Universal Coefficient Spectral Sequence, a cohomology spectral sequence with \[E_2^{p,q}=\text{Ext}_{\mathbb{Z}[\pi_1(M)]}^p\big(H_q(-;\mathbb{Z}[\pi_1(M)),\mathbb{Z}[\pi_1(M)]\big)\] which converges to $H^*(-;\mathbb{Z}[\pi_1(M)])$. See \cite{LevineJ77} Theorem 2.3 or \cite{Rotman} Theorem 11.34 (note that the pages of this spectral sequence are the transposes of those seen in \cite{LevineJ77}, with the differentials modified accordingly). From this spectral sequence, we obtain the following natural short exact sequence akin to the Universal Coefficient Theorem for cohomology. For convenience, we suppress the coefficient ring $\mathbb{Z}[\pi_1(M)]$ in some places.
\begin{align*}
0\to\text{Ext}^1\big(H_0(-),\mathbb{Z}[\pi_1(M)]\big)&\to H^1(-;\mathbb{Z}[\pi_1(M)])\to \\
&\to\ker\Big(\text{Ext}^0\big(H_1(-),\mathbb{Z}[\pi_1(M)]\big)\to\text{Ext}^2\big(H_0(-),\mathbb{Z}[\pi_1(M)]\big)\Big)\to 0.
\end{align*}
The map $f:M\to X_n(L)$ induces a map of spectral sequences, and therefore a map of short exact sequences of the type seen above. By condition (2), $f$ induces isomorphisms on $H_i(-;\mathbb{Z}[\pi_1(M)])$, $i=0,1$, which implies $f^*:H^1(X_n(L);\mathbb{Z}[\pi_1(M)])\to H^1(M;\mathbb{Z}[\pi_1(M)])$ is an isomorphism by the Five Lemma.

We now apply the \hyperref[THMsd]{Stallings-Dwyer Theorem} with $A=\Gamma'$ and $B=\Gamma/\Gamma_n$ (see Section~\ref{prelims-lcs}). By assumption, $\Gamma'\to\Gamma/\Gamma_n$ induces an isomorphism on $n^{\text{th}}$ lower central quotients. Therefore, with $\mathbb{Z}[\pi_1(M)]$ coefficients, the map $H_2(\pi')\to H_2(\pi/\Gamma_n)/K_j(\pi/\Gamma_n)$, which is the composition of the three maps on the bottom row of diagram (\ref{diagram-CO}), is surjective: Note that $H_2(\pi';\mathbb{Z}[\pi_1(M)])$ is the second homology of the $\pi_1(M)$-cover of a $K(\pi',1)$, which is a $K(\Gamma',1)$, so we have $H_2(\pi';\mathbb{Z}[\pi_1(M)])=H_2(\Gamma';\mathbb{Z})$. Similarly, $H_2(\pi/\Gamma_n;\mathbb{Z}[\pi_1(M)])=H_2(\Gamma/\Gamma_n;\mathbb{Z})$. Combining with the surjectivity of $H_2(E_{L'})\twoheadrightarrow H_2(\pi')$ and following the diagram counter-clockwise beginning with $H_2(E_{L'})$, we see that $H_2(E_{L'})$ surjects onto the image of $H_2(\pi/\Gamma_n)/K_j(\pi/\Gamma_n)$ in $H_2(X_n(L))/\text{im}(K_j(\pi/\Gamma_n))$. Thus, the composition of maps following the diagram clockwise from $H_2(E_{L'})$ to $H_2(X_n(L))/\text{im}(K_j(\pi/\Gamma_n))$ also surjects onto the image of $H_2(\pi/\Gamma_n)/K_j(\pi/\Gamma_n)$ in $H_2(X_n(L))/\text{im}(K_j(\pi/\Gamma_n))$. This implies the image of $\text{proj}\circ(\cap f_*[M])$ contains the image of $H_2(\pi/\Gamma_n)/K_j(\pi/\Gamma_n)$. Thus, condition (3) holds for all $j<n$. \vspace{1em}

\noindent ($\impliedby$). To prove the converse, suppose $f\in[M,X_n(L)]_0$ satisfies conditions (1)-(3). We may assume the image of $f$ contains $L\subset X_n(L)$. Define $L'=f^{-1}(L)$, the transverse preimage of $L$ under $f$. The restriction of $f$ to $E_{L'}$ yields a map $E_{L'}\to M_{p_n}^\times$, hence a homomorphism $\pi'\to\pi/\Gamma_n$. Since $f_*=\iota_n(L)_*$ on fundamental groups by condition (2), this is a homomorphism over $\pi_1(M)$. In Claim 1 below, we will show that the induced homomorphism $\pi'/\Gamma'_n\to\pi/\Gamma_n$ is surjective. Claim 3 will upgrade this to an isomorphism. In between, we will need Claim 2 to arrange that $L'$ has the same number of components as $L$. \vspace{1em}

\noindent Claim 1. The restriction of $f$ to $E_{L'}$ induces a surjective homomorphism $\pi'/\Gamma'_n\twoheadrightarrow\pi/\Gamma_n$. \vspace{1em}

First, we prove $\Gamma'/\Gamma'_2\to\Gamma/\Gamma_2$ is surjective. Consider the diagram of Mayer-Vietoris sequences below corresponding to the decompositions $M=E_{L'}\cup_{\partial E_{L'}}\nu L'$ and $X_n(L)=M_{p_n}^\times\cup_{\partial E_L}\nu L$. Recall that $M_{p_n}^\times$ is a $K(\pi/\Gamma_n,1)$. All coefficients are in $\mathbb{Z}[\pi_1(M)]$.
\begin{center}
\begin{tikzcd}
H_1(\partial E_{L'}) \arrow[r,twoheadrightarrow] \arrow[d,twoheadrightarrow] & H_1(E_{L'})\oplus H_1(\nu L') \arrow[r] \arrow[d] & 0 \\
H_1(\partial E_L) \arrow[r,twoheadrightarrow] & H_1(\pi/\Gamma_n)\oplus H_1(\nu L) \arrow[r] & 0
\end{tikzcd}
\end{center}

The vertical map on the left is surjective by construction. The vertical map on the right splits as a direct sum of homomorphisms. Thus, by the diagram, \[ \Gamma'/\Gamma'_2=H_1(E_{L'};\mathbb{Z}[\pi_1(M)])\to H_1(\pi/\Gamma_n;\mathbb{Z}[\pi_1(M)])=\Gamma/\Gamma_2\] is surjective. By an argument given in the proof of the \hyperref[THMsd]{Stallings-Dwyer Theorem} (see \cite{ChaOrr}), the homomorphism $\Gamma'/\Gamma'_n\to\Gamma/\Gamma_n$ is surjective. Because $\pi'/\Gamma'_n\to\pi/\Gamma_n$ is a homomorphism over $\pi_1(M)$, the Five Lemma implies that $\pi'/\Gamma'_n\to\pi/\Gamma_n$ is surjective. \vspace{1em}

\noindent Claim 2. We may reduce the number of components of $L'$ without altering the homotopy class of the map $f$ such that the resulting link, which we continue to call $L'$, has $m$ components and maps homeomorphically to $L$. \vspace{1em}

The proof of Claim 2 is similar to the proof of Claim 1 of (1)$\implies$(2) from the proof of Theorem~\ref{THMcob} with one notable exception. When we finish reducing the number of components of $L'$, we want each component of the resulting link to map homeomorphically onto a component of $L$. Suppose some component $L_i$ of $L$ has a transverse preimage which is not connected, that is, $f^{-1}(L_i)$ is more than one component of the link $L'$. Similarly to the proof of Theorem~\ref{THMcob}, we attach 4-dimensional 1-handles to $M\times\{1\}\subset M\times[0,1]$ in which we connect the components of $f^{-1}(L_i)$, along with cancelling 4-dimensional 2-handles. We are left with a homotopy $M\times[0,1]\to X_n(L)$ which is $f$ on $M\times\{0\}$ and a map $g$ on $M\times\{1\}$ such that the preimage of $L_i$ under $g$ is a single component $L'_i$. Just as in the proof of Theorem~\ref{THMcob}, we leverage the surjection $\pi'/\Gamma'_n\twoheadrightarrow\pi/\Gamma_n$. By condition (1), $\Delta g_*[M]=\Delta f_*[M]=\sum_i[T_i]$. Thus, by additional homotopy of $g$ (which we may take to be the identity outside a neighborhood of $\nu L'_i$), we may arrange that $g$ maps $L'_i$ homeomorphically to $L_i$. We do the same for all components of $L$ and obtain a map, which we continue to call $g$, such that $g\simeq f$ and $g^{-1}(L)$ is an $m$-component link. This proves Claim 2. \vspace{1em}

\noindent Claim 3. We have an isomorphism $\pi'/\Gamma'_n\xrightarrow{\cong}\pi/\Gamma_n$. \vspace{1em}

Observe that the surgeries performed in the proof of Claim 2 did not alter the surjectivity of the homomorphism $\pi'/\Gamma'_n\twoheadrightarrow\pi/\Gamma_n$.

To begin, we show $\Gamma'/\Gamma'_2=H_1(E_{L'};\mathbb{Z}[\pi_1(M)])\to H_1(E_L;\mathbb{Z}[\pi_1(M)])=\Gamma/\Gamma_2$ is now an isomorphism. Consider the diagram of Mayer-Vietoris sequences below corresponding to the decompositions $M=E_{L'}\cup_{\partial E_{L'}}\nu L'$, $X_n(L)=M_{p_n}^\times\cup_{\partial E_L}\nu L$, and $M=E_L\cup_{\partial E_L}\nu L$. All coefficients are in $\mathbb{Z}[\pi_1(M)]$.

\begin{center}
\begin{tikzcd}
H_2(M) \arrow[r] \arrow[d,"g_*"] & H_1(\partial E_{L'}) \arrow[d,"\cong"] \arrow[r] & H_1(E_{L'})\oplus H_1(\nu L') \arrow[r] \arrow[d] & 0 \\
H_2(X_n(L)) \arrow[r] & H_1(\partial E_L) \arrow[r] & H_1(\pi/\Gamma_n)\oplus H_1(\nu L) \arrow[r] & 0 \\
H_2(M) \arrow[r] \arrow[u,"\iota_n(L)_*"'] & H_1(\partial E_L) \arrow[r] \arrow[u,equal] & H_1(E_L)\oplus H_1(\nu L) \arrow[u,"\cong"'] \arrow[r] & 0
\end{tikzcd}
\end{center}

The rightmost vertical maps split as direct sums of homomorphisms, and $H_1(\nu L')$ maps isomorphically onto $H_1(\nu L)$. It suffices to show the
composition of the second column of vertical maps, which are induced by restrictions of $g$ and $\iota_n(L)$, sends $\text{im}\big(H_2(M)\to H_1(\partial E_{L'})\big)$ isomorphically to $\text{im}\big(H_2(M)\to H_1(\partial E_L)\big)$. To show this, we prove the diagram below commutes. All coefficients are still in $\mathbb{Z}[\pi_1(M)]$.

\begin{center}
\begin{tikzcd}
H^1(M) \arrow[r,"\cap{[M]}","\cong"'] & H_2(M) \arrow[r,"\Delta"] & H_1(\partial E_{L'}) \arrow[d,"\cong"] \\
H^1(X_n(L)) \arrow[u,"g^*","\cong"'] \arrow[d,"\iota_n(L)^*"',"\cong"] & & H_1(\partial E_L) \arrow[d,equal] \\
H^1(M) \arrow[r,"\cap{[M]}","\cong"'] & H_2(M) \arrow[r,"\Delta"] & H_1(\partial E_L)
\end{tikzcd}
\end{center}

The leftmost vertical arrows are isomorphisms by the same Universal Coefficient Spectral Sequence argument used in the forward direction of this proof, since $g_*$ and $\iota_n(L)_*$ are isomorphisms on $H_i(-;\mathbb{Z}[\pi_1(M)])$, $i=0,1$. Consider the two commutative squares below, again with coefficients in $\mathbb{Z}[\pi_1(M)]$. We use these squares to see that the diagram above commutes. The commutativity of the squares below follows from the relationship between the connecting homomorphism in a Mayer-Vietoris sequence and the cap product; see \cite{Spanier}.

\begin{center}
\begin{tabular}{cc}
\begin{tikzcd}
H^1(X_n(L)) \arrow[r,"\text{inc}^*"] \arrow[d,"\cap{(\iota_n(L)_*[M])}"'] & H^1(\partial E_L) \arrow[d,"\cap(\Delta{\iota_n(L)_*[M]})"] \\
H_2(X_n(L)) \arrow[r,"\Delta"] & H_1(\partial E_L)
\end{tikzcd}
&
\begin{tikzcd}
H^1(X_n(L)) \arrow[r,"\text{inc}^*"] \arrow[d,"\cap{(g_*[M])}"'] & H^1(\partial E_L) \arrow[d,"\cap(\Delta{g_*[M]})"] \\
H_2(X_n(L)) \arrow[r,"\Delta"] & H_1(\partial E_L)
\end{tikzcd}
\end{tabular}
\end{center}

We check that the large square commutes. Let $\alpha\in H^1(X_n(L);\mathbb{Z}[\pi_1(M)])$. Then \begin{align*}
g_*\Delta(g^*\alpha\cap[M])&=\Delta g_*(g^*\alpha\cap [M])=\Delta(\alpha\cap g_*[M])=\text{inc}^*(\alpha)\cap(\Delta g_*[M]) \\
&=\text{inc}^*(\alpha)\cap(\Delta \iota_n(L)_*[M])=\Delta(\alpha\cap \iota_n(L)_*[M]) \\
&=\Delta \iota_n(L)_*(\iota_n(L)^*\alpha\cap [M])=\iota_n(L)_*\Delta(\iota_n(L)^*\alpha\cap[M]).
\end{align*}
The third and fifth equalities follow from the commutative squares above. The fourth equality follows from condition (1) or by construction. Thus, the diagram commutes, and we obtain an isomorphism $\Gamma'/\Gamma'_2\xrightarrow{\cong}\Gamma/\Gamma_2$.

We consider the diagram (\ref{diagram-CO}) again, but with $g$ in place of $f$. Recall all coefficients are in $\mathbb{Z}[\pi_1(M)]$. The diagram commutes for all $j<n$. We require $j<n$ simply to ensure that $K_j(\pi/\Gamma_n)$ is defined. By assumption (3), the image of $\text{proj}\circ(\cap g_*[M])=\text{proj}\circ(\cap f_*[M])$ contains the image of $H_2(\pi/\Gamma_n)/K_j(\pi/\Gamma_n)$ in $H_2(X_n(L))/\text{im}(K_j(\pi/\Gamma_n))$. We claim the composition $H_2(\pi')\to H_2(\pi/\Gamma_n)/K_j(\pi/\Gamma_n)$ of the three arrows on the bottom row is onto. Claim 3 then follows from the \hyperref[THMsd]{Stallings-Dwyer Theorem}.

To prove this surjectivity, we consider the diagram of Mayer-Vietoris sequences seen previously. We copy a different part of this diagram below. All coefficients are in $\mathbb{Z}[\pi_1(M)]$. We will use that the leftmost and rightmost vertical maps are isomorphisms.

\begin{center}
\begin{tikzcd}
H_2(\partial E_{L'}) \arrow[r] \arrow[d,"\cong"] & H_2(E_{L'}) \arrow[r] \arrow[d] & H_2(M) \arrow[r] \arrow[d,"g_*"] & H_1(\partial E_{L'}) \arrow[d,"\cong"] \\
H_2(\partial E_L) \arrow[r] & H_2(\pi/\Gamma_n) \arrow[r] & H_2(X_n(L)) \arrow[r] & H_1(\partial E_L) \arrow[d,equal] \\
H_2(\partial E_L) \arrow[r] \arrow[u,equal] & H_2(E_L) \arrow[r] \arrow[u] & H_2(M) \arrow[u,"\iota_n(L)_*"'] \arrow[r] & H_1(\partial E_L)
\end{tikzcd}
\end{center}

A diagram chase using diagram (\ref{diagram-CO}) and the diagram above completes the proof of Claim 3. Let $x\in H_2(\pi/\Gamma_n)$. By condition (3), there exists $\alpha\in H^1(X_n(L))$ such that the images of $x$ and $\alpha$ agree in $H_2(X_n(L))/\text{im}(K_j(\pi/\Gamma_n))$. Let $\overline{x}$ and $a$ denote the respective images of these elements in $H_2(X_n(L))$. Then $\overline{x}-a\in\text{im}(K_j(\pi/\Gamma_n))$, so there exists some $y\in K_j(\pi/\Gamma_n)\leq H_2(\pi/\Gamma_n)$ whose image in $H_2(X_n(L))$ is $\overline{x}-a$. The element $x-y\in H_2(\pi/\Gamma_n)$ maps to $a\in H_2(X_n(L))$, so $a$ is in the image of $H_2(\pi/\Gamma_n)$. In particular, this implies that $\Delta(a)=0$, where $\Delta$ is the connecting homomorphism in the appropriate Mayer-Vietoris sequence. We apply the two isomorphisms at the top of diagram (\ref{diagram-CO}) to $\alpha$ to produce an element $b\in H_2(M)$ such that $g_*(b)=a$. By the rightmost vertical isomorphisms in the Mayer-Vietoris diagram, $\Delta(b)=0$, too. Thus, there exists $c\in H_2(E_{L'})$ which maps to $b$ in $H_2(M)$. Following diagram (\ref{diagram-CO}) counterclockwise from $H_2(E_{L'})$ to $H_2(\pi/\Gamma_n)$, consider the image $z$ of the element $c$. Since $x-y$ and $z$ both map to $a\in H_2(X_n(L))$, the Mayer-Vietoris sequence implies that $z-(x-y)$ is in the image of $H_2(\partial E_L)$. Using the isomorphism at the upper-left of the Mayer-Vietoris diagram, subtract from $c\in H_2(E_{L'})$ the image of this element. This yields $c'\in H_2(E_{L'})$ whose image in $H_2(\pi/\Gamma_n)$ is $z-(z-(x-y))=x-y$. Since $y\in K_j(\pi/\Gamma_n)$, $x-y$ and $x$ have the same image in $H_2(\pi/\Gamma_n)/K_j(\pi/\Gamma_n)$. 

Thus, the map $H_2(E_{L'})\to H_2(\pi/\Gamma_n)/K_j(\pi/\Gamma_n)$ is surjective, and so the map $H_2(\pi')\to H_2(\pi/\Gamma_n)/K_j(\pi/\Gamma_n)$ is, too. By the \hyperref[THMsd]{Stallings-Dwyer Theorem} with $A=\Gamma'$ and $B=\Gamma/\Gamma_n$, $g$ induces an isomorphism $\Gamma'/\Gamma'_n\xrightarrow{\cong}\Gamma/\Gamma_n$, hence an isomorphism $\pi'/\Gamma'_n\xrightarrow{\cong}\pi/\Gamma_n$ by the Five Lemma. This proves Claim 3. \vspace{1em}

To complete the proof of the theorem, we note that the isomorphism $\phi:\pi'/\Gamma'_n\xrightarrow{\cong}\pi/\Gamma_n$ we obtain from Claim 3 along with the isomorphism $\phi_\partial:H_1(\partial E_{L'})\xrightarrow{\cong} H_1(\partial E_L)$ constitutes an $n$-basing for $L'$ relative to $L$. Thus, $h_n(L',\phi)$ is defined, and \[h_n(L',\phi)=g=f\in[M,X_n(L)]_0.\]
\end{proof}

By the above proof, we observe

\begin{corollary}\label{htpychar-thmC-corpi2}
Let $L\subset M$ be an $m$-component link, and let $f\in[M,X_n(L)]_0$. Suppose that $\pi_2(M)=0$ or, more generally, that the following composition is zero: \[H_2(E_L;\mathbb{Z}[\pi_1(M)])\to H_2(M;\mathbb{Z}[\pi_1(M)])\to H_2(X_n(L);\mathbb{Z}[\pi_1(M)]).\]
If $f$ satisfies conditions (1) and (2) of Theorem~\ref{THMrealization}, then $f$ also satisfies condition (3).
\end{corollary}

This fact is particularly useful for applications of Theorem~\ref{THMrealization} to open questions like the \hyperref[CONJ]{Almost-Concordance Conjecture}, as it shows there are many instances where we do not need to check the somewhat technical condition (3).

\begin{proof}
Unless otherwise stated, all homology groups in this proof have coefficients in $\mathbb{Z}[\pi_1(M)]$. Suppose that $\pi_2(M)=0$. Then $H_2(M)=\pi_2(M)=0$. Beginning with the isomorphism $\Gamma'/\Gamma'_2\xrightarrow{\cong}\Gamma/\Gamma_2$ from the proof of Theorem~\ref{THMrealization} which did not depend on condition (3), we inductively show $\pi'/\Gamma'_k\xrightarrow{\cong}\pi/\Gamma_k$ for all $k\leq n$. 

Suppose we have an isomorphism for some $2\leq k<n$. By the final Mayer-Vietoris diagram appearing in the proof of Theorem~\ref{THMrealization}, the homomorphisms $H_2(\partial E_{L'})\to H_2(E_{L'})$ and $H_2(\partial E_L)\to H_2(E_L)$ are surjective. This implies that the images of $H_2(\pi')$ and $H_2(\pi)$ in $H_2(\pi'/\Gamma'_k)$ and $H_2(\pi/\Gamma_k)$, respectively, are equal to the images of $H_2(\partial E_{L'})$ and $H_2(\partial E_L)$ and are therefore isomorphic. By the Stallings exact sequence, we obtain an isomorphism $\Gamma'_k/\Gamma'_{k+1}\xrightarrow{\cong}\Gamma_k/\Gamma_{k+1}$, hence an isomorphism $\pi'/\Gamma'_{k+1}\xrightarrow{\cong}\pi/\Gamma_{k+1}$. This completes the induction. 

By the proof of the \hyperref[THMsd]{Stallings-Dwyer Theorem}, $H_2(\pi')\to H_2(\pi/\Gamma_n)/K_j(\pi/\Gamma_n)$ is surjective for all $j<n$. From diagram (\ref{diagram-CO}) in the proof of Theorem~\ref{THMrealization}, $H_2(E_{L'})$ surjects onto the image of $H_2(\pi/\Gamma_n)/K_j(\pi/\Gamma_n)$ in $H_2(X_n(L))/\text{im}(K_j(\pi/\Gamma_n))$. This image is 0 because the homomorphism from $H_2(E_{L'})$ factors through $H_2(M)=0$. Thus, condition (3) of Theorem~\ref{THMrealization} holds, regardless of the image of $\text{proj}\,\circ\,(\cap f_*[M])$. (In fact, the image of the map $\text{proj}\,\circ\,(\cap f_*[M])$ is also 0 because $H^1(X_n(L))\cong H^1(M)\cong H_2(M)=0$.)

The proof in the general case where the composition $H_2(E_L)\to H_2(M)\to H_2(X_n(L))$ is zero is nearly identical and leverages the same Mayer-Vietoris diagram.
\end{proof}


\section{Characterization and realization for \texorpdfstring{$\theta_n$}{thetan} and \texorpdfstring{$\overline{\mu}_n$}{mun}}\label{homchar}

In this section, we prove Theorems~\ref{THMcoker},~\ref{THMhcharacterization}, and~\ref{THMhrealization} concerning the homology concordance invariants $\theta_n$ and $\overline{\mu}_n$. Recall that Theorem~\ref{THMhcharacterization} states that the $\mathbb{Z}[G]$-homology concordance invariants $\overline{\mu}_n$ determine whether an $(n+1)$-basing exists, assuming the existence of an $n$-basing. In other words, the $G$-Milnor's invariants $\overline{\mu}_n$ determine the $G$-lower central quotients one step at a time. Theorem~\ref{THMhcharacterization} will follow quickly from Theorem~\ref{THMcoker}, which characterizes when the invariants $\theta_n$ vanish in the cokernel of $\mathcal{R}_{n+1}(L)\to\mathcal{R}_n(L)$. When $G$ is the trivial group, Theorem~\ref{THMhrealization} gives precise conditions under which elements of $H_3(X_n(L))$ are realized as (integral) homology concordance invariants of links in 3-manifolds. In this section, links are not assumed to live in the same 3-manifold. Throughout this section, given a properly embedded surface $\Sigma$ in some 4-dimensional cobordism $W$ of 3-manifolds $M$ and $M'$, we will use the notation $\partial\nu\Sigma$ to denote the intersection $\partial E_\Sigma\cap\nu\Sigma$, not the manifold boundary of $\nu\Sigma$ which also includes $(M\sqcup M')\cap\nu\Sigma$. Some of the proofs in this section are similar to arguments seen in \cite{ChaOrr}.


\subsection{Proof of Theorem~\ref{THMhcharacterization}}\label{homchar-thmD}

We prove the following theorem toward eventually proving Theorem~\ref{THMcoker} and its generalization, Theorem~\ref{THMhcharacterization}. Recall
that Theorem~\ref{THMhcharacterization} and its siblings, Theorem~\ref{THMcoker} and Corollaries~\ref{CORinductive} and~\ref{CORlift}, say that an $n$-basing lifts to an $(n+1)$-basing if and only if the corresponding $\overline{\mu}$-invariant (or $\theta$-invariant, depending on context) vanishes in an appropriate sense. Furthermore, these invariants are defined inductively and characterize when the $h$-invariants satisfy a lifting property.

\begin{theorem}\label{homchar-thmD-thmlift}
Fix an $m$-component link $L\subset M$, and let $n\geq 2$. Suppose $L'\subset M'$ and $L''\subset M''$ are $m$-component links equipped with $n$-basings $(\phi',\phi'_\partial)$ and $(\phi'',\phi''_\partial)$ over $G$ relative to $L$, so that $(\psi,\psi_\partial)=((\phi')^{-1}\circ\phi'',(\phi'_\partial)^{-1}\circ\phi''_\partial)$ is an $n$-basing over $G$ for $L''$ relative to $L'$. If $\theta^G_n(L',\phi')=\theta^G_n(L'',\phi'')\in H_3(X^G_n(L))$, then the $n$-basing $(\psi,\psi_\partial)$ lifts to an $(n+1)$-basing over $G$ for $L''$ relative to $L'$.
\end{theorem}

\begin{proof}
We suppress the group $G$ from our notation. Suppose $\theta_n(L',\phi')=\theta_n(L'',\phi'')$. Recall that $H_3(X_n(L))$ is isomorphic to the oriented bordism group $\Omega_3^{SO}(X_n(L))$. By definition, $\theta_n(L',\phi')=h_n(L',\phi')_*[M']$ and $\theta_n(L'',\phi'')=h_n(L'',\phi'')_*[M'']$. Since $\theta_n(L',\phi')=\theta_n(L'',\phi'')$, there is an oriented cobordism $W$ from $M'$ to $M''$ over the space $X_n(L)$ as in the diagram below.
\begin{center}
\begin{tikzcd}
M'' \arrow[dr, "h_n(L''{,}\phi'')"] \arrow[d, hookrightarrow] & \\
W \arrow[r, "f"] & X_n(L) \\
M' \arrow[ur, "h_n(L'{,}\phi')"'] \arrow[u, hookrightarrow] &
\end{tikzcd}
\end{center}

Taking the transverse preimage of $L\subset X_n(L)$ under the map $f$, we obtain a surface $\Sigma\subset W$ cobounded by the links $L'$ and $L''$. This surface may have more than $m$ components. We use an argument similar to the one seen in Claim 1 of $(1)\implies(2)$ in the proof of Theorem~\ref{THMcob} to reduce the number of components of $\Sigma$.

Suppose $f^{-1}(L_i)$ has more than one component for some $i$. Using the same method as the proof of Theorem~\ref{THMcob}, we attach 5-dimensional 1-handles to $W\times\{1\}\subset W\times[0,1]$ in which we connect the components of $f^{-1}(L_i)$. It is not necessary in this case, as it was in the proof of Theorem~\ref{THMcob}, to attach cancelling 2-handles. Nevertheless, we can attach cancelling 5-dimensional 2-handles and extend the map to $X_n(L)$ over these 2-handles. This produces a homotopy rel boundary $W\times[0,1]\to X_n(L)$. The preimage of $L$ under the map $W\times\{1\}\to X_n(L)$ is an $m$-component surface, which we continue to call $\Sigma$, cobounded by $L$ and $L'$. We also refer to this new map $W\to X_n(L)$ as $f$.

The pair $(E_\Sigma,\partial\nu\Sigma)$ is a cobordism of pairs between $(E_{L'},\partial E_{L'})$ and $(E_{L''},\partial E_{L''})$. We have the following diagram of pairs, obtained by removing $\nu L$ from $X_n(L)$ and its corresponding preimages under the maps $h_n(L',\phi')$, $h_n(L'',\phi'')$, and $f$:

\begin{center}
\begin{tikzcd}
(E_{L''}{,}\partial E_{L''}) \arrow[dr] \arrow[d, hookrightarrow] & \\
(E_\Sigma{,}\partial\nu\Sigma) \arrow[r] & (M_{p_n}^\times{,}\partial E_L). \\
(E_{L'}{,}\partial E_{L'}) \arrow[ur] \arrow[u, hookrightarrow] &
\end{tikzcd}
\end{center}

\noindent We postcompose with the homotopy inverse of the homotopy equivalence $h_{\phi'}:(M_{p_n'}^\times,\partial E_{L'})\xrightarrow{\simeq}(M_{p_n}^\times,\partial E_L)$ provided by Proposition~\ref{nbasing-propnbasing} induced by the $n$-basing $\phi'$ to obtain the diagram

\begin{center}
\begin{tikzcd}
(E_{L''}{,}\partial E_{L''}) \arrow[dr] \arrow[drr,"h_\psi\circ p_n''"] \arrow[d, hookrightarrow] & & \\
(E_\Sigma{,}\partial\nu \Sigma) \arrow[r] & (M_{p_n}^\times{,}\partial E_L) \arrow[r, "\simeq"] & (M_{p_n'}^\times,\partial E_{L'}). \\
(E_{L'}{,}\partial E_{L'}) \arrow[ur] \arrow[urr, "p_n'"'] \arrow[u, hookrightarrow] & &
\end{tikzcd}
\end{center}
The $n$-basing $(\psi,\psi_\partial)$ of $L''$ relative to $L'$ is now encoded in this diagram via the corresponding homotopy equivalence $h_\psi:(M_{p''_n}^\times,\partial E_{L''})\xrightarrow{\simeq}(M_{p'_n}^\times,\partial E_{L'})$. In order to lift $\psi$ to an $(n+1)$-basing, we will obtain a lift
\begin{center}
\begin{tikzcd}
& (M_{p'_{n+1}}^\times,\partial E_{L'}) \arrow[d, "p'_{n+1{,}n}"] \\
(E_{L''},\partial E_{L''}) \arrow[r,"h_\psi\circ p_n''"'] \arrow[ur, dashed] & (M_{p'_n}^\times,\partial E_{L'}).
\end{tikzcd}
\end{center}
The lift of $\partial E_{L''}$ is already determined by the isomorphism $\psi_\partial$. We therefore have a relative lifting problem
\begin{center}
\begin{tikzcd}
\partial E_{L''} \arrow[d, hookrightarrow] \arrow[r, "\psi_\partial"] & M_{p'_{n+1}}^\times \arrow[d, "p'_{n+1{,}n}"] \\
E_{L''} \arrow[r, "h_\psi\circ p_n''"'] \arrow[ur, dashed, "?\exists"] & M_{p'_n}^\times.
\end{tikzcd}
\end{center}

Since $p'_{n+1,n}$ is not necessarily a fibration, we replace $M_{p'_{n+1}}^\times$ with the mapping path space $P$ of $p'_{n+1,n}$. It suffices to find a relative lift of $h_\psi\circ p''_n$ to $P$. We may lift the 1-skeleton with no obstruction. Since $P$ is a $K(\pi'/\Gamma'_{n+1},1)$, if we can lift the 2-skeleton of $E_{L''}$, then we can also lift the 3-skeleton (i.e. all of $E_{L''}$). The map $P \to M_{p'_n}^\times$ is a fibration from a $K(\pi'/\Gamma'_{n+1},1)$ to a $K(\pi'/\Gamma'_n,1)$. Its fiber is therefore a $K(\Gamma'_n/\Gamma'_{n+1},1)$. This fibration defines a system of local coefficients given by $A=\Gamma'_n/\Gamma'_{n+1}$ as a module over $\mathbb{Z}[\pi'/\Gamma'_n]$. This system corresponds to a homomorphism $\pi'/\Gamma'_n\to\text{Aut}(A)$. Since $A$ is central in $\Gamma'/\Gamma'_{n+1}$, the restriction of this homomorphism to $\Gamma'/\Gamma'_n$ is trivial, inducing a homomorphism $G\to\text{Aut}(A)$ and giving $A$ the structure of a $\mathbb{Z}[G]$ module:
\begin{center}
\begin{tikzcd}
\Gamma'/\Gamma'_n \arrow[d, hookrightarrow] \arrow[dr,"0"] \\
\pi'/\Gamma'_n \arrow[d,twoheadrightarrow] \arrow[r] & \text{Aut}(A). \\
G \arrow[ur,dashed,"\exists"']
\end{tikzcd}
\end{center}

The obstruction to lifting the 2-skeleton of $E_{L''}$ is an element $o_{L''}\in H^2(E_{L''},\partial E_{L''};A)$, where we understand coefficients in $A$ to mean the system of local coefficients from the fibration pulled back over the appropriate spaces. This obstruction vanishes if and only if we have a relative lift of $E_{L''}$ to $P$ and, in turn, to $M_{p'_{n+1}}^\times$. Analogously, we have obstruction classes $o_{L'}$ and $o_\Sigma$ corresponding to lifting $E_{L'}$ and $E_\Sigma$. The naturality of obstruction classes gives us
\begin{center}
\begin{tikzcd}[row sep=0.5em]
H^2(E_{L'},\partial E_{L'};A) & H^2(E_\Sigma,\partial\nu \Sigma;A) \arrow[l] \arrow[r] & H^2(E_{L''},\partial E_{L''};A) \\
o_{L'} & o_\Sigma \arrow[l, mapsto] \arrow[r, mapsto] & o_{L''}.
\end{tikzcd}
\end{center}
The relative lift of $E_{L'}$ exists, as we have the inclusion map $E_{L'}\hookrightarrow M_{p'_{n+1}}^\times$. Thus, $o_{L'}=0$. It therefore suffices to prove \[ \ker\Big(H^2(E_\Sigma,\partial\nu \Sigma;A)\to H^2(E_{L'},\partial E_{L'};A)\Big)=\ker\Big(H^2(E_\Sigma,\partial\nu \Sigma;A)\to H^2(E_{L''},\partial E_{L''};A)\Big). \]

\noindent We prove this equality through a series of reductions.

First, consider the following diagram, where the horizontal arrows are Poincar\'{e} duality and each lowest vertical arrow is projection onto the first factor:

\begin{center}
\begin{tikzcd}
H^2(E_\Sigma,\partial\nu \Sigma;A) \arrow[d,"\text{inc}^*"'] \arrow[r, "\cong"]  & H_2(E_\Sigma,E_{L'}\sqcup E_{L''};A) \arrow[d,"\Delta"] \\
H^2(E_{L'}\sqcup E_{L''},\partial E_{L'}\sqcup \partial E_{L''};A) \arrow[d,equal] \arrow[r, "\cong"] & H_1(E_{L'}\sqcup E_{L''};A) \arrow[d,equal]  \\
H^2(E_{L'},\partial E_{L'};A)\oplus H^2(E_{L''},\partial E_{L''};A) \arrow[d,twoheadrightarrow] \arrow[r,"\cong"] & H_1(E_{L'};A)\oplus H_1(E_{L''};A) \arrow[d,twoheadrightarrow,"p"]  \\
H^2(E_{L'},\partial E_{L'};A) \arrow[r,"\cong"] & H_1(E_{L'};A).
\end{tikzcd}
\end{center}

\noindent As in the proof of Claim 3 of Theorem~\ref{THMrealization}, the top square commutes by the relationship between cap product and the connecting homomorphism in a Mayer-Vietoris sequence; see \cite{Spanier}. By the diagram, $\ker\big(H^2(E_\Sigma,\partial\nu \Sigma;A)\to H^2(E_{L'},\partial E_{L'};A)\big)=PD^{-1}(\ker(p\circ\Delta))$, where $PD$ is the Poincar\'{e} duality map. If $p|_{\text{im}\Delta}$ injects, then this equals $PD^{-1}(\ker\Delta)$. On the other hand, \[\ker\big(H^2(E_\Sigma,\partial\nu \Sigma;A)\to H^2(E_{L''},\partial E_{L''};A)\big)=PD^{-1}(\ker(q\circ\Delta))\supseteq PD^{-1}(\ker\Delta),\] where $q:H_1(E_{L'};A)\oplus H_1(E_{L''};A)\twoheadrightarrow H_1(E_{L''};A)$ is projection onto the second factor. Thus, the equality we wish to establish follows if both $p|_{\text{im}\Delta}$ and $q|_{\text{im}\Delta}$ inject. But
\begin{align*}
\text{im}\Delta & =\ker\Big(H_1(E_{L'}\sqcup E_{L''};A)\to H_1(E_\Sigma;A)\Big) \\
& =\Big\{\,(x,y)\in H_1(E_{L'};A)\oplus H_1(E_{L''};A)\,\Big|\,i'_*(x)+i''_*(y)=0\,\Big\},
\end{align*}
where $i'_*:H_1(E_{L'};A)\to H_1(E_\Sigma;A)$ and $i''_*:H_1(E_{L''};A)\to H_1(E_\Sigma;A)$ are induced by inclusion. Suppose $(x,y)\in\text{im}\Delta$ with $p(x,y)=x=0$. Then $i''_*(y)=0$. Similarly, if $q(x,y)=y=0$, then $i'_*(x)=0$. It therefore suffices to prove that $i'_*$ and $i''_*$ inject.

Similar to the proof of Theorem~\ref{THMrealization}, we invoke the (homology) Universal Coefficient Spectral Sequence, a homology spectral sequence with $E^2_{p,q}=\text{Tor}^{\mathbb{Z}[G]}_p(H_q(-;\mathbb{Z}[G]),A)$ which converges to $H_*(-;A)$; see \cite{LevineJ77} Theorem 2.3 or \cite{Rotman} Theorem 11.34. From this spectral sequence, we obtain the following natural short exact sequence akin to the Universal Coefficient Theorem for homology:
\[0\to \frac{H_1(-;\mathbb{Z}[G])\otimes_{\mathbb{Z}[G]} A}{\text{im}(\text{Tor}_2(H_0(-;\mathbb{Z}[G]),A))}\to H_1(-;A)\to \text{Tor}_1(H_0(-;\mathbb{Z}[G]),A)\to 0.\]

The composition $E_{L'}\to M_{p'_n}^\times\xrightarrow{\simeq} M_{p_n}^\times$ induces maps of spectral sequences, and therefore a diagram of short exact sequences of the type seen above. The isomorphisms \[H_0(E_{L'};\mathbb{Z}[G])\xrightarrow{\cong}H_0(\pi'/\Gamma'_n;\mathbb{Z}[G])\xrightarrow{\cong}H_0(\pi/\Gamma_n;\mathbb{Z}[G])\] and \[H_1(E_{L'};\mathbb{Z}[G])\xrightarrow{\cong}H_1(\pi'/\Gamma'_n;\mathbb{Z}[G])\xrightarrow{\cong}H_1(\pi/\Gamma_n;\mathbb{Z}[G])\] imply that $H_1(E_{L'};A)\xrightarrow{\cong}H_1(\pi/\Gamma_n;A)$ is an isomorphism. Note that we use $n\geq 2$ to obtain the isomorphism $H_1(E_{L'};\mathbb{Z}[G])\xrightarrow{\cong}H_1(\pi/\Gamma_n;\mathbb{Z}[G])$. The isomorphism $H_1(E_{L'};A)\xrightarrow{\cong}H_1(\pi/\Gamma_n;A)$ factors as
\begin{center}
\begin{tikzcd}
H_1(E_{L'};A) \arrow[r,"\cong"] \arrow[d,"i'_*"] & H_1(\pi/\Gamma_n;A). \\
H_1(E_\Sigma;A) \arrow[ur]
\end{tikzcd}
\end{center}
Thus, $i'_*$ injects. To show $i''_*$ injects, we repeat the argument and apply the spectral sequence to $E_{L''}$. 

The injectivity of $i'_*$ and $i''_*$ implies \[ \ker\Big(H^2(E_\Sigma,\partial\nu \Sigma;A)\to H^2(E_{L'},\partial E_{L'};A)\Big)=\ker\Big(H^2(E_\Sigma,\partial\nu \Sigma;A)\to H^2(E_{L''},\partial E_{L''};A)\Big).\] Since the obstruction class $o_{L'}$ vanishes, the obstruction class $o_{L''}$ vanishes. We obtain a lift of $h_\psi\circ p''_n:E_{L''}\to M_{p'_n}^\times$ to $M_{p'_{n+1}}^\times$ rel $\psi_\partial$, that is, a map of pairs $(E_{L''},\partial E_{L''})\to(M_{p'_{n+1}},\partial E_{L'})$. 

This map factors as 
\begin{center}
\begin{tikzcd}
(E_{L''},\partial E_{L''}) \arrow[r] \arrow[d] & (M_{p'_{n+1}}^\times,\partial E_{L'}), \\
(M_{p''_{n+1}}^\times,\partial E_{L''}) \arrow[ur, "\widetilde{h}"']
\end{tikzcd}
\end{center}
\noindent as the obstruction to such a factorization vanishes: The homomorphism $\pi''\twoheadrightarrow\pi'/\Gamma'_{n+1}$ factors through $\pi''/\Gamma''_{n+1}$ and yields a homomorphism $\widetilde{\psi}:\pi''/\Gamma''_{n+1}\to\pi'/\Gamma'_{n+1}$ lifting $\psi$. We wish to show $(\widetilde{\psi},\psi_\partial)$ is an $(n+1)$-basing for $L''$ relative to $L'$ which is a lift of $(\psi,\psi_\partial)$. By Proposition~\ref{nbasing-propnbasing}, it suffices to show that $\widetilde{\psi}:\pi''/\Gamma''_{n+1}\to\pi'/\Gamma'_{n+1}$ is an isomorphism, so that $\widetilde{h}$ is a homotopy equivalence of pairs.

To do so, we reverse the roles of $L'$ and $L''$ and instead consider the diagram
\begin{center}
\begin{tikzcd}
(E_{L''}{,}\partial E_{L''}) \arrow[dr] \arrow[drr,"p_n''"] \arrow[d, hookrightarrow] & & \\
(E_\Sigma{,}\partial\nu \Sigma) \arrow[r] & (M_{p_n}^\times{,}\partial E_L) \arrow[r, "\simeq"] & (M_{p_n''},\partial E_{L''}). \\
(E_{L'}{,}\partial E_{L'}) \arrow[ur] \arrow[urr, "h_{\psi^{-1}}\circ p_n'"'] \arrow[u, hookrightarrow] & &
\end{tikzcd}
\end{center}
We solve the relative lifting problem
\begin{center}
\begin{tikzcd}
\partial E_{L'} \arrow[d, hookrightarrow] \arrow[r, "\psi_\partial^{-1}"] & M_{p''_{n+1}}^\times \arrow[d, "p''_{n+1{,}n}"] \\
E_{L'} \arrow[r, "h_{\psi^{-1}}\circ p_n'"'] \arrow[ur, dashed, "?\exists"] & M_{p''_n}^\times
\end{tikzcd}
\end{center}
and obtain a homomorphism $\widetilde{\psi}':\pi'/\Gamma'_{n+1}\to\pi''/\Gamma''_{n+1}$ lifting $\psi^{-1}$. 

We now show that $\widetilde{\psi}$ and $\widetilde{\psi}'$ are mutually inverse. Consider the following diagram of short exact sequences. Note that $\widetilde{\psi}\circ\widetilde{\psi}'$ is a lift of $\psi\circ\psi^{-1}=\text{id}_{\pi'/\Gamma'_n}$.
\begin{center}
\begin{tikzcd}
1 \arrow[r] & \Gamma'_n/\Gamma'_{n+1} \arrow[r, hookrightarrow] \arrow[d] & \pi'/\Gamma'_{n+1} \arrow[r, twoheadrightarrow] \arrow[d,"\widetilde{\psi}\circ\widetilde{\psi}'"] & \pi'/\Gamma'_n \arrow[d, equal] \arrow[r] & 1 \\
1 \arrow[r] & \Gamma'_n/\Gamma'_{n+1} \arrow[r, hookrightarrow] & \pi'/\Gamma'_{n+1} \arrow[r,twoheadrightarrow] & \pi'/\Gamma'_n \arrow[r] & 1
\end{tikzcd}
\end{center}
To show $\widetilde{\psi}\circ\widetilde{\psi}'$ is the identity, it suffices to show $\widetilde{\psi}\circ\widetilde{\psi}'$ is the identity on $\Gamma'_n/\Gamma'_{n+1}$. Since $\widetilde{\psi}\circ\widetilde{\psi}'$ is a lift of $\text{id}_{\pi'/\Gamma'_n}$, it is a homomorphism over $G$ and therefore sends $\Gamma'/\Gamma'_{n+1}$ to itself. We show $\widetilde{\psi}\circ\widetilde{\psi}'$ is, in fact, the identity on $\Gamma'_2/\Gamma'_{n+1}$.

Let $g\in\Gamma'_2/\Gamma'_{n+1}$. Then $g$ may be written as $g=\prod_i[a_i,b_i]$, where $a_i,b_i\in\Gamma'/\Gamma'_{n+1}$. Since $\widetilde{\psi}\circ\widetilde{\psi}'$ is a lift of $\text{id}_{\pi'/\Gamma'_n}$, we have $\widetilde{\psi}\circ\widetilde{\psi}'(a_i)=a_iu_i$ and $\widetilde{\psi}\circ\widetilde{\psi}'(b_i)=b_iv_i$ for some $u_i,v_i\in\Gamma'_n/\Gamma'_{n+1}$. But $\Gamma'_n/\Gamma'_{n+1}$ is central in $\Gamma'/\Gamma'_{n+1}$, so $[a_iu_i,b_iv_i]=[a_i,b_i]$ for all $i$. Thus, \[ (\widetilde{\psi}\circ\widetilde{\psi}')(g)=\prod_i\big[(\widetilde{\psi}\circ\widetilde{\psi}')(a_i),(\widetilde{\psi}\circ\widetilde{\psi}')(b_i)\big]=\prod_i[a_iu_i,b_iv_i]=\prod_i[a_i,b_i]=g.\] It follows that $\widetilde{\psi}\circ\widetilde{\psi}'=\text{id}_{\pi'/\Gamma'_{n+1}}$. 

To complete the proof, reverse the roles of $\widetilde{\psi}$ and $\widetilde{\psi}'$ to see that $\widetilde{\psi}'=\widetilde{\psi}^{-1}$. Thus, $(\widetilde{\psi},\psi_\partial)$ constitutes an $(n+1)$-basing for $L''$ relative to $L'$ over $G$ which is a lift of the $n$-basing $(\psi,\psi_\partial)$.
\end{proof}

We now apply Theorem~\ref{homchar-thmD-thmlift} to characterize the notion of ``vanishing in the cokernel". Recall that Theorem~\ref{THMcoker} asserts that the invariant $\theta_n(L')$ relative to $L$ vanishes in the cokernel of $\mathcal{R}_{n+1}(L)\to\mathcal{R}_n(L)/\text{Aut}(\pi/\Gamma_n,\partial)$ if and only if $L'$ admits some $(n+1)$-basing relative to $L$.

\begin{proof}[Proof of Theorem~\ref{THMcoker}]
We again suppress the fixed group $G$. Fix $L\subset M$, and suppose $L'\subset M'$ admits an $n$-basing $(\phi,\phi_\partial)$ relative to $L$. The proof of this theorem follows quickly from Theorem~\ref{homchar-thmD-thmlift}. 

Suppose $L'$ admits an $(n+1)$-basing $\widetilde{\phi}$ over $G$ relative to $L$ which is a lift of $\phi$. The invariant $\theta_{n+1}(L',\widetilde{\phi})$ is sent to $\theta_n(L',\phi)$ under the map $\mathcal{R}_{n+1}(L)\to\mathcal{R}_n(L)$. Thus, by definition, $\theta_n(L',\phi)$ vanishes in the cokernel of $\mathcal{R}_{n+1}(L)\to\mathcal{R}_n(L)$.

Conversely, suppose $\theta_n(L',\phi)$ vanishes in the cokernel of $\mathcal{R}_{n+1}(L)\to\mathcal{R}_n(L)$. Then there exists a 3-manifold $(M'',\varphi'')$ with $\varphi'':\pi_1(M'')\twoheadrightarrow G$ and a link $L''\subset M''$ equipped with an $(n+1)$-basing $\widetilde{\phi'}$ over $G$ relative to $L$ such that $\theta_{n+1}(L'',\widetilde{\phi'})$ is sent to $\theta_n(L',\phi)$ under the map $\mathcal{R}_{n+1}(L)\to\mathcal{R}_n(L)$. Let $\phi'$ be the $n$-basing for $L''$ induced by the $(n+1)$-basing $\widetilde{\phi'}$. Then $\theta_n(L'',\phi')=\theta_n(L',\phi)$. By Theorem~\ref{homchar-thmD-thmlift}, the $n$-basing $\psi=(\phi')^{-1}\circ\phi$ of $L'$ relative to $L''$ lifts to an $(n+1)$-basing $\widetilde{\psi}$. We obtain an $(n+1)$-basing $\widetilde{\phi}$ for $L'$ relative to $L$ which is a lift of $\phi$ by composing with the $(n+1)$-basing $\widetilde{\phi'}$. Define $\widetilde{\phi}$ to be the $(n+1)$-basing $\widetilde{\phi'}\circ\widetilde{\psi}$. This proves the equivalence of (1) and (2).

We now prove the equivalence of the analogous basing-independent statements (1$'$) and (2$'$). Suppose $L'$ admits an $(n+1)$-basing $\widetilde{\phi'}$ over $G$ relative to $L$. Let $\phi'$ be the $n$-basing induced by $\widetilde{\phi'}$. Note that $\psi:=\phi\circ(\phi')^{-1}$ is a self-$n$-basing of $L$. Then $\theta_{n+1}(L',\widetilde{\phi'})$ maps to $\theta_n(L',\phi')$ under $\mathcal{R}_{n+1}(L)\to\mathcal{R}_n(L)$, and $\theta_n(L',\phi)=\psi\cdot\theta_n(L',\phi')$, that is, $\theta_n(L')$ vanishes in the cokernel of $\mathcal{R}_{n+1}(L)\to\mathcal{R}_n(L)/\textnormal{Aut}(\pi/\Gamma_n,\partial)$.

Conversely, suppose $\theta_n(L')$ vanishes in the cokernel of $\mathcal{R}_{n+1}(L)\to\mathcal{R}_n(L)/\textnormal{Aut}(\pi/\Gamma_n,\partial)$. Then there exists a self-$n$-basing $\psi$ for $L$ such that $\psi\cdot\theta_n(L',\phi)=\theta_n(L',\psi\circ\phi)$ is in the image of $\mathcal{R}_{n+1}(L)$. By statement (2) of this theorem, $L'$ admits an $(n+1)$-basing $\widetilde{\phi'}$ relative to $L$ which is a lift of $\psi\circ\phi$.
\end{proof}

Before proceeding to the proof of Theorem~\ref{THMhcharacterization}, we recall the definition of the equivalence relation $\sim$ defined on $\mathcal{R}_n(L)$. For $\theta\in\mathcal{R}_n(L)$, there exists a link $L'$ in some 3-manifold $M'$ with $\pi_1(M')\twoheadrightarrow G$ which admits an $n$-basing $\phi$ over $G$ relative to $L$ such that $\theta_n(L',\phi)=\theta\in H_3(X_n(L))$. The set $I_\theta$ is the image of the composition $\mathcal{R}_{n+1}(L')\to\mathcal{R}_n(L')\xrightarrow[\phi_*]{\cong}\mathcal{R}_n(L)$. We say $\theta\sim\theta'$ if $\theta'\in I_\theta$.

\begin{lemma} \label{homchar-thmD-lemmaequiv1}
The set $I_\theta$ is well-defined and $I_\theta=I_{\theta'}$ if and only if $\theta'\in I_\theta$, that is, the sets $I_\theta$ form a partition of $\mathcal{R}_n(L)$ and $\sim$ is a well-defined equivalence relation.
\end{lemma}

\begin{proof}
Let $\theta\in \mathcal{R}_n(L)$ with $\theta=\theta_n(L',\phi)$ for some link $L'$ in some 3-manifold $M'$ which admits the $n$-basing $\phi$ relative to $L$. Suppose $\theta'\in I_\theta$. In particular this means $\theta'\in\mathcal{R}_n(L)$, so that $\theta'=\theta_n(L'',\phi')$ for some link $L''\subset M''$. Then $\phi^{-1}\circ\phi'$ is an $n$-basing for $L''$ relative to $L'$. Since $\theta'\in I_\theta$, we have that $\theta_n(L'',\phi^{-1}\circ\phi')\in\text{im}\big(\mathcal{R}_{n+1}(L')\to\mathcal{R}_n(L')\big)$. By Theorem~\ref{THMcoker}, the $n$-basing $\phi^{-1}\circ\phi'$ lifts to an $(n+1)$-basing $\psi$ for $L''$ relative to $L'$. 

Consider the following diagram, where the horizontal arrows are bijections of sets:
\begin{center}
\begin{tikzcd}
\mathcal{R}_{n+1}(L'') \arrow[rr,"\psi_*"', "\cong"] \arrow[d] & & \mathcal{R}_{n+1}(L') \arrow[d] \\
\mathcal{R}_n(L'') \arrow[r, "\phi'_*"', "\cong"] & \mathcal{R}_n(L) & \mathcal{R}_n(L'). \arrow[l,"\phi_*", "\cong"']
\end{tikzcd}
\end{center}
To prove $I_{\theta'}=I_\theta$, we need to see that this diagram commutes. Suppose $\theta''\in\mathcal{R}_{n+1}(L'')$ with $\theta''=\theta_{n+1}(L''',\widetilde{\phi''})$, where $\widetilde{\phi''}$ is an $(n+1)$-basing for $L'''$ relative to $L''$ inducing the $n$-basing $\phi''$. Considering the images of $\theta_{n+1}(L''',\widetilde{\phi''})$ under the maps in the diagram above, we have

\begin{center}
\begin{tikzcd}
\theta_{n+1}(L''',\widetilde{\phi''}) \arrow[rr, mapsto, "\psi_*"] \arrow[d, mapsto] & & \theta_{n+1}(L''',\psi\circ\widetilde{\phi''}) \arrow[d, mapsto] \\
\theta_n(L''',\phi'') \arrow[r, mapsto, "\phi'_*"] & \theta_n(L''',\phi'\circ\phi'') & \theta_n(L''',\phi^{-1}\circ\phi'\circ\phi''). \arrow[l, mapsto, "\phi_*"']
\end{tikzcd}
\end{center}

\noindent Thus, $I_{\theta'}=I_\theta$. Taking $\theta'=\theta$ shows that the set $I_\theta$ is well-defined, independent of our choice of $(L',\phi)$.
\end{proof}

Lemma~\ref{homchar-thmD-lemmaequiv1} allows us to prove a generalization of Theorem~\ref{THMcoker} involving the equivalence relation $\sim$. Instead of comparing one link $L'$ to a fixed link $L$, we can again compare links $L'$ and $L''$ \emph{over the fixed link $L$}.

\begin{corollary}\label{homchar-thmD-corequiv}
Fix an $m$-component link $L\subset M$, and let $n\geq 2$. Suppose $L'\subset M'$ and $L''\subset M''$ are $m$-component links equipped with $n$-basings $(\phi',\phi'_\partial)$ and $(\phi'',\phi''_\partial)$ over $G$ relative to $L$, so that $(\psi,\psi_\partial)=((\phi')^{-1}\circ\phi'',(\phi'_\partial)^{-1}\circ\phi''_\partial)$ is an $n$-basing over $G$ for $L''$ relative to $L'$. Then the $n$-basing $(\psi,\psi_\partial)$ lifts to an $(n+1)$-basing for $L''$ over $G$ relative to $L'$ if and only if $\theta_n(L',\phi')\sim\theta_n(L'',\phi'')$ in $\mathcal{R}_n(L)$.
\end{corollary}

\noindent This is a generalization of Theorem~\ref{THMcoker} because $\theta\sim\theta_n(L,\text{id})$ if and only if $\theta$ vanishes in the cokernel of $\mathcal{R}_{n+1}(L)\to\mathcal{R}_n(L)$.

\begin{proof}
By definition, $\theta_n(L'',\phi'')\in I_{\theta_n(L',\phi')}$ if and only if $\theta_n(L'',\psi)=\theta_n(L'',(\phi')^{-1}\circ\phi'')$ vanishes in the cokernel of $\mathcal{R}_{n+1}(L')\to\mathcal{R}_n(L')$. By Theorem~\ref{THMcoker}, this happens if and only if $\psi$ lifts to an $(n+1)$-basing of $L''$ relative to $L'$.
\end{proof}

Recall that the Milnor invariant of length $n$ of the pair $(L',\phi)$, where $\phi$ is an $n$-basing over $G$ relative to $L$, is the equivalence class $\overline{\mu}_n(L',\phi)=[\theta_n(L',\phi)]$ of the $\theta$-invariant under the equivalence relation $\sim$. The Milnor invariant of length $n$ of the link $L'$ is the equivalence class $\overline{\mu}_n(L')=[\theta_n(L',\phi)]$ under the equivalence relation $\approx$, where we further quotient by the action of $\text{Aut}(\pi/\Gamma_n,\partial)$.

We now prove Theorem~\ref{THMhcharacterization} which generalizes Theorem~\ref{THMcoker}. We leverage this theorem in Section~\ref{previous} to show that the vanishing of the lower central homotopy invariants, taken in appropriate contexts, implies the vanishing of all previous versions of Milnor's invariants in dimension 3.

\begin{proof}[Proof of Theorem~\ref{THMhcharacterization}]
The equivalence of statements (1)-(3) and the equivalence of statements (1$'$)-(3$'$) are the conclusions of Theorem~\ref{THMcoker}. The equivalence of statements (4) and (5), which compare links $L'\subset M'$ and $L''\subset M''$ relative to the fixed link $L\subset M$, is a rewriting of the conclusion of Corollary~\ref{homchar-thmD-corequiv}. Finally, we prove that the analogous basing-independent statements (4$'$) and (5$'$) are equivalent. 

Let $L\subset M$, $L'\subset M'$, $L''\subset M''$, $\phi'$, and $\phi''$ be as in Corollary~\ref{homchar-thmD-corequiv}. Suppose there exists an $(n+1)$-basing $\widetilde{\psi}$ of $L''$ relative to $L'$. Let $\psi$ be the induced $n$-basing, so that $\psi=(\phi')^{-1}\circ(\phi'\circ\psi)$ lifts. By Corollary~\ref{homchar-thmD-corequiv}, $\theta_n(L',\phi')\sim\theta_n(L'',\phi'\circ\psi)$. Equivalently, $\theta_n(L',\phi')\sim((\phi'\circ\psi)\circ(\phi'')^{-1})\cdot\theta_n(L'',\phi'')$, that is, $\overline{\mu}_n(L')=\overline{\mu}(L'')$. Conversely, suppose $\overline{\mu}_n(L')=\overline{\mu}(L'')$. Then there exists a self-$n$-basing $\psi$ of $L$ such that $\psi\cdot\theta_n(L'',\phi'')\sim\theta_n(L',\phi')$, that is, $\theta_n(L'',\psi\circ\phi'')\sim\theta_n(L',\phi')$. We obtain the desired $(n+1)$-basing for $L''$ relative to $L'$ from Corollary~\ref{homchar-thmD-corequiv}.
\end{proof}


\subsection{Proof of Theorem~\ref{THMhrealization}}\label{homchar-thmE}

For the remainder of this section, let $G$ be the trivial group. As $G$ will indicate the trivial group throughout, we will write $\mathcal{R}_n(L)$, $X_n(L)$, $h_n(L',\phi)$, and $\theta_n(L',\phi)$ when we mean $\mathcal{R}^G_n(L)$, $X^G_n(L)$, $h^G_n(L',\phi)$, and $\theta^G_n(L',\phi)$. 

Recall that Theorem~\ref{THMhrealization} characterizes which elements of $H_3(X_n(L))$ can be realized as $\theta_n(L',\phi)$ for some link $L'$ in some 3-manifold $M'$ admitting an $n$-basing relative to $L$ over $G$ (i.e. in the case of integral homology concordance). See Section~\ref{results-homreal} for the full statement. Theorem~\ref{THMhrealization} relies heavily on a 3-dimensional homology surgery theorem of V. Turaev \cite{Turaev84}. In the case where $G\neq\{1\}$, a realization theorem for $\theta$-invariants should follow from similar arguments, but one would first need to prove a $\mathbb{Z}[G]$-coefficient version of Turaev's homology surgery result. This is a topic of future consideration. We copy Turaev's result below as Lemma~\ref{homchar-thmE-lemmaturaev}.

\begin{lemma}[Turaev \cite{Turaev84}, Lemma 2.2]\label{homchar-thmE-lemmaturaev}
Suppose $g:N\to X$ is a map of a closed oriented 3-manifold $N$ to a CW-complex $X$ with finitely generated $\pi_1(X)$ such that the cap product \[\cap\, g_*[N]:tH^2(X)\to tH_1(X)\] is an isomorphism. Then $(N,g)$ is oriented bordant over $X$ to a pair $(M,f)$ of a closed oriented 3-manifold $M$ and a map $f:M\to X$ which induces an isomorphism $f_*:H_1(M)\xrightarrow{\cong} H_1(X)$.
\end{lemma}

Before proving Theorem~\ref{THMhrealization}, we briefly remark that conditions (2) and (3) in the statement of the theorem can be viewed as Poincar\'{e} duality restrictions imposed on the homology class $\theta\in H_3(X_n(L))$. The utility of condition (2) in the following proof is to apply Turaev's theorem, Lemma~\ref{homchar-thmE-lemmaturaev}. A homology class $\theta\in H_3(X_n(L))$ will yield a bordism class of 3-manifolds over $X_n(L)$ which we modify to ensure the desired homology. See Section~\ref{results-homreal} for a comparison of the statement of Theorem~\ref{THMhrealization} with the statement of Theorem~\ref{THMrealization}, the realization theorem for $h$-invariants.

\begin{proof}[Proof of Theorem~\ref{THMhrealization}]
($\implies$). Suppose $\theta\in\mathcal{R}_n(L)$. Then there exists a 3-manifold $M'$ and an $m$-component link $L'\subset M'$, where $L'$ admits an $n$-basing $(\phi,\phi_\partial)$ relative to $L$ such that $\theta_n(L',\phi)=\theta$. In other words, $\theta=h_n(L',\phi)_*[M']$. For simplicity of notation, let $f$ denote the map $h_n(L',\phi)$. First, since $f$ is a map of pairs $(M',\nu L')\to(X_n(L),\nu L)$ restricting to an admissible homeomorphism $\partial E_{L'}\xrightarrow{\cong}\partial E_L$, $\Delta(\theta)=\Delta f_*[M']=\sum_i[T_i]$, where $\Delta$ is the connecting homomorphism in the Mayer-Vietoris sequence corresponding to the decomposition $X_n(L)=M_{p_n}^\times\cup_{\partial E_L} \nu L$, and where $\partial E_{L}=\sqcup_i T_i$. Thus, condition (1) holds.

Next, consider the commutative square below.
\begin{center}
\begin{tikzcd}
tH^2(M') \arrow[d,"\cap\, {[M']}"'] & tH^2(X_n(L)) \arrow[l,"f^*"'] \arrow[d,"\cap\, \theta"] \\
tH_1(M') \arrow[r,"f_*"'] & tH_1(X_n(L))
\end{tikzcd}
\end{center}
The cap product $\cap\,[M']$ is an isomorphism by Poincar\'{e} duality, and \[\pi_1(X_n(L))\cong\pi_1(M)/\pi_1(M)_n\cong\pi_1(M')/\pi_1(M')_n\] (see Proposition~\ref{appendix-proppi1G} in Section~\ref{appendix-alg} of the Appendix). Thus, $f_*:H_1(M')\to H_1(X_n(L))$ and its restriction $tH_1(M')\to tH_1(X_n(L))$ are isomorphisms. But $tH^2(-)\cong\text{Ext}(H_1(-),\mathbb{Z})$, so $f^*:tH^2(X_n(L))\to tH^2(M')$ is also an isomorphism. By the diagram, $\cap\,\theta$ is an isomorphism, too. This proves condition (2) holds.

To prove condition (3) holds, we consider diagram (\ref{diagram-CO2}) below, where all coefficients are in $\mathbb{Z}$. Diagram (\ref{diagram-CO2}) is similar to diagram (\ref{diagram-CO}) seen in the proof of Theorem~\ref{THMrealization}.

\begin{equation*}
\begin{tikzcd}
& H^1(M') \arrow[d,"\cap{[M']}"',"\cong"] & H^1(X_n(L)) \arrow[l,"f^*"',"\cong"] \arrow[d,"\cap\, \theta"] \arrow[dr,"\text{proj}\circ(\cap\,\theta)"] & \\
H_2(E_{L'}) \arrow[r,"j'_*"] \arrow[d,twoheadrightarrow] & H_2(M') \arrow[r,"f_*"] & H_2(X_n(L)) \arrow[r,twoheadrightarrow,"\text{proj}"] & H_2(X_n(L))/\text{im}(K_j(\pi/\pi_n)) \\
H_2(\pi') \arrow[r] & H_2(\pi'/\pi'_n) \arrow[r,"\phi_*","\cong"'] & H_2(\pi/\pi_n) \arrow[r,twoheadrightarrow,"\text{proj}"] \arrow[u] & H_2(\pi/\pi_n)/K_j(\pi/\pi_n) \arrow[u]
\end{tikzcd} \tag{$\ast\ast$} \label{diagram-CO2}
\end{equation*}

\noindent The map $f^*:H^1(X_n(L))\to H^1(M')$ is an isomorphism since $H^1(-)\cong\text{Hom}(H_1(-),\mathbb{Z})$ and $f_*:H_1(M')\xrightarrow{\cong} H_1(X_n(L))$ is an isomorphism. By the \hyperref[THMsd]{Stallings-Dwyer Theorem}, $H_2(\pi')\to H_2(\pi/\pi_n)/K_j(\pi/\pi_n)$ is surjective since $\phi:\pi'/\pi'_n\xrightarrow{\cong}\pi/\pi_n$ is an isomorphism. Just as in Theorem~\ref{THMrealization}, this implies the image of $\text{proj}\circ(\cap\,\theta)$ contains the image of $H_2(\pi/\pi_n)/K_j(\pi/\pi_n)$ for all $j<n$. This proves condition (3) holds. \vspace{1em}

\noindent ($\impliedby$). The converse combines Turaev's result, Lemma~\ref{homchar-thmE-lemmaturaev}, with ingredients from the proof of Theorem~\ref{THMrealization}. When possible, instead of repeating arguments from the proof of Theorem~\ref{THMrealization}, we will reference that proof and sketch analogues of the arguments in this setting.

Suppose conditions (1)-(3) hold for some class $\theta\in H_3(X_n(L))$. Since $H_3(-)\cong\Omega_3^{\text{SO}}(-)$, there exists a 3-manifold $M''$ and a map $g:M''\to X_n(L)$ such that $g_*[M'']=\theta$. We first invoke the homology surgery result of Turaev to alter $(M'',g)$. We then use the transverality techniques and arguments seen in the proof of Theorem~\ref{THMrealization} to produce the desired link realizing the class $\theta$.

By Turaev's result and condition (2), $(M'',g)$ is oriented bordant over $X_n(L)$ to some $(M',f)$, where $M'$ is a closed orientable 3-manifold and $f:M'\to X_n(L)$ induces an isomorphism on $H_1(-)$. Since $(M',f)$ is oriented bordant to $(M'',g)$, we have $f_*[M']=g_*[M'']=\theta$. As in the proof of Theorem~\ref{THMrealization}, we may assume that the image of the map $f$ contains $L$. Define $L'\subset M'$ to be the transverse preimage $f^{-1}(L)$ of $L$. \vspace{1em}

\noindent Claim 1. The restriction of $f$ to $E_{L'}$ induces a surjective homomorphism $\pi'/\pi'_n\to \pi/\pi_n$. \vspace{1em}

We consider a Mayer-Vietoris diagram similar to one seen in the proof of Theorem~\ref{THMrealization}.
\begin{center}
\begin{tikzcd}
H_1(\partial E_{L'}) \arrow[r] \arrow[d,twoheadrightarrow] & H_1(E_{L'})\oplus H_1(\nu L') \arrow[r,twoheadrightarrow] \arrow[d] & H_1(M') \arrow[r] \arrow[d,"\cong"] & 0 \\
H_1(\partial E_L) \arrow[r] & H_1(\pi/\pi_n)\oplus H_1(\nu L) \arrow[r,twoheadrightarrow] & H_1(X_n(L)) \arrow[r] & 0
\end{tikzcd}
\end{center}
The middle vertical arrow splits as a direct sum of homomorphisms. A diagram chase proves $\pi'/\pi'_2=H_1(E_{L'})\to H_1(\pi/\pi_n)=\pi/\pi_2$ is surjective. As in the proof of Theorem~\ref{THMrealization}, Claim 1 now follows from a technique used in the proof of the \hyperref[THMsd]{Stallings-Dwyer Theorem}; see \cite{ChaOrr}. \vspace{1em}

\noindent Claim 2. We may reduce the number of components of $L'$ without altering the homotopy class of the map $f$ such that the resulting link, which we continue to call $L'$, has $m$ components and maps homeomorphically to $L$. \vspace{1em}

The proof of Claim 2 is entirely similar that of Claim 2 in the proof of Theorem~\ref{THMrealization} and requires condition (1); see also Claim 1 of $(1)\implies (2)$ from the proof of Theorem~\ref{THMcob}. \vspace{1em}

\noindent Claim 3. The map $f$ induces an isomorphism $\pi'/\pi'_n\xrightarrow{\cong}\pi/\pi_n$. \vspace{1em}

The proof of Claim 3 proceeds in the same way as the proof of Claim 3 of Theorem~\ref{THMrealization}. We briefly summarize the argument. 

We first show $\pi'/\pi'_2\to\pi/\pi_2$ is an isomorphism. Consider the diagram of Mayer-Vietoris sequences below.
\begin{center}
\begin{tikzcd}
H_2(M') \arrow[r] \arrow[d,"f_*"] & H_1(\partial E_{L'}) \arrow[d,"\cong"] \arrow[r] & H_1(E_{L'})\oplus H_1(\nu L') \arrow[r] \arrow[d] & H_1(M') \arrow[d,"\cong"] \arrow[r] & 0 \\
H_2(X_n(L)) \arrow[r] & H_1(\partial E_L) \arrow[r] & H_1(\pi/\pi_n)\oplus H_1(\nu L) \arrow[r] & H_1(X_n(L)) \arrow[r] & 0 \\
H_2(M) \arrow[r] \arrow[u,"\iota_n(L)_*"'] & H_1(\partial E_L) \arrow[r] \arrow[u,equal] & H_1(E_L)\oplus H_1(\nu L) \arrow[u,"\cong"'] \arrow[r] & H_1(M) \arrow[u,"\cong"'] \arrow[r] & 0
\end{tikzcd}
\end{center}
Using this diagram, it suffices to prove that the composition $H_1(\partial E_{L'})\xrightarrow{\cong} H_1(\partial E_L)$ of the isomorphisms in the second column from the left, induced by restrictions of $f$ and $\iota_n(L)$, sends $\text{im}\big(H_2(M')\to H_1(\partial E_{L'})\big)$ isomorphically to $\text{im}\big(H_2(M)\to H_1(\partial E_L)\big)$. We accomplish this by appealing to the commutative diagram below. The desired isomorphism $\pi'/\pi'_2\to\pi/\pi_2$ then follows from a diagram chase.

\begin{center}
\begin{tikzcd}
H^1(M') \arrow[r,"\cap{[M']}","\cong"'] & H_2(M') \arrow[r,"\Delta"] & H_1(\partial E_{L'}) \arrow[d,"\cong"] \\
H^1(X_n(L)) \arrow[u,"f^*","\cong"'] \arrow[d,"\iota_n(L)^*"',"\cong"] & & H_1(\partial E_L) \arrow[d,equal] \\
H^1(M) \arrow[r,"\cap{[M]}","\cong"'] & H_2(M) \arrow[r,"\Delta"] & H_1(\partial E_L)
\end{tikzcd}
\end{center}

To complete the proof of Claim 3, we again consider diagram (\ref{diagram-CO2}) above. By assumption (3), the image of $\text{proj}\circ(\cap\,\theta)=\text{proj}\circ(\cap f_*[M'])$ contains the image of $H_2(\pi/\pi_n)/K_j(\pi/\pi_n)$ in $H_2(X_n(L))/\text{im}(K_j(\pi/\pi_n))$. The surjectivity of $H_2(\pi')\to H_2(\pi/\pi_n)/K_j(\pi/\pi_n)$ then follows from diagram (\ref{diagram-CO2}) and a different part of the above diagram of Mayer-Vietoris sequences, copied below. After a diagram chase similar to the proof of Theorem~\ref{THMrealization}, Claim 3 follows from the \hyperref[THMsd]{Stallings-Dwyer Theorem}.

\begin{center}
\begin{tikzcd}
H_2(\partial E_{L'}) \arrow[r] \arrow[d,"\cong"] & H_2(E_{L'}) \arrow[r] \arrow[d] & H_2(M') \arrow[r] \arrow[d,"f_*"] & H_1(\partial E_{L'}) \arrow[d,"\cong"] \\
H_2(\partial E_L) \arrow[r] & H_2(\pi/\pi_n) \arrow[r] & H_2(X_n(L)) \arrow[r] & H_1(\partial E_L) \arrow[d,equal] \\
H_2(\partial E_L) \arrow[r] \arrow[u,equal] & H_2(E_L) \arrow[r] \arrow[u] & H_2(M) \arrow[u,"\iota_n(L)_*"'] \arrow[r] & H_1(\partial E_L)
\end{tikzcd}
\end{center}

To complete the proof of the theorem, we note that the isomorphism $\pi'/\pi'_n\xrightarrow{\cong}\pi/\pi_n$ we obtain from Claim 3, along with the isomorphism $H_1(\partial E_{L'})\xrightarrow{\cong} H_1(\partial E_L)$, constitutes an $n$-basing over the trivial group for $L'$ relative to $L$. Thus, the homotopy class $h_n(L',\phi)$ is defined, and $h_n(L',\phi)_*[M']=f_*[M']=\theta\in H_3(X_n(L))$.
\end{proof}


\section{Relationships to previous work}\label{previous}

In this section, we investigate relationships between the lower central homotopy invariants and previous versions of Milnor's invariants. We apply Theorem~\ref{THMhcharacterization} in Section~\ref{previous-milnor} to show that for links in $S^3$ and $n\geq 2$ the invariant $\overline{\mu}_n$ relative to the unlink is equivalent to Milnor's $\overline{\mu}$-invariant of length $n+1$ \cite{Milnor}. The connection between the stronger invariants $h_n$ and $\theta_n$ and Milnor's $\overline{\mu}$-invariants passes through invariants of Orr \cite{Orr89}. In Section~\ref{previous-mfld}, we show that the invariants $\theta_n$ and $\overline{\mu}_n$ for empty links are precisely the Cha-Orr ``Milnor's invariants of 3-manifolds" \cite{ChaOrr}. Finally, in Section~\ref{previous-other}, we again leverage Theorem~\ref{THMhcharacterization} to prove that the vanishing of the lower central homotopy invariants, taken in specific contexts, implies the vanishing of other extensions of Milnor's invariants. We consider D. Miller's $\overline{\tau}$-invariants for knots in Seifert-fibered 3-manifolds \cite{MillerD} and M. Kuzbary's Dwyer number \cite{Kuzbary}.


\subsection{Invariants of links in \texorpdfstring{$S^3$}{S3}} \label{previous-milnor}

We begin by understanding the relationship between the new invariants defined in this paper and Milnor's original $\overline{\mu}$-invariants. K. Orr's homotopy-theoretic version of Milnor's $\overline{\mu}$-invariants is essential for understanding this connection. His approach significantly improved the understanding of Milnor's invariants and inspired much of the work in this paper. In \cite{Orr89}, Orr proves a realization theorem for his invariants, shows his invariants are equivalent to Milnor's, computes the number of linearly independent invariants for any fixed length, and establishes a connection to $n$-cobordism. We first apply Theorem~\ref{THMhcharacterization} to establish the equivalence for $n\geq 2$ of the invariant $\overline{\mu}_n$ relative to the unlink and Milnor's $\overline{\mu}$-invariant of length $n+1$ \cite{Milnor}. Then, we show that the lower central homotopy invariants for links in $S^3$ relative to the $m$-component unlink are equivalent to Orr's invariants \cite{Orr89}. Finally, we leverage work of Igusa-Orr \cite{IgusaOrr} to establish a relationship between the $h_n$ and Milnor's $\overline{\mu}$-invariants. For all of Section~\ref{previous-milnor}, we assume all links live in $S^3$, we always choose the unlink as our fixed link, and we take the word ``longitude" to mean ``preferred longitude".

We briefly summarize Orr's construction. Suppose $L\subset S^3$ is an $m$-component link with vanishing Milnor's $\overline{\mu}$-invariants of lengths $\leq n$ for some $n\geq 2$. Equivalently, the longitudes of $L$ are in $\pi_n$, so their images under the map $E_L\to K(\pi/\pi_n,1)$ are nullhomotopic. Choosing meridians for $L$ (i.e. a basing for $L$) defines an isomorphism $\tau:F/F_n\xrightarrow{\cong}\pi/\pi_n$, where $F=F(m)$ is the $m$-generator free group. We may view its inverse $\phi:\pi/\pi_n\xrightarrow{\cong}F/F_n$ as an isomorphism to the $n^\text{th}$ lower central quotient of the link group of the $m$-component unlink $U$. This yields a map $E_L\to K(F/F_n,1)$ which we may assume sends the components of $\partial E_L$ by projection onto the images of the corresponding meridians. The map $E_L\to K(F/F_n,1)$ therefore extends to a map $S^3\to K_n$, where $K_n=K(F/F_n,1)\cup_{(\sqcup m_i)}(\sqcup D^2)$ is obtained by gluing $m$ disks $D^2$ to the Eilenberg-MacLane space along the images of the meridians. The map $S^3\to K_n$ sends the regular neighborhood of each component of $L$ by projection onto the corresponding $D^2$.

Orr denotes the resulting homotopy class of maps $S^3\to K_n$ by $\theta_n(L,\tau)\in\pi_3(K_n)$. To avoid confusion with our notation and to emphasize the relationship with the invariants from Section~\ref{results-homotopy}, we will instead denote this class by $h^O_n(L,\phi)\in\pi_3(K_n)$. Similarly to the invariant $h_n$, one could remove the dependency on $\phi$ and define $h^O_n(L)$. Note that $h^O_n(U,\text{id})$ is nullhomotopic because the map $E_U\to K(F/F_n,1)$ factors through the wedge of circles to which the disks are attached. Thus, $h^O_n(L,\phi)=h^O_n(U,\text{id})$ if and only if $h^O_n(L,\phi)$ is nullhomotopic. In other words, $h^O_n(L,\phi)$ vanishes if and only if it is nullhomotopic. We will construct a map $r_n:X_n(U)\to K_n$ relating the lower central homotopy invariant $h_n$ to Orr's invariant $h^O_n$, so that $r_n\circ h_n(L,\phi)=h^O_n(L,\phi)$.

\begin{lemma}\label{previous-lemmamubasing}
An $n$-basing for $L$ relative to $U$ exists if and only if Milnor's $\overline{\mu}$-invariants for $L$ of lengths $\leq n$ vanish.
\end{lemma}

\begin{proof}
($\Longrightarrow$). An $n$-basing $(\phi,\phi_\partial)$ for $L$ relative to $U$ includes an isomorphism $\phi:\pi/\pi_n\xrightarrow{\cong}F/F_n$ which sends (based) longitudes of $L$ to (based) longitudes of $U$. This implies the longitudes of $L$ are trivial in $\pi/\pi_n$ or, equivalently, Milnor's $\overline{\mu}$-invariants of lengths $\leq n$ vanish. \\
($\Longleftarrow$). Now suppose Milnor's $\overline{\mu}$-invariants of lengths $\leq n$ vanish. The longitudes of $L$ are therefore in $\pi_n$, so we may define an isomorphism $\phi:\pi/\pi_n\xrightarrow{\cong}F/F_n$ which preserves meridians and longitudes. In other words, there is an admissible isomorphism $\phi_\partial:H_1(\partial E_L)\xrightarrow{\cong} H_1(\partial E_U)$ compatible with an isomorphism $\phi:\pi/\pi_n\xrightarrow{\cong} F/F_n$. Thus, we have produced an $n$-basing $(\phi,\phi_\partial)$ for $L$ relative to $U$.
\end{proof}

\begin{corollary}\label{previous-cormu}
For $n\geq 2$, the invariant $\overline{\mu}_n(L)$ vanishes if and only if Milnor's $\overline{\mu}$-invariants for $L$ of lengths $\leq n+1$ vanish.
\end{corollary}

\begin{proof}
Let $n\geq 2$. By Theorem~\ref{THMhcharacterization}, $\overline{\mu}_n(L)$ vanishes if and only if there exists an $(n+1)$-basing for $L$ relative to $U$. By Lemma~\ref{previous-lemmamubasing}, this is true if and only if Milnor's $\overline{\mu}$-invariants for $L$ of lengths $\leq n+1$ vanish.
\end{proof}

\begin{corollary}\label{previous-cordefined}
For $n\geq 2$, the invariant $h_n(L)$ is defined if and only if Orr's invariant $h^O_n(L)$ is defined.
\end{corollary}

Relative versions of Lemma~\ref{previous-lemmamubasing} and Corollary~\ref{previous-cormu} hold:

\begin{lemma}\label{previous-lemmamubasingrel}
Suppose $L$ and $L'$ are links with vanishing Milnor's $\overline{\mu}$-invariants of lengths $\leq n$. An $(n+1)$-basing for $L'$ relative to $L$ exists if and only if Milnor's $\overline{\mu}$-invariants for $L$ and $L'$ of length $n+1$ agree.
\end{lemma}

\begin{corollary}\label{previous-cormurel}
Suppose $L$ and $L'$ are links with vanishing Milnor's $\overline{\mu}$-invariants of lengths $\leq n$ for some $n\geq 2$. Then $\overline{\mu}_n(L)=\overline{\mu}_n(L')$ if and only if Milnor's $\overline{\mu}$-invariants for $L$ and $L'$ of length $n+1$ agree.
\end{corollary}

\noindent Thus, for $n\geq 2$, the invariant $\overline{\mu}_n$ is equivalent to Milnor's (total) $\overline{\mu}$-invariant of length $n+1$. We provide a brief sketch of the proof. Details involving the additivity of the first nonvanishing Milnor invariants appear in \cite{Cochran90} and \cite{Orr89}. 

\begin{proof}[Sketch of the proofs of Lemma~\ref{previous-lemmamubasingrel} and Corollary~\ref{previous-cormurel}]
For links $L$ and $L'$ with vanishing Milnor's $\overline{\mu}$-invariants of lengths $\leq n$, Milnor's $\overline{\mu}$-invariants of length $n+1$ for $L$ and $L'$ agree if and only if $L\# (-L')$ (with any chosen bands for the connected sum) has vanishing Milnor's invariants of length $n+1$ or, equivalently, the differences of the longitudes of $L$ and $L'$ lie in the $(n+1)^{\text{st}}$ lower central subgroup. Thus, Milnor's $\overline{\mu}$-invariants of length $n+1$ for $L$ and $L'$ agree if and only if we can define an isomorphism $\pi'/\pi'_{n+1}\xrightarrow{\cong}\pi/\pi_{n+1}$ preserving (based) meridians and longitudes, that is, an $(n+1)$-basing for $L'$ relative to $L$. Corollary~\ref{previous-cormurel} then follows from Theorem~\ref{THMhcharacterization}.
\end{proof}

We now explore the relationship between the space $X_n(U)$ and Orr's space $K_n$. Recall that $X_n(U)$ is formed by gluing $\nu U$ to $M_{p^U_n}^\times$, the reduced mapping cylinder of $p^U_n:E_U\to K(F/F_n,1)$. We view this gluing in two stages: First, attach a disk $D^2$ to each meridian of $U$. Then, fill in the remainder of each component of $\partial E_U$ with a 3-ball. Because $M_{p^U_n}^\times$ is also a $K(F/F_n,1)$, the inclusion map $E_U\hookrightarrow M_{p^U_n}^\times$ is homotopic to a map which projects each component of $\partial E_U$ to the corresponding meridian in $M_{p^U_n}^\times$. This factors through a homotopy self-equivalence $M_{p^U_n}^\times\xrightarrow{\simeq}M_{p^U_n}^\times$ which behaves this way on $\partial E_U$.

After the first stage of gluing, we see that $M_{p^U_n}^\times\cup_{(\sqcup m_i)}(\sqcup D^2)$ is a model for the space $K_n$. We may extend the homotopy self-equivalence $M_{p^U_n}^\times\xrightarrow{\simeq}M_{p^U_n}^\times$ by the identity on each $D^2$, giving a homotopy equivalence $M_{p^U_n}^\times\cup_{(\sqcup m_i)}(\sqcup D^2)\xrightarrow{\simeq}K_n$. This map extends over each 3-ball we attach to form $X_n(U)$ by viewing $D^3=D^1\times D^2$ and projecting to $D^2\subset K_n$. We thus obtain a map $r_n:X_n(U)\to K_n$. By construction, $r_n\circ h_n(U,\text{id})=r_n\circ \iota_n(U)=h^O_n(U,\text{id})$. We have the diagram

\begin{center}
\begin{tikzcd}
E_U \arrow[dr,hookrightarrow] \arrow[rr] \arrow[dd,hookrightarrow] & & M_{p^U_n}^\times \arrow[dd,hookrightarrow] \\
& M_{p^U_n}^\times \arrow[ur,"\simeq"] & \\
S^3 \arrow[dr,hookrightarrow,"\iota_n(U)"'] \arrow[rr,"h^O_n(U{,}\text{id})" near end] & & K_n. \\
& X_n(U) \arrow[ur,"r_n"'] \arrow[from=uu,hookrightarrow,crossing over] &
\end{tikzcd}
\end{center}
The image in $K_n$ of the attaching map of each 3-ball we glue to $M_{p^U_n}^\times\cup_{(\sqcup m_i)}(\sqcup D^2)$ to form $X_n(U)$ is nullhomotopic. Thus, $X_n(U)\simeq K_n\vee(\vee^m S^3)$, and the map $r_n$ may be viewed as collapsing each of these $S^3$ factors to the basepoint.

Now suppose $L$ admits an $n$-basing $(\phi,\phi_\partial)$ relative to $U$. By Lemma~\ref{previous-cormu}, this is the case if and only if Milnor's $\overline{\mu}$-invariants of lengths $\leq n$ vanish for $L$. Equivalently, the longitudes of $L$ are in $\pi_n$. Define a homotopy self-equivalence $M_{p_n}^\times\xrightarrow{\simeq}M_{p_n}^\times$, similar to the self-equivalence of $M_{p^U_n}^\times$ above, which projects each component of $\partial E_L$ to the corresponding meridian. The rightmost square in the diagram above fits into the diagram below.
\begin{center}
\begin{tikzcd}
E_L \arrow[dr,hookrightarrow] \arrow[rr] \arrow[dd,hookrightarrow] & & M_{p_n}^\times \arrow[rr,"\simeq"] & & M_{p^U_n}^\times \arrow[dd,hookrightarrow] \\
& M_{p_n}^\times \arrow[ur,"\simeq"] \arrow[rr,"\simeq","h_\phi"'] & & M_{p^U_n}^\times \arrow[ur,"\simeq"] & \\
S^3 \arrow[dr,hookrightarrow,"\iota_n(L)"'] \arrow[rrrr,"h^O_n(L{,}\phi)"] \arrow[drrr,"h_n(L{,}\phi)"] & & & & K_n \\
& X_n(L) \arrow[from=uu,hookrightarrow,crossing over] \arrow[rr,"\simeq","\overline{h}_\phi"'] & & X_n(U) \arrow[ur,"r_n"'] \arrow[from=uu,hookrightarrow,crossing over] &
\end{tikzcd}
\end{center}

The maps $M_{p_n}^\times\xrightarrow{\simeq} M_{p^U_n}^\times$ on the top square of the diagram are both homotopy equivalences of pairs; in particular, the map on the front face sends $\partial E_L$ to $\partial E_U$ by an admissible homeomorphism, and the map on the back face sends the meridians of $L$ to the meridians of $U$. The extensions of the maps on the top of the diagram to the maps on the bottom therefore commute as claimed. Commutativity of the diagram proves the following proposition.

\begin{proposition}\label{previous-propr}
The map $r_n:X_n(U)\to K_n$ satisfies $r_n\circ h_n(L,\phi)=h^O_n(L,\phi)$.
\end{proposition}

Recall from Section~\ref{results-homotopy} that we say $h_n(L,\phi)$ \emph{vanishes} if $h_n(L,\phi)=h_n(U,\text{id})$. By Proposition~\ref{previous-propr}, the vanishing of $h_n(L,\phi)$ implies the vanishing of $h^O_n(L,\phi)$: If $h_n(L,\phi)=h_n(U,\text{id})$, then $h^O_n(L,\phi)=r_n\circ h_n(L,\phi)=r_n\circ h_n(U,\text{id})=h^O_n(U,\text{id})$. The converse also holds.

\begin{proposition}\label{previous-propOrrh}
The invariants $h_n$ and $h^O_n$ are equivalent, that is, $h_n(L)=h_n(L')$ if and only if $h^O_n(L)=h^O_n(L')$. In particular, $h_n(L)$ vanishes if and only if $h^O_n(L)$ is nullhomotopic.
\end{proposition}

\begin{proof}
We could prove this directly, but a faster proof leverages the relationships between the invariants in question and $n$-cobordism. By Theorem~\ref{THMcharacterization}, the invariants $h_n(L)$ and $h_n(L')$ agree if and only if $L$ and $L'$ are $n$-cobordant. By Theorem 7 of \cite{Orr89}, this is the case if and only if $h^O_n(L)=h^O_n(L')$.
\end{proof}
\noindent To prove Proposition~\ref{previous-propOrrh} directly, one views $X_n(U)$ as $K_n\vee(\vee^m S^3)$ and observes that all invariants differ only on $K_n$.

Proposition~\ref{previous-propOrrh} implies that the sets of realizable classes in $\pi_3(X_n(U))$ and $\pi_3(K_n)$ are in bijective correspondence. By Theorem 4 of \cite{Orr89}, all classes in $\pi_3(K_n)$ are realizable. Since $\pi_2(S^3)=0$, condition (3) of the realization theorem for $h$-invariants, Theorem~\ref{THMrealization}, follows from conditions (1) and (2) by Corollary~\ref{htpychar-thmC-corpi2}. Condition (2) of Theorem~\ref{THMrealization} is vacuously satisfied when $M=S^3$, so every homotopy class in $\pi_3(X_n(U))$ satisfying condition (1) of Theorem~\ref{THMrealization} is realizable. Thus, homotopy classes in $\pi_3(X_n(U))$ satisfying condition (1) of Theorem~\ref{THMrealization} correspond bijectively with elements of $\pi_3(K_n)$.

To conclude this section, we consider other versions of Orr's invariant. Let $\theta^O_n(L,\phi)$ be the homology class $\theta^O_n(L,\phi)=h^O_n(L,\phi)_*[S^3]\in H_3(K_n)$, and let $\overline{\mu}^O_n(L,\phi)$ be the image of this class in $\text{coker}\big(H_3(K_{n+1})\to H_3(K_n)\big)$, which is isomorphic to $\text{coker}\big(\pi_3(K_{n+1})\to\pi_3(K_n)\big)$ as seen in \cites{Orr89, IgusaOrr}. Again, one could remove the dependency on $\phi$ and define $\theta^O_n(L)$ and $\overline{\mu}^O_n(L)$; in fact, Orr shows that $\overline{\mu}^O_n(L,\phi)$ is independent of the choice of $\phi$. Theorem 12 of \cite{Orr89} states that $\overline{\mu}^O_n(L)$ vanishes if and only if $L$ has vanishing Milnor's $\overline{\mu}$-invariants of lengths $\leq n+1$. Combining this discussion with Corollary~\ref{previous-cormu}, we have

\begin{corollary}
For $n\geq 2$, the invariant $\overline{\mu}_n(L)$ vanishes if and only if the invariant $\overline{\mu}^O_n(L)$ vanishes.
\end{corollary}

Work of Igusa-Orr \cite{IgusaOrr} establishes the exact relationship between Milnor's $\overline{\mu}$-invariants and the stronger invariants $\theta^O_n(L)$ and $h^O_n(L)$ for $n\geq 2$. We summarize their results in the table below.

\begin{center}
\begin{tabular}{|c|c|}
\hline
Orr's invariant & Equivalent Milnor's $\overline{\mu}$-invariant(s) \\
\hhline{|=|=|}
$\overline{\mu}^O_n$ & $\overline{\mu}$ of length $n+1$ \\
\hline
$\theta^O_n$ & $\overline{\mu}$ of lengths $n+1\leq k \leq 2n-1$ \\
\hline
$h^O_n$ & $\overline{\mu}$ of lengths $n+1\leq k \leq 2n$ \\
\hline
\end{tabular}
\end{center}

\noindent The final row of this table is known as the \emph{$k$-slice theorem} due to the relationship between the invariant $h^O_k$ and $k$-cobordism. A link which is $k$-cobordant to the unlink is called \emph{$k$-slice}. By Proposition~\ref{previous-propOrrh}, the $k$-slice theorem still holds if we replace Orr's invariant $h^O_k$ with the lower central homotopy invariant $h_k$. One may ask whether analogous relationships among the new invariants $h_n$, $\theta_n$, and $\overline{\mu}_n$ from Sections~\ref{results-homotopy} and~\ref{results-homology} hold for links in other 3-manifolds. This is a topic of future exploration; see Section~\ref{future-invariants}.


\subsection{Milnor invariants of 3-manifolds} \label{previous-mfld}

One of the most recent versions of Milnor's invariants is a collection of homology cobordism invariants of 3-manifolds due to Cha-Orr \cite{ChaOrr}. In addition to Orr's work \cite{Orr89}, the paper \cite{ChaOrr} also inspired many ideas in this paper. Cha-Orr refer to their invariants as ``Milnor invariants of 3-manifolds" for good reason; the invariants exhibit behavior similar to Milnor's link invariants. In fact, Cha-Orr show one can recover Milnor's $\overline{\mu}$-invariants of a link $L\subset S^3$ as invariants of the 3-manifold which is 0-framed surgery on every component of $L$. 

In this section, we recover the Cha-Orr invariants from the lower central homology invariants $\theta_n$ and Milnor invariants $\overline{\mu}_n$ for empty links. Comparing empty links is equivalent to comparing 3-manifolds. Theorem~\ref{THMhinvariance} proves that $\theta^G_n$ and $\overline{\mu}^G_n$ are invariants of $\mathbb{Z}[G]$-homology cobordism, as a homology concordance of empty links is the same as a homology cobordism of 3-manifolds. For this section, we use the notation from Section~\ref{results-homology}. In particular, any 3-manifold $M$ we consider will be equipped with a fixed surjective homomorphism $\varphi:\pi_1(M)\twoheadrightarrow G$ onto a fixed group $G$.

Consider such a pair $(M,\varphi)$. The empty link $\varnothing_M$ in $M$ has exterior $E_{\varnothing_M}=M$ and link group $\pi_1(M)$. The appropriate $\mathbb{Z}[G]$-lower central quotients are the $\mathbb{Z}[G]$-lower central quotients of $\pi_1(M)$. Following our standard notation, we set $\pi=\pi_1(M)$ and $\Gamma=\ker\varphi$. The homotopy pushout diagram $\mathcal{D}^G_n(L)$ from Section~\ref{results-overG} becomes
\begin{center}
\begin{tikzcd}
M \arrow[r] \arrow[d,equal] & K(\pi/\Gamma_n,1). \\
M
\end{tikzcd}
\end{center}
Thus, the space $X_n(\varnothing_M)$ is just the reduced mapping cylinder of $M\to K(\pi/\Gamma_n,1)$, which is itself a $K(\pi/\Gamma_n,1)$, and the map $\iota_n(\varnothing_M)$ is a map $M\to K(\pi/\Gamma_n,1)$ induced by the canonical projection $\pi\twoheadrightarrow\pi/\Gamma_n$.

Now fix the empty link $\varnothing_M$, and suppose we wish to compare another empty link $\varnothing_{M'}$ to $\varnothing_M$. To define the invariants from Section~\ref{results-homology}, we require an $n$-basing over $G$ for $\varnothing_{M'}$ relative to $\varnothing_M$. Because the empty link exteriors have no boundary, an $n$-basing over $G$ for $\varnothing_{M'}$ relative to $\varnothing_M$ is just an isomorphism $\phi:\pi'/\Gamma'_n\xrightarrow{\cong}\pi/\Gamma_n$ over $G$. Given such an isomorphism $\phi$, the invariant $\theta_n(\varnothing_{M'},\phi)$ is the image of $[M']$ under \[H_3(M')\to H_3(\pi'/\Gamma'_n)\xrightarrow[\phi_*]{\cong}H_3(\pi/\Gamma_n).\] This produces a $\mathbb{Z}[G]$-version of the Cha-Orr invariant from \cite{ChaOrr}. 

Taking $G$ to be the trivial group (so that we obtain invariants of $\mathbb{Z}$-homology cobordism) yields exactly the Cha-Orr invariant $\theta_n(M',\phi)$. In this setting, the set of realizable classes $\mathcal{R}_n(\varnothing_M)\subseteq H_3(\pi/\pi_n)$ is the set of realizable classes of 3-manifold invariants $\mathcal{R}_n(\pi)$ from \cite{ChaOrr}. The definitions of the equivalence relations on these sets of realizable classes coincide, so the invariant $\overline{\mu}_n(\varnothing_{M'})$ agrees with the Cha-Orr Milnor invariant $\overline{\mu}_n(M')$.

Observe that the invariants $\theta_n(L')$ and $\overline{\mu}_n(L')$ for a link $L'\subset M'$ relative to a fixed link $L\subset M$ determine the 3-manifold invariants $\theta_n(M')$ and $\overline{\mu}_n(M')$ relative to $M$. An $n$-basing of links induces an $n$-basing of corresponding sublinks, and the invariant for a link determines the invariants of its sublinks, including the empty link.
Thus, if one wishes to compute the ``number" of Milnor invariants relative to a fixed link $L$, one should restrict to counting invariants $\overline{\mu}_n(L')$ whose associated 3-manifold invariants $\overline{\mu}_n(M')$ vanish.

Because the invariants from Section~\ref{results-homology}  reduce to the Cha-Orr invariants, some of the main results from \cite{ChaOrr} are special cases of results in this paper. When restricted to the natural numbers, Theorems B, C, and E and Corollary D from \cite{ChaOrr} follow from our Theorem~\ref{THMhcharacterization}. The realization result Theorem G from \cite{ChaOrr} restricted to natural numbers follows from our Theorem~\ref{THMhrealization}. Cha-Orr extend their invariants to the transfinite setting, and their results hold in this more general context. We discuss such an extension of our work in Section~\ref{future-extensions}.


\subsection{Other extensions of Milnor's invariants} \label{previous-other}

Theorem~\ref{THMhcharacterization} states that the Milnor invariants $\overline{\mu}^G_n$ determine the $\mathbb{Z}[G]$-lower central quotients associated to a link $L\subset M$ one step at a time. The implications of Theorem~\ref{THMhcharacterization} include the equivalence for $n\geq 2$ of $\overline{\mu}_n$ for links in $S^3$ relative to the unlink with Milnor's $\overline{\mu}$-invariant of length $n+1$ as seen in Section~\ref{previous-milnor}. In this section, we again leverage Theorem~\ref{THMhcharacterization} to show that the vanishing of $\overline{\mu}_n$, defined in particular settings, implies the vanishing of extensions of Milnor's invariants due to D. Miller \cite{MillerD} and M. Kuzbary \cite{Kuzbary}.


\subsubsection{D. Miller's \texorpdfstring{$\overline{\tau}$}{tau}-invariants}

D. Miller considers knots in aspherical Seifert-fibered 3-manifolds in the free homotopy class of the Seifert fiber $K$ in \cite{MillerD}. Miller fixes the Seifert fiber $K$ and compares other knots in the homotopy class of the Seifert fiber relative to $K$. We remind the reader that Seifert-fibered 3-manifolds are precisely those $M$ for which $\pi_1(M)$ contains a normal subgroup isomorphic to $\mathbb{Z}$. This normal subgroup is generated by the class of the Seifert fiber, and $\pi_1(M)$ may be described as the extension \[1\to\mathbb{Z}\hookrightarrow\pi_1(M)\twoheadrightarrow S\to 1,\] where $S$ is the orbifold fundamental group of the base surface of the fibration.

In our notation, Miller takes $G=S$, so that $\Gamma=\ker(\pi\twoheadrightarrow\pi_1(M)\twoheadrightarrow S)$, where $\pi=\pi_1(E_K)$, and $\Gamma/\Gamma_n\cong\mathbb{Z}\times F/F_n$, where $F$ is a free group. These quotients are analogous to the lower central quotients of the unlink $U\subset S^3$. Given a knot $K'$ freely homotopic to $K$, Miller finds a presentation for $\Gamma'/\Gamma'_n$ analogous to Milnor's for $L\subset S^3$. He then proceeds in a manner analogous to Milnor \cite{Milnor}, using the Magnus expansion to define integer-valued invariants. These integers, which Miller calls the \emph{$\overline{\tau}$-invariants} of $K'$, are the coefficients of the images of distinguished longitudes of the $S$-cover of $K'$ in a power series ring. Note that the image in $S$ of the homotopy class of the Seifert fiber is trivial, so the preimage of any knot in that class under the $S$-cover of $M$ is a link with infinitely many components, indexed by elements of $S$.

Miller proves that all $\overline{\tau}$-invariants vanish for the fixed knot $K$. If $\Gamma'/\Gamma'_n\cong\Gamma/\Gamma_n$ preserving meridians and distinguished longitudes, then the $\overline{\tau}$-invariants for $K'$ of lengths $\leq n$ vanish. Miller's work implies a theorem analogous to \hyperref[THMmilnor]{Milnor's Theorem}, although it is not directly proven in \cite{MillerD}: We have an isomorphism $\Gamma'/\Gamma'_{n+1}\cong\Gamma/\Gamma_{n+1}$ or, equivalently, an isomorphism $\Gamma'/\Gamma'_n\cong\Gamma/\Gamma_n$ preserving meridians and distinguished longitudes \emph{if and only if} the $\overline{\tau}$-invariants for $K'$ of lengths $\leq n$ vanish. Miller concludes \cite{MillerD} by proving a constructive realization theorem for the $\overline{\tau}$-invariants.

\begin{proposition}
Let $M$ be an aspherical Seifert-fibered 3-manifold, $K$ the Seifert fiber of $M$, and $K'$ some knot freely homotopic to $K$. Let $S$ be the quotient of $\pi_1(M)$ by the infinite cylic normal subgroup corresponding to the Seifert fiber. Suppose for some $n\geq 2$ that the $n^{\text{th}}$ $S$-Milnor invariant $\overline{\mu}^S_n(K')$ vanishes, that is, $\overline{\mu}^S_n(K')$ is defined and equals $\overline{\mu}^S_n(K)$. Then Miller's $\overline{\tau}$-invariants for $K'$ of lengths $\leq n+1$ vanish.
\end{proposition}

\begin{proof}
By Theorem~\ref{THMhcharacterization}, since $\overline{\mu}^S_n(K)=\overline{\mu}^S_n(K')$, $K'$ admits an $(n+1)$-basing over $S$ relative to $K$. This consists of an isomorphism $\phi:\pi'/\Gamma'_{n+1}\xrightarrow{\cong}\pi/\Gamma_{n+1}$ which preserves meridians and distinguished longitudes. This yields an isomorphism $\Gamma'/\Gamma'_{n+1}\xrightarrow{\cong}\Gamma/\Gamma_{n+1}$ which preserves meridians and distinguished longitudes. By Miller's work, the $\overline{\tau}$-invariants of lengths $\leq n+1$ vanish.
\end{proof}


\subsubsection{M. Kuzbary's Dwyer number}

M. Kuzbary develops the \emph{Dwyer number}, an invariant for knots in closed oriented 3-manifolds that behaves like the first non-vanishing Milnor invariant, in \cite{Kuzbary}. This invariant compares the lower central quotients $\pi'/\pi'_n$ associated to a knot $K'\subset M$ to the quotients $\pi/\pi_n$ for a fixed knot $K\subset M$. Using the geometrically-defined Dwyer subgroups $\Phi_n(-)\leq H_2(-;\mathbb{Z})$ first formulated in \cite{FreedmanTeichner} (see also \cite{CochranHarvey}), Kuzbary defines the \emph{Dwyer number of $K'$ relative to $K$} as \[D(K',K)=\max\bigg\{\,n\,\bigg|\,\frac{H_2(E_{K'})}{\Phi_n(E_{K'})}\cong \frac{H_2(E_K)}{\Phi_n(E_K)}\,\bigg\}\in\mathbb{N}\cup\{\infty\}. \]
Kuzbary shows that $H_2(X)/\Phi_n(X)\cong H_2(\pi_1(X))/K_n(\pi_1(X))$. Recall from the statements of Theorems~\ref{THMrealization} and~\ref{THMhrealization} that for a group $H$ the abelian group $K_n(H)$ is defined as $K_n(H)=\ker\big(H_2(H)\to H_2(H/H_n)\big)$. The \hyperref[THMsd]{Stallings-Dwyer Theorem} implies $D(K',K)\geq n$ if and only if we have an isomorphism $\pi'/\pi'_{n+1}\cong\pi/\pi_{n+1}$. Kuzbary leverages the Dwyer number to analyze concordance of null-homologous knots in $\#^k(S^1\times S^2)$ and establish a relationship with Massey products.

\begin{proposition}
Fix a knot $K$ in a closed orientable 3-manifold $M$, and let $K'\subset M$ be another knot freely homotopic to $K$. Let $G$ be the trivial group. Suppose for some $n\geq 2$ that the $n^{\text{th}}$ Milnor's invariant $\overline{\mu}^G_n(K')$ vanishes, that is, $\overline{\mu}^G_n(K')$ is defined and equals $\overline{\mu}^G_n(K)$. Then $D(K',K)\geq n$.
\end{proposition}

\begin{proof}
By Theorem~\ref{THMhcharacterization}, since $\overline{\mu}^G_n(K')=\overline{\mu}^G_n(K)$, $K'$ admits an $(n+1)$-basing over $G$ relative to $K$. In particular, this includes an isomorphism $\phi:\pi'/\pi'_{n+1}\xrightarrow{\cong}\pi/\pi_{n+1}$. Thus, $D(K',K)\geq n$.
\end{proof}


\section{Future directions}\label{future}

We conclude this paper with a discussion of several current and future directions, including a number of open questions.


\subsection{Further investigation of new invariants} \label{future-invariants}

The fact that the invariants from Sections~\ref{results-homotopy} and~\ref{results-homology} reduce to Milnor's $\overline{\mu}$-invariants in the case where we compare links in $S^3$ to the unlink motivates us to ask whether known results for Milnor's $\overline{\mu}$-invariants extend to the more general setting of this work.


\subsubsection*{Algebraic information captured by $h_n$ and $\theta_n$}

Igusa-Orr answer a long-standing open question known as the $k$-slice conjecture affirmatively in \cite{IgusaOrr}: A link $L\subset S^3$ is $k$-cobordant to the unlink if and only if Milnor's $\overline{\mu}$-invariants of lengths $1\leq i\leq 2k$ vanish. Igusa-Orr prove this using Igusa Pictures \cite{Igusa} and Orr's invariants from \cite{Orr89} which characterize $k$-cobordism and are analogous to the $h$-invariants from Section~\ref{results-homotopy}. Recall from Section~\ref{previous-milnor} that the lower central homotopy invariant $h_k$ is equivalent to Orr's invariant of length $k$. As discussed in Section~\ref{previous-milnor}, Igusa-Orr show that Orr's length $k$ invariant captures Milnor's $\overline{\mu}$-invariants of lengths $k+1\leq i\leq 2k$. In the more general setting of this work, it is probable that we also lose significant information passing from $h_k$ or $\theta_k$ to the invariants $\overline{\mu}_k$. Understanding how much more information these stronger invariants capture is equivalent to understanding to what extent a relative version of the $k$-slice conjecture holds for links in closed orientable 3-manifolds.

\begin{question}\label{future-qkslice}
Does $h_k$ vanish if and only if $\overline{\mu}_i$ vanishes for $k\leq i\leq 2k-1$?
\end{question}


\subsubsection*{Whitney towers and surface towers}

Interpretations of Milnor's $\overline{\mu}$-invariants in terms of twisted Whitney towers and surface towers (also called gropes in the literature) are given in \cites{FreedmanTeichner,CST11,CST14,Cha18}. Cha-Orr give a similar interpretation for ``Milnor invariants of 3-manifolds" and extend it to the transfinite setting \cite{ChaOrr}. 

\begin{goal}\label{future-goaltower}
Provide an interpretation of the invariants from this paper in terms of Whitney towers and surface towers.
\end{goal}


\subsubsection*{Finite type invariants}

Bar-Natan \cite{BarNatan} and X. Lin \cite{Lin} showed Milnor's invariants for string links are of finite type. Habegger-Masbaum determine a formula for computing Milnor's $\overline{\mu}$-invariants for string links from the Kontsevich integral in \cite{HabeggerMasbaum}. 

\begin{goal}\label{future-goalfinitetype}
Recast the invariants from this paper as finite type invariants.
\end{goal}

\noindent M. Powell has suggested that the theory of claspers due to Gusarov \cite{Gusarov} and Habiro \cite{Habiro} may seed a constructive realization machine for these invariants.


\subsubsection*{Invariants comparing links in $S^3$ to a nontrivial link}

In the case where we compare links in $S^3$ to the unlink, the invariants $h_n$ from Section~\ref{results-homotopy} reduce to invariants of Orr \cite{Orr89} which capture Milnor's $\overline{\mu}$-invariants.

\begin{question}\label{future-qS3}
Do we gain anything by fixing a nontrivial link $L\subset S^3$ instead of the unlink?
\end{question}

\noindent Despite the additivity of Milnor's invariants, due to Cochran \cite{Cochran90} and Orr \cite{Orr89} in the case of the first nonvanishing invariant and more generally due to Krushkal \cite{Krushkal}, we expect the answer is yes.


\subsection{Extensions of this work} \label{future-extensions}

In addition to the potential new interpretations of the invariants in this paper described in Section~\ref{future-invariants}, we may consider whether certain known extensions of Milnor's $\overline{\mu}$-invariants carry over to the setting of this work.


\subsubsection*{Extending to the transfinite setting}

The existence of nontrivial transfinite Milnor invariants for links has been an open question since Milnor's work \cite{Milnor}. Milnor originally asked about invariants detecting phenomena in the transfinite lower central quotients of link groups for links in $S^3$. Using J. Levine's algebraic closure of groups \cites{LevineJ89a,LevineJ89b}, equivalent to Vogel's homology localization of groups \cite{Vogel78}, Cha-Orr extend their homology cobordism invariants to the transfinite setting and answer a similar question in \cite{ChaOrr}. They exhibit 3-manifolds with vanishing finite-length Milnor invariants and non-trivial transfinite invariants. One expects analogous extensions of the invariants $h_n$, $\theta_n$, and $\overline{\mu}_n$ should exist, in particular because the $\theta$- and $\overline{\mu}$-invariants contain the Cha-Orr invariants as special cases. These extensions would involve a tower of spaces $X_\omega(L)$ with maps from $M$, where $\omega$ is a countable ordinal. 

\begin{goal}\label{future-goaltransfinite}
Leverage the examples from \cite{ChaOrr} to find the first instances of nontrivial transfinite Milnor invariants for links.
\end{goal}


\subsubsection*{J. Levine's group closure}

One expects an extension of the invariants $h_n$, $\theta_n$, and $\overline{\mu}_n$ even beyond the previous idea to invariants that detect phenomena at the level of J. Levine's algebraic closure \cites{LevineJ89a,LevineJ89b}. These invariants would probe this algebraic closure without passing to (finite or transfinite) lower central quotients. They would be ``universal $\pi_1$ invariants" in the sense that they would detect information about the $\mathbb{Z}[G]$-homology of link exteriors coming from link groups at the deepest possible level. In other words, these invariants would detect information about the fundamental group of the $\mathbb{Z}[G]$-Vogel homology localization of the link exterior invisible to any lower central quotient; see \cite{Vogel78}. Such invariants would be analogous J. Levine's invariant for links in $S^3$ from \cite{LevineJ89b}. Whether this invariant is nontrivial is an open question. 

Just as in the transfinite setting, Cha-Orr extend their homology cobordism invariants to this universal $\pi_1$ setting and prove the existence of 3-manifolds with vanishing finite- and transfinite-length Milnor invariants but with non-trivial universal invariants \cite{ChaOrr}. Again, one expects analogous extensions of the invariants $h_n$, $\theta_n$, and $\overline{\mu}_n$ since the $\theta$- and $\overline{\mu}$-invariants contain the Cha-Orr invariants as special cases. These invariants would involve spaces $X(L,1)$ equipped with maps from $M$. The various maps from $M$ would fit into a diagram as follows:

\begin{center}
\begin{tikzcd}
& M \arrow[ddl] \arrow[ddr] \arrow[ddrrr] \arrow[ddrrrrr] & & & & & & \\
\\
X(L,1) \arrow[r] & \cdots \arrow[r] & X_\omega(L) \arrow[r] & \cdots \arrow[r] & X_n(L) \arrow[r] & \cdots \arrow[r] & X_1(L).
\end{tikzcd}
\end{center}

\begin{goal}\label{future-goaluniversal}
Leverage the work from \cite{ChaOrr} to find links with nontrivial universal Milnor invariants.
\end{goal}


\subsubsection*{Milnor's invariants for knotted surfaces}

The study of surface links (including single-component knotted surfaces) in 4-manifolds is currently a very active area in 4-dimensional topology. Concordance in this setting has garnered a lot of recent attention \cites{Sunukjian,Schwartz,AMY,KlugMiller21,MillerM,KlugMiller,SchneidermanTeichner}, but little is understood overall. Despite recent work, there is a general shortage of concordance invariants for links of surfaces, particularly when the surfaces have postitive genus.

Audoux-Meilhan-Yasuhara recently explored an analogue of Milnor's invariants for concordance of knotted surfaces in 4-space \cite{AMY}. They leverage the combinatorial method of cut-diagrams to apply tools similar to Milnor's work \cite{Milnor} and develop concordance invariants. Analogous versions of Orr's invariants \cite{Orr89} in this setting are likely amenable to computing the number of such invariants.

\begin{goal}\label{future-goalsurfaces}
Extend the invariants of \cite{AMY} to links of positive genus surfaces in closed orientable 4-manifolds. More broadly, extend the invariants from this work to concordance invariants of embeddings of $n$-manifolds in $(n+2)$-manifolds.
\end{goal}

Progress toward Goal~\ref{future-goalsurfaces} will appear in forthcoming joint work with M. Miller.


\subsubsection*{Variations on the definition of concordance}

We list a few variations of this work adapted to studying links in slightly different contexts:
\begin{itemize}
\item One could study concordance of unordered links, unoriented links, or both unordered and unoriented links by slightly changing the definitions of an $n$-basing and the group $\text{Aut}(\pi/\Gamma_n,\partial)$ of self-$n$-basings given in Section~\ref{nbasing}. Versions of the main results of this paper still follow.
\item Instead of studying $\mathbb{Z}[G]$-homology concordance of links in 3-manifolds $M$ with fixed homomorphisms $\varphi:\pi_1(M)\twoheadrightarrow G$ to a fixed group $G$, as in Section~\ref{results-homology}, one could remove such a fixed homomorphism $\varphi$ as part of the structure of $M$ and only ask that $\pi_1(M)$ admits some homomorphism onto $G$. Again, this change would slightly alter the definitions of an $n$-basing and the group of self-$n$-basings.
\end{itemize}


\subsection{Applications of invariants} \label{future-applications}

We list several possible applications of the new invariants defined in this work to existing open questions.


\subsubsection*{Computing collections of realizable classes}

One can leverage Theorems~\ref{THMrealization} and~\ref{THMhrealization} to compute the collection of Milnor invariants relative to a fixed link. One might even hope for a constructive algorithm to realize invariants by specific links in $M$. Such computations for specific knots $K\subset M$ can partially answer the \hyperref[CONJ]{Almost-Concordance Conjecture} \cites{Celoria,FNOP}. Partial progress toward this conjecture has been made in \cites{MillerD,Celoria,Yildiz,FNOP,NOPP,Kuzbary}. We conjecture that the invariants $h_n$ are sufficient for proving the remaining cases of the \hyperref[CONJ]{Almost-Concordance Conjecture}.


\subsubsection*{Applications to homology cobordism}

Recall that the invariants $\theta_n$ and $\overline{\mu}_n$ are homology concordance invariants and reduce to the homology cobordism invariants of Cha-Orr \cite{ChaOrr} when specialized to empty links. 

\begin{question}\label{future-qhcob}
Are there applications of the invariants in this work to the study of homology cobordism of 3-manifolds which do not pass through the Cha-Orr invariants?
\end{question}

The invariants $\theta_n$ and $\overline{\mu}_n$ may also be useful in answering the following question:

\begin{question}\label{future-qhomconc}
Fix a 3-manifold $M$. Does there exist a 3-manifold $M'$, homology cobordant to $M$, and a knot $K\subset M$ such that $K$ is not topologically homology concordant to any knot in $M'$?
\end{question}

A. Levine provides an example of a knot in a homology null-cobordant homology sphere which is not smoothly concordant to any knot in $S^3$ in \cite{LevineA}. The existence of such an example in the topological category is still open; C. Davis shows in \cite{Davis} that such an example cannot be detected by the powerful Cochran-Orr-Teichner filtration on the topological concordance group \cite{COT}. The invariants defined in this work may be able to provide analogous examples in the topological setting in 3-manifolds which are not homology spheres. They may also interact in interesting ways with recent work on homology concordance such as \cite{Zhou} and \cite{HLL}.


\pagebreak


\addcontentsline{toc}{section}{References}
\bibliography{mathrefs.bib}


\appendix


\section{Properties of \texorpdfstring{$X_n(L)$}{Xn(L)} and \texorpdfstring{$\iota_n(L)$}{iotan(L)}}\label{appendix}

In this brief appendix, we collect a number of properties of the spaces $X_n(L)$ and maps $\iota_n(L):M\hookrightarrow X_n(L)$ defined in section~\ref{results-homotopy}. These include a number of well-definedness and naturality properties; see \cite{Stees} for more details.


\subsection{Well-definedness and naturality properties}\label{appendix-welldefnat}

Observe the following well-definedness and naturality properties of $X_n(L)$ and $\iota_n(L)$. These follow from the naturality of the constructions of Section~\ref{results-homotopy} and algebraic-topological properties of homotopy pushouts.
\begin{itemize}
\item The space $X_n(L)$ is well-defined up to homotopy equivalence, independent of the choices of model for $K(\pi/\Gamma_n,1)$ and based map $p_n:E_L\to K(\pi/\Gamma_n,1)$ induced by the canonical projection $\pi\twoheadrightarrow\pi/\Gamma_n$.
\item The model obtained for $X_n(L)$ depends naturally on the choices of model for $K(\pi/\Gamma_n,1)$ and based map $p_n$. Two sets of choices yield a canonical homotopy class of homotopy equivalences $\psi$ between the two resulting models.
\item The homotopy class $\iota_n(L)\in[M,X_n(L)]_0$ is well-defined.
\item The inclusions $\iota_n(L):M\hookrightarrow X_n(L)$ into two models for $X_n(L)$ are related, up to homotopy, via the canonical homotopy equivalence $\psi$ above.
\end{itemize}

Recall the maps $\psi_{m,n}:X_m(L)\to X_n(L)$ for $m\geq n$ defined in Section~\ref{results-homotopy}. The $\psi_{m,n}$ satisfy the following:
\begin{itemize}
\item The homotopy classes $\psi_{m,n}\in[X_m(L),X_n(L)]_0$ are well-defined.
\item The maps $\psi_{m,n}$ between two different models of each of $X_m(L)$ and $X_n(L)$ are related, up to based homotopy, via canonical homotopy equivalences.
\item The collection $\{\psi_{i,j}\}$ satisfies $\psi_{m,p}\simeq\psi_{n,p}\psi_{m,n}$.
\item The collections $\{\psi_{i,j}\}$ and $\{\iota_k(L)\}$ satisfy $\psi_{i,j}\iota_i(L)\simeq\iota_j(L)$.
\end{itemize}

The spaces $X^G_n(L)$ and maps $\iota^G_n(L)$ and $\psi^G_{m,n}$ defined in Section~\ref{results-overG} satisfy analogous well-definedness and naturality properties.


\subsection{Algebraic properties}\label{appendix-alg}

\begin{proposition} \label{appendix-proppi1} Let the fixed group $G$ be $G=\pi_1(M)$, and let $\varphi:\pi_1(M)\twoheadrightarrow G$ be the identity homomorphism.
\begin{enumerate}[label=(\arabic*)]
\item The map $\iota_n(L):M\hookrightarrow X_n(L)$ induces an isomorphism \[\iota_n(L)_*:\pi_1(M)\xrightarrow{\cong}\pi_1(X_n(L)).\]
\item The abelian group $H_3(X_n(L))$ contains a summand canonically isomorphic to $H_3(M)\cong\mathbb{Z}$.
\end{enumerate}
\end{proposition}

\begin{proof}
Since $X_n(L)$ may be obtained from the reduced mapping cylinder $M_{p_n}^\times$ by gluing $\nu_L$ to $\partial E_L\times\{1\}$, the homomorphism $\pi/\Gamma_n\to\pi_1(X_n(L))$ is surjective and its kernel is normally generated by the images of the meridians of $L$ in $\pi/\Gamma_n$. This is precisely the subgroup $\Gamma/\Gamma_n$, since $\Gamma\trianglelefteq\pi$ is the subgroup of $\pi$ normally generated by the meridians of $L$. Thus, \[\pi_1(X_n(L))\cong(\pi/\Gamma_n)/(\Gamma/\Gamma_n)\cong\pi/\Gamma\cong\pi_1(M).\]

To see that this isomorphism is induced by the map $\iota_n(L)$, we instead view $X_n(L)$ as obtained from $M$ by attaching cells of dimension 2 and higher. The attaching maps of the 2-cells represent elements of $\Gamma_n$, which are already nullhomotopic in $M$, and the higher-dimensional cells do not affect the fundamental group. Therefore, $\iota_n(L)_*$ induces an isomorphism. This proves (1).

To prove statement (2), we consider the Mayer-Vietoris sequence corresponding to the decomposition $X_n(L)=M_{p_n}^\times\cup_{\partial E_L}\nu L$. The connecting homomorphism $H_3(X_n(L))\to H_2(\partial E_L)\cong\mathbb{Z}^m$ sends the inclusion-induced image of $[M]$ to $(1,1,\dots,1)$ (i.e. the image of $[M]$ is the sum of the fundamental classes of the tori which form $\partial E_L$). Projecting $H_2(\partial E_L)$ onto $\mathbb{Z}$ generated by $(1,1,\dots,1)$ gives a surjective homomorphism $H_3(X_n(L))\twoheadrightarrow\mathbb{Z}$ which sends the image of $[M]$ to 1. This yields a splitting of $H_3(X_n(L))$ as $H_3(M)\oplus A$ for some abelian group $A$. As the inclusion map $\iota_n(L):M\hookrightarrow X_n(L)$ is canonical, this splitting is canonical. This proves (2).
\end{proof}

The spaces $X^G_n(L)$ from Section~\ref{results-overG} have analogous algebraic properties:

\begin{proposition}\label{appendix-proppi1G}
Let $G$ be some group with $\varphi:\pi_1(M)\twoheadrightarrow G$.
\begin{enumerate}[label=(\arabic*)]
\item Let $\Gamma^G=\ker\varphi$. Then $\pi_1(X^G_n(L))\cong\pi_1(M)/\Gamma^G_n$, the $n^{\text{th}}$ $G$-lower central quotient of $\pi_1(M)$, and the map $\iota^G_n(L):M\hookrightarrow X^G_n(L)$ induces the canonical projection $\pi_1(M)\twoheadrightarrow\pi_1(M)/\Gamma^G_n$.
\item The abelian group $H_3(X^G_n(L))$ contains a summand canonically isomorphic to $H_3(M)\cong\mathbb{Z}$.
\end{enumerate}
\end{proposition}


\end{document}